\documentclass[12pt, reqno]{amsart}

\usepackage{graphicx}
\usepackage{amsthm}
\usepackage{amsmath, amssymb}
\usepackage{enumerate}
\usepackage{subfigure}
\usepackage{hyperref}
\hypersetup{
 colorlinks = true,
  citecolor = {green},
  }
 
\numberwithin{equation}{section}
\theoremstyle{plain}
\newtheorem{Thm}{Theorem}[section]
\newtheorem{MThm}[Thm]{Main Theorem}
\newtheorem{Coro}[Thm]{Corollary}
\newtheorem{Lem}[Thm]{Lemma}

\newtheorem{Que}[Thm]{Question}

\newtheorem{Rmk}[Thm]{Remark}

\theoremstyle{definition}
\newtheorem{Def}[Thm]{Definition}

\begin{document}

\begin{center}
 \title[P-moves between pants-block decompositions]{P-moves between pants-block decompositions of 3-manifolds}
\end{center}




\maketitle{}
\begin{center}
 \author{Pengcheng Xu}
\end{center}
\vspace{.5cm}

\textbf{Abstract:} A pants-block decomposition of a 3-manifold is similar to a triangulation of a 3-manifold in many aspects.
In this paper we show that any two pants-block decompositions of a 3-manifold are related by a finite 
sequence of moves which are called P-moves. The P-moves between pants-block decompositions are similar to the Pachner moves between 
triangulations. Moreover, we also give a list of types of P-moves. The main tools we used in this paper are the Morse 2-functions,
Reeb complexes and a new 2-dimensional complex called P-complex.

\section{Introduction}\label{secintro}

A pants-block decomposition of a compact, connected, closed, orientable 3-manifold $M$ is a decomposition of $M$ that cuts the manifold into a collection 
of fundamental pieces called pants blocks, which will be described in Section \ref{pb}. Pants blocks were first introduced 
in Minsky's paper ~\cite{Minsky}, to construct geometric models of ends of hyperbolic manifolds with prescribed 
ending laminations.

Agol ~\cite{Agol} and Li ~\cite{li} used pants-block decompositions to construct closed non-Haken 3-manifolds with certain 
properties, and Johnson  ~\cite{Johnson}  showed the existence of 
pants-block decompositions for all compact hyperbolic 3-manifolds (see Theorem \ref{thm2}, and the main theorem of 
~\cite{Johnson}). He also pointed out a connection to the layered triangulations, studied by Jaco and Rubinstein
~\cite{JR}. There are many analogies between triangulations and pants-block decompositions: they both contain 
1-skeletons (edges versus links), which are the boundaries of two dimensional pieces (triangles versus pairs of 
pants), and cut the 3-manifold into fundamental three dimensional pieces (tetrahedra versus pants blocks).
Based on the similarities between triangulations and pants-block decompositions, one can ask a natural question:
\begin{Que}\label{que1}
Is there a collection of moves between pants-block decompositions of the same manifold analogous to 3-dimensional Pachner 
moves between triangulations?
\end{Que}

As we will see below, a small collection of adjacent pants blocks can often be seen as defining a path in the 
pants complex for a small surface embedded in the manifold. Replacing this path by a different 
path with the same endpoints defines a new set of pants blocks that can be used to replace the original collection, 
defining a new pants-block decomposition of the manifold. By applying this construction to the 2-cells in the pants 
complex, we will define below a collection of moves between pants-block decompositions of a 3-manifold called P-moves,
and prove the following:

\begin{MThm}\label{MThm}
 Given a compact, closed, hyperbolic 3-manifold $M$, any two  pants-block decompositions are related to each other
  by a finite sequence of P-moves.
\end{MThm}

The outline of the paper is as follows: We first define pants-block decompositions in Section \ref{pre}, then 
present the HLS relations in Section \ref{sechls}, and use them to define
path moves in the pants complex in Section \ref{sec1}. We 
review the theory of Morse 2-functions
and Reeb complexes in Section \ref{secmorse}. In order to understand the decomposition better,
we introduce a new complex that we call a P-complex in Section \ref{secpcomplex}, and use this construction to prove Theorem
\ref{MThm} in Section \ref{secproof}.

\subsection*{Acknowledgement}

I want to thank Jesse Johnson, my thesis advisor, for giving me a lot of help and useful discussions in the direction
of proving this theorem, and also want to thank Henry Segerman, David Gay and David Futer for many helpful suggestions.

\section{Preliminaries}\label{pre}
In this paper, $S_{g,n}$ indicates a surface with genus $g$ and $n$ boundary components. We also say the surface is of \textit{type}
$(g,n)$. Note that we are working on curves in the surface $S$, so the distinction between punctures and boundary components
is not so important, and we do not distinguish them. We will only consider compact orientable surfaces unless otherwise specified.
\subsection{Pants decomposition}
Pants decompositions of surfaces is a widely explored field. The work most related to our paper is done by Hatcher, Lochak and Schneps ~\cite{HLS},
which we will explain in Section \ref{sechls}. 
\begin{Def}
 Given a compact, orientable surface $S$,
 a \textit{pants decomposition} for $S$ is a set $\mathcal{P}$ of pairwise disjoint, essential
simple closed curves in $S$ such that each component of $S\backslash \mathcal{P}$ is a pair of pants, i.e., type (0,3).
If $S$ admits a collection of pants decompositions
$\{\mathcal{P}_i\}$ then $S$ is called the \textit{base surface} with respect to $\{\mathcal{P}_i\}$.
\end{Def}

It is straightforward to check that  tori, annuli, disks and spheres are the only compact orientable surfaces that don't admit a pants decomposition,
so a surface admits a pants decomposition if and only if its Euler characteristic is less than or equal to $-1$. 
The only cases with $\chi=-1$ are $S_{0,3}$ and $S_{1,1}$. Another interesting category is $\chi=-2$, which contains
three cases $S_{0,4}$, $S_{1,2}$ and $S_{2,0}$. Surfaces $S_{1,1}$ and $S_{0,4}$ carry fundamental moves which will be explained in the 
following paragraph.

\begin{figure}[ht]
\centering
\includegraphics[width=.8\textwidth,height=.15 \textwidth]{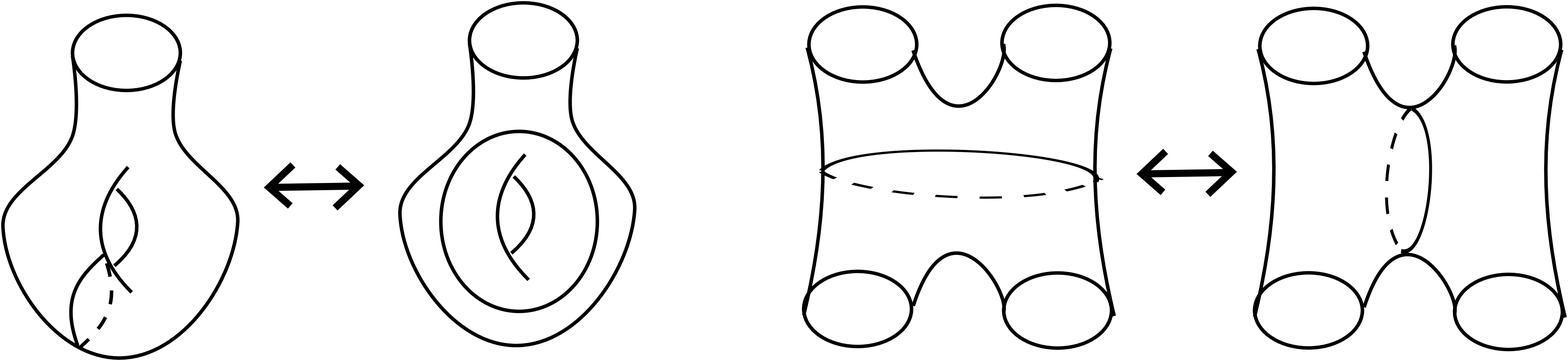}
\put(-265,5){$\alpha$}
\put(-195,35){$\alpha'$}
\put(-100,15){$\beta$}
\put(-24,25){$\beta'$}
\put(-240,-20){S-move}
\put(-80,-20){A-move}
\caption{Examples of pants moves}
\label{fig1}
\end{figure}

Each surface $S$ admits many different pants decompositions up to isotopy. 
We say two pants decompositions $\mathcal{P}$ and
$\mathcal{P}'$ of $S$ differ by a \textit{simple move},
or an \textit{S-move}, if we can find two loops $\alpha$ and $\alpha'$ from $\mathcal{P}$ and
$\mathcal{P}'$ respectively so that their geometric intersection number is 1, the complements $\mathcal{P}\backslash \{\alpha\}$
and $\mathcal{P'}\backslash \{\alpha'\}$ in $S$ both contain a type (1,1) component, and all other loops in $\mathcal{P}$ can
be isotoped to  loops in $\mathcal{P}'$. 
Similarly, we say two pants decompositions $\mathcal{P}$ and
$\mathcal{P}'$ of $S$ differ by an \textit{associativity move},
or an \textit{A-move}, if we can find two loops $\beta$ and $\beta'$ from $\mathcal{P}$ and
$\mathcal{P}'$ respectively so that the geometric intersection number is 2, the complements $\mathcal{P}\backslash \{\beta\}$
and $\mathcal{P'}\backslash \{\beta'\}$ both contain a type (0,4) component, and all other loops in $\mathcal{P}$ can
be isotoped to loops in $\mathcal{P}'$. These are the two fundamental moves between pants decompositions.
A \textit{pants move} between pants decompositions of a surface is either an S-move or an A-move. These are shown in Figure \ref{fig1}.
A pants move is the \textit{inverse} of another pants move if they are in different directions of one of the cases in Figure \ref{fig1}.

The pants moves connect different pants decompositions on a surface, which leads to the definition of a topological
graph based on this surface: 

\begin{Def}\label{def1}
 The \textit{pants graph} for a compact, orientable surface $S$ is the graph $G(S)$ whose vertices are isotopy classes of 
 pants decompositions for $S$ and with edges connecting vertices that differ by a pants move.  We say $S$ is 
 the \textit{base surface} of $G(S)$.
\end{Def}

It is not difficult to see that a pants graph for any surface is an infinite graph, for the reason that if a surface
admits a pants decomposition, it has infinitely many isotopy classes of pants decompositions. In a pants graph,  an edge 
corresponding to an $S$-move is called an $S$-edge, an edge corresponding to an $A$-move is called an $A$-edge.
\subsection{Pants blocks}\label{pb}
Given a surface $S$ and its pants graph $G(S)$, we want to consider a single edge and 
its two endpoints. This edge corresponds to a pants move, and the base surface for this pants move is either
a once-punctured torus $S_{1,1}$ or a four-punctured sphere $S_{0,4}$. Let $S'$ be a one of these subsurfaces.

 Consider the surface cross
interval $S'\times [0,1]$ such that $S'\times \{0\}$ and $S'\times \{1\}$ contain the loops from the pants decompositions
$\mathcal{P}$ and $\mathcal{P}'$ respectively. By collapsing the boundary annuli of $S$ to their core circles (see Figure
\ref{collapse}), we obtain 
a three-dimensional object with three or six loops in its boundary. Depending on the base surface, we have two different pants 
blocks: the (1,1)-block based on $S_{1,1}$ with three loops or the (0,4)-block based on $S_{0,4}$ with six loops,
 shown in Figure \ref{pantsblock}.

 \begin{figure}[ht]
\centering
\includegraphics[width=.6\textwidth,height=.15 \textwidth]{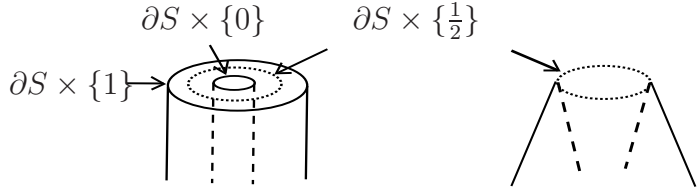}
\put(-210,60){$\partial S\times \{0\}$}
\put(-130,60){$\partial S\times \{\frac{1}{2}\}$}
\put(-260,35){$\partial S\times \{1\}$}
\caption{Collapse a boundary annulus onto its core circle.}
\label{collapse}
\end{figure}

 \begin{figure}[ht]
\centering
\includegraphics[width=.7\textwidth,height=.3 \textwidth]{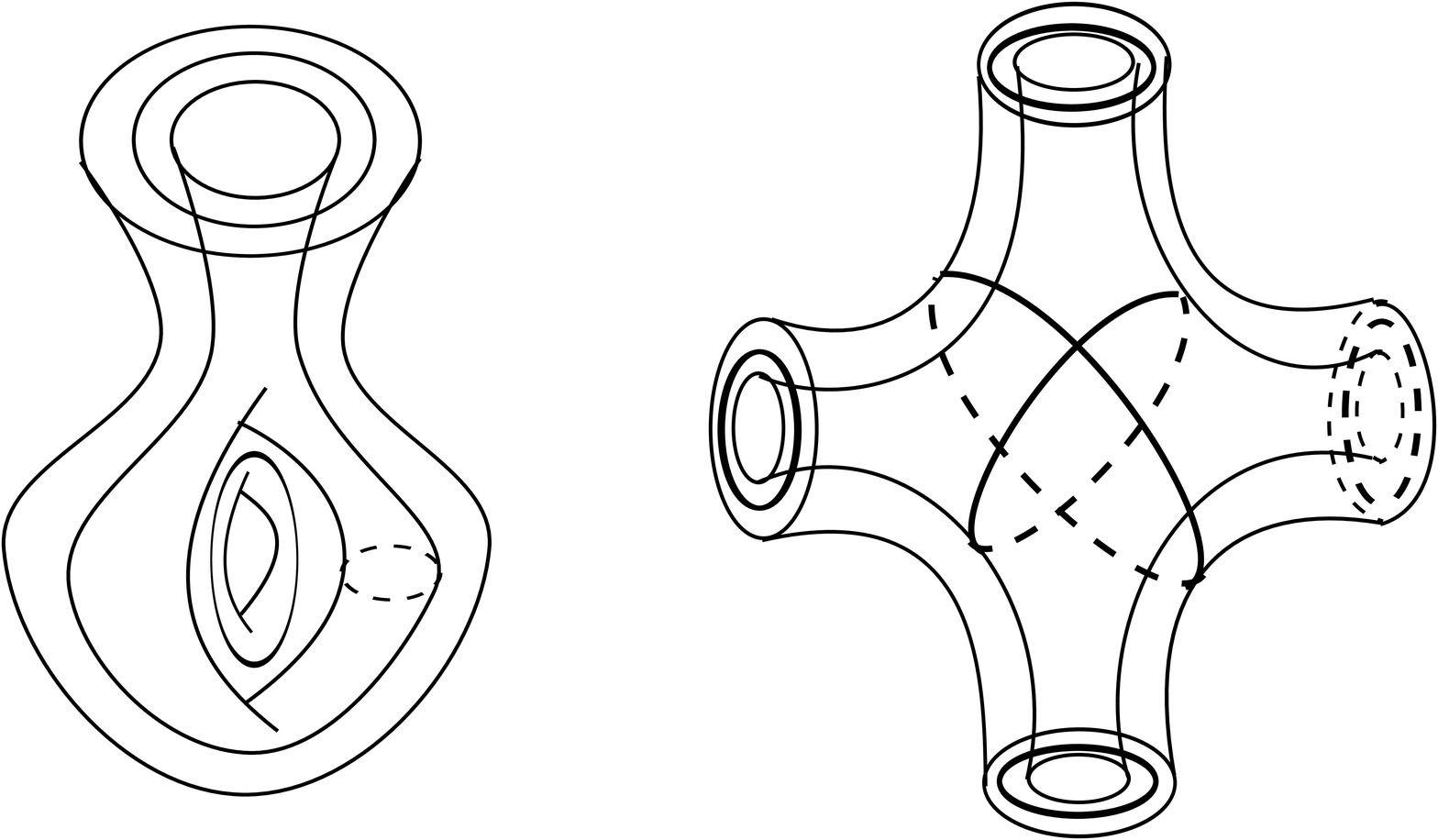}

\caption{Two fundamental pants blocks before collapsing boundary annuli.}
\label{pantsblock}
\end{figure}

Above is the rough idea of pants block. Yair Minsky ~\cite{Minsky} was the first one who introduced the idea of pants
blocks. Here is a precise definition:
\begin{Def}\label{defpantsblock}
 Let $S$ be a base surface of type (1,1) or type (0,4), i.e., $S=S_{1,1}$ or $S=S_{0,4}$. In $S\times [0,1]$, we say $S\times \{0\}$ 
 is the \textit{bottom surface} and $S\times \{1\}$ is the \textit{top surface}. For each component of $\partial S\times [0,1]$, we will collapse it to
 its deformation retract $\partial S\times \{\frac{1}{2}\}$ (see Figure \ref{collapse}).
 A \textit{pants block} is a handlebody
 $S\times [0,1]$ with a collection of essential loops in its boundary, forming a three dimensional object of one of two
 forms:

 (1) (1,1)-block: one loop is the deformation retract of the annulus $\partial S_{1,1}\times [0,1]$, and the other two loops are 
 contained
 in $S_{1,1}\times \{0\}$ and $S_{1,1}\times \{1\}$ so that the pants decompositions of top and bottom surfaces differ by 
 an S-move.

 (2) (0,4)-block: four loops are the deformation retract of components of $\partial S_{0,4}\times [0,1]$, and the other two loops 
 are contained in $S_{0,4}\times \{0\}$ and $S_{0,4}\times \{1\}$ so that the pants decompositions of top and bottom surfaces
  differ by an A-move.
\end{Def}

The above two pants blocks are called the \textit{fundamental blocks}.

\subsection{Pants-block decomposition}

The idea of a pants-block decomposition is to use two-dimensional pieces (pairs of pants)
with their boundaries (a link) to cut the manifold into three dimensional simple pieces (pants blocks).

\begin{Def}
 A \textit{pants-block decomposition}\index{pants-block decomposition} of a 3-manifold $M$ is a triple $(L, \mathcal{P}_L,\mathcal{B})$,
 for $L\subset M$ an embedded link, $\mathcal{P}_L$ a set of immersed pairs of pants whose interiors are pairwise 
 disjoint and embedded, with boundaries contained in $L$, and $\mathcal{B}$  a collection of embedded pants 
 blocks which are bounded by pairs of pants in $\mathcal{P}_L$ and whose union is all of $M$.
 In particular, gluing the blocks in $\mathcal{B}$ along pairs of pants in 
 $\mathcal{P}_L$ with boundaries of pants sent to $L$ gives us $M$.
\end{Def}

We denote $PB=(L,\mathcal{P}_L,\mathcal{B})$.
Note that there may be many choices of $L,\mathcal{P}_L$ and $\mathcal{B}$ for any given 3-manifold. Jesse Johnson ~\cite{Johnson} 
showed the following existence theorem for pants-block decompositions. He
called them model decompositions. In this paper we use the terminology pants-block decomposition for the same reason as 
we use pants decomposition.
\begin{Thm}\label{thm2}
Every compact, closed, hyperbolic 3-manifold admits a pants-block decomposition.
\end{Thm}

\section{HLS relations}\label{sechls}
In ~\cite{HLS}, Hatcher, Lochak and  Schneps extended Hatcher-Thurston's pants graph to a two-dimensional complex called
the \textit{pants complex} and showed that it is simply-connected. We will use this complex to  define
path moves and P-moves, so we devote this section to explaining the required details.

\begin{Def}\label{hls2}
 The \textit{pants complex} \index{pants complex}  for a compact, orientable surface $S$  is a two dimensional cell
 complex whose 1-skeleton is the pants graph with two-dimensional faces of the following five patterns:
\end{Def}

 \begin{figure}[ht]
\centering
\includegraphics[width=.7\textwidth,height=.3 \textwidth]{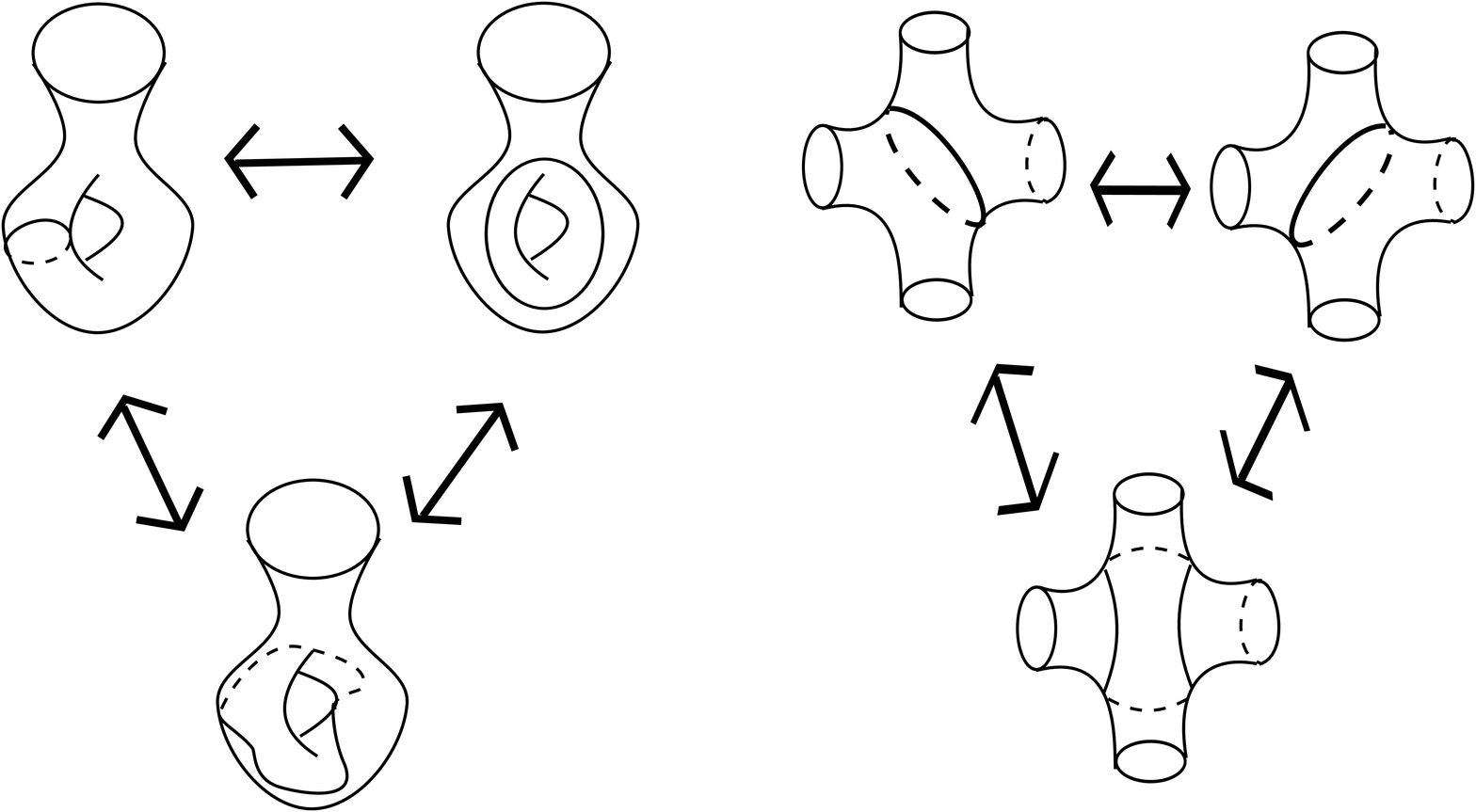}
\caption{S-triangle and A-triangle.}
\label{fig5}
\end{figure}

(3S) If three pants decompositions pairwise differ by S-moves as in Figure \ref{fig5}, on the same type (1,1) subsurface, 
we fill in the triangle in the pants complex with a 2-cell called an \textit{S-triangle}\index{S-triangle}. We call this the (3S)-relation,
following the terminology in the original paper ~\cite{HLS}.

(3A) If three pants decompositions pairwise differ by an A-move as in Figure \ref{fig5}, on the same type (0,4) subsurface, 
we call this 2-cell an \textit{A-triangle}\index{A-triangle}, and call this the (3A)-relation.

 \begin{figure}[ht]
\centering
\includegraphics[width=.5\textwidth,height=.3 \textwidth]{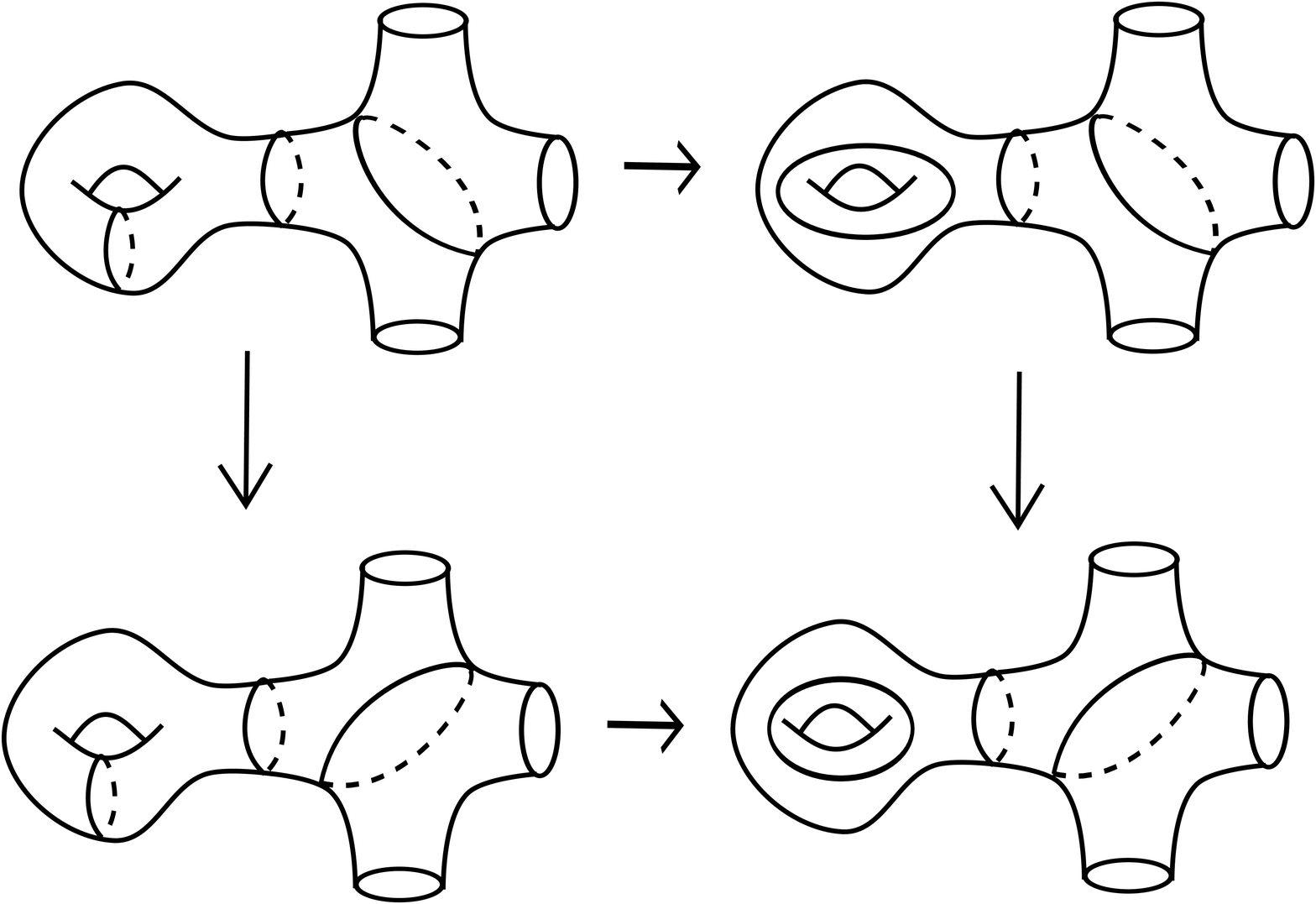}
\caption{Commutativity.}
\label{fig6}
\end{figure}

(C) If two moves are supported in disjoint subsurfaces of $S$, as in Figure \ref{fig6}, then they commute, and their
commutator forms a cycle of four moves. Here we only show one case, though there are two other cases involving S-moves
or A-moves only. We fill in this cycle with a quadrilateral and call this relation \textit{commutativity}\index{commutativity} or a (C)-relation.

 \begin{figure}[ht]
\centering
\includegraphics[width=.5\textwidth,height=.5 \textwidth]{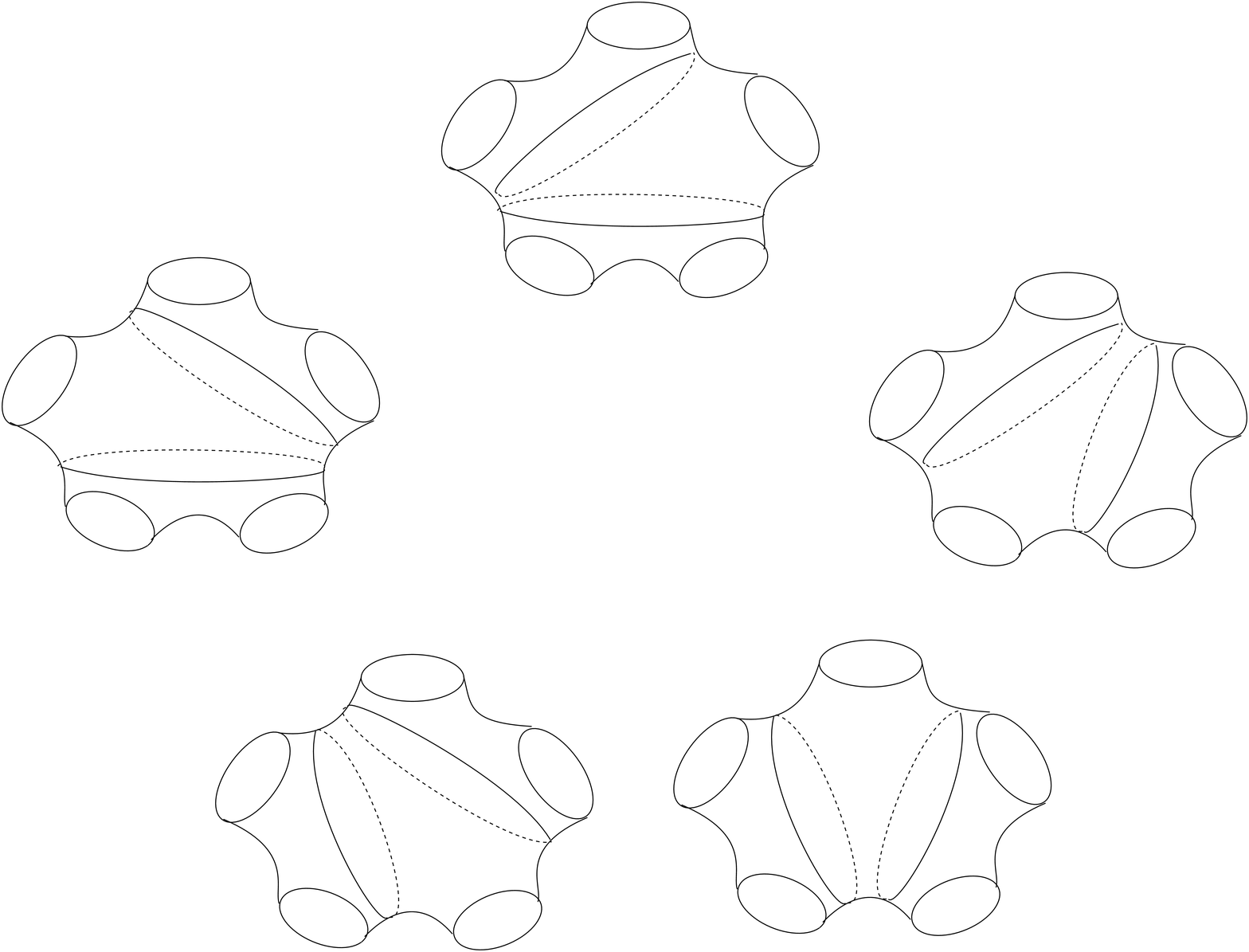}
\caption{A-pentagon.}
\label{fig7}
\end{figure}

(5A) Suppose deleting two loops from a pants decomposition creates a complementary component of type (0,5). On such
a type (0,5) subsurface, we need two essential simple closed curves to cut it into pairs of pants, and five
different ways to do that are as shown in Figure \ref{fig7}. Each one is related to the next one by an A-move, hence they form
 a cycle of five A-moves. We fill in this cycle with a 2-cell called an \textit{A-pentagon}\index{A-pentagon}
 and call this the (5A)-relation.

 \begin{figure}[ht]
\centering
\includegraphics[width=.8\textwidth,height=.3 \textwidth]{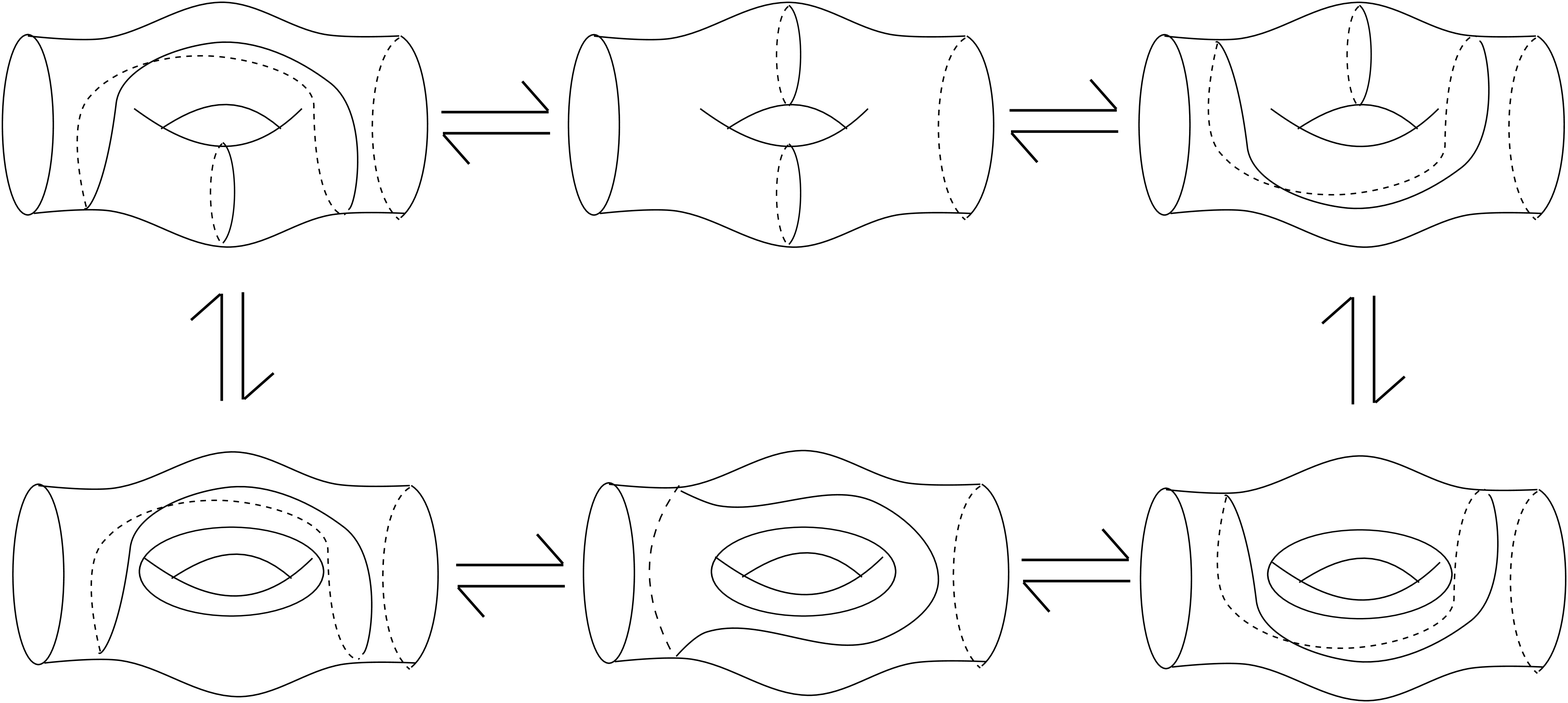}
\put(-200,100){$a_1$}
\put(-200,75){$a_1^{-1}$}
\put(-100,100){$a_2$}
\put(-100,75){$a_2^{-1}$}
\put(-200,5){$a_4$}
\put(-200,25){$a_4^{-1}$}
\put(-100,5){$a_3$}
\put(-100,25){$a_3^{-1}$}
\put(-30,55){$s_1$}
\put(-60,55){$s_1^{-1}$}
\put(-270,55){$s_2$}
\put(-240,55){$s_2^{-1}$}
\caption{Mixed hexagon. The labels are for a discussion on path moves in Section \ref{sec1}.}
\label{fig8}
\end{figure}

(6AS) Suppose deleting two loops from a pants decomposition creates a complementary component of type (1,2).
Six pants decompositions that form a relation for such a subsurface are shown in Figure \ref{fig8}.
This cycle contains four A-moves and two S-moves, where the two S-moves must be on the opposite edges of the hexagon.
 We fill in this cycle with a 2-cell called a \textit{mixed hexagon}\index{mixed hexagon} and call  this the (6AS)-relation.

We call these relations \textit{HLS relations}\index{HLS relations}, after Hatcher, Lochak and Schneps. 
They showed that the pants complex for $S$, with the 2-cells defined by
these relations, is simply-connected~\cite{HLS}. We will use these relations  in the next section
to introduce \textit{path moves} in the pants complex.

\section{Path moves and P-moves}\label{sec1}

In this section, we want to give more details on paths in the pants complex, which leads to the discussion of path moves and P-moves.
Throughout this paper, every edge path in the pants complex is finite. 

The theorem that the pants complex of a surface is simply-connected implies any two edge paths
joining two given vertices are related by sliding across a finite number of the polygons in Figure
\ref{fig9} of HLS relations, together with the trivial operation of inserting or deleting a  move followed by its inverse
(called a
\textit{cancelling pair}). We call
these operations \textit{path moves} in the pants complex.

 \begin{figure}[ht]
\centering
\includegraphics[width=.7\textwidth,height=.1 \textwidth]{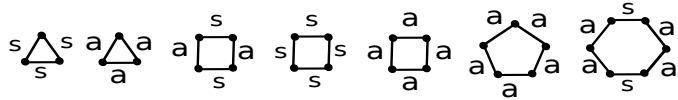}

\caption{HLS relations in polygon form}
\label{fig9}
\end{figure}

Consider the polygons defined by the HLS relations, which are shown  in Figure  \ref{fig9}.
In order to explain the idea of path moves clearly, we label each edge with a letter $s$ if it corresponds to an S-move,
and a letter $a$ if it corresponds to an A-move. An edge path is labelled by a sequence of letters. The notation 
$\leftrightarrow$ here means replacing one (sub)edgepath with the other if they have the same endpoints.
Furthermore, we use the notation $s^{-1}$ to represent the same S-move as $s$ but with the starting point and the  
ending point in the pants complex reversed. We also use $s_i$ and 
$a_j$ to represent different S-moves and A-moves if there are more than one S-move or A-move in an edgepath. 
Since the polygons in HLS relations are the fundamental 2-cells in pants complex, it suffices to discuss
the path moves based on these polygons. The following list with five categories is complete (up to rotations and reversed directions, 
$\emptyset$ is the empty set):

\begin{itemize}
  \item Cancelling pair: $ss^{-1}\leftrightarrow \emptyset$,\ \ \ \ $aa^{-1}\leftrightarrow \emptyset$
  
  \item S-triangle  \ \ \ \ \ \ \ \ \ \ \ \ \ \ \ \ \ \ \ \ \ \ \ \ \ \   A-triangle

  1-2 type: $s_1^{-1}\leftrightarrow s_2s_3$\ \ \ \ \ \ \ \ \ \ \ \ \ \ \ \ \ \ \ $a_1^{-1}\leftrightarrow a_2a_3$

  \item Commutativity
  \begin{itemize}
   \item 1-3 type: $s_1\leftrightarrow s_2s_1s_2^{-1}$, $a_1\leftrightarrow a_2a_1a_2^{-1}$, 
   
   \ \ \ \ \ \ \ \ \ \ \ \ \ $a\leftrightarrow sas^{-1}$, $s\leftrightarrow asa^{-1}$.

   \item 2-2 type: $a_1a_2\leftrightarrow a_2a_1$, $s_1s_2\leftrightarrow s_2s_1$, $sa \leftrightarrow as$.
  \end{itemize}
\item A-pentagon
\begin{itemize}
 \item 1-4 type: $a_1^{-1}\leftrightarrow a_2a_3a_4a_5$

 \item 2-3 type: $a_2^{-1}a_1^{-1} \leftrightarrow a_3a_4a_5$
\end{itemize}
\item Mixed hexagon
\begin{itemize}
 \item 1-5 type: $a_1^{-1}\leftrightarrow a_2s_1a_3a_4s_2$, $a_1 \leftrightarrow s_2^{-1}a_4^{-1}a_3^{-1}s_1^{-1}a_2^{-1}$, 
 
 \ \ \ \ \ \ \ \ \ \ \ \ \ $s_2^{-1}\leftrightarrow a_1a_2s_1a_3a_4$

 \item 2-4 type: $a_1a_2\leftrightarrow s_2^{-1}a_4^{-1}a_3^{-1}s_1^{-1}$, $a_1^{-1}s_2^{-1}\leftrightarrow a_2s_1a_3a_4$, 
 
 \ \ \ \ \ \ \ \ \ \ \ \ \ $s_2^{-1}a_4^{-1}\leftrightarrow a_1a_2s_1a_3$

 \item 3-3 type: $a_1a_2s_1\leftrightarrow s_2^{-1}a_4^{-1}a_3^{-1}$, $a_1^{-1}s_2^{-1}a_4^{-1}\leftrightarrow a_2s_1a_3$.
\end{itemize}
 \end{itemize}

With the above discussion, we want to see how to use edge paths in a pants complex to descibe pants-block decompositions.
To illustrate this connection, we present Lemma \ref{lem1}. 
\begin{Lem}\label{lem1}
 Given a pants complex of a surface, every edge path in which all of its pants decompositions have no loops in common
  defines a pants-block decomposition for a surface-cross-interval.
\end{Lem}
\begin{proof}
 Let $S$ be a base surface. Given a path $\{v_0,...,v_n\}$ in the pants complex for $S$,
consider $S\times [0,n]$. For each vertex $v_i$, choose a representative pants decomposition $P_i$ on $S\times \{i\}$ so that 
isotopic loops in $P_i$ and $P_{i+1}$ are setwise the same exact loop. Of course, for each pair of $P_i$ and $P_{i+1}$ there are 
two loops in $P_i$ and $P_{i+1}$ differ by a pants move by the definition of pants complex.
Let $\ell$ be a loop that is in both $P_i$ and $P_{i+1}$ and consider the annulus $A_{\ell,i}=\ell\times [i,i+1]$.
Let $X$ be the union of all these annuli (for all such $\ell$ and all $i$) plus all loops in $P_i$ for all $i$.

In the construction of $X$, two annuli are either pairwise disjoint or meet up with each other along their boundary 
loops forming longer annulus. Some annuli have boundary loops in $S\times \{0\}$ or $S\times \{n\}$, while some have 
boundary loops in the interior surface $S\times \{i\}$ for $1\leq i\leq n-1$. The loops in $X$ may be in the interior
or boundary of these annuli, or may be isolated. An annulus in $X$ 
is called a \textit{maximum annulus} if its boundaries are in both $S\times \{0\}$ and $S\times \{n\}$.
From the assumption we know that there is no maximum annulus in $X$.

Consider the union in $S\times [0,n]$ of the annuli in $X$ with $S\times \{i\}$ for all $i$. This union cuts $S\times [0,n]$ into 
blocks of the following forms: 1)(pair-of-pants)$\times [i,i+1]$ and 2) pants blocks before collapsing boundary annuli.
We will call the pants-cross-interval pieces \textit{trivial blocks}. Let $Y$ be the result of taking $S\times [0,n]$
and collapsing each trivial block vertically down to a pair of pants and each annulus in $X$ vertically down to a loop.

If two loops are in the same annulus of $X$, they are collapsed to the same loop in $Y$, so the image in $Y$ of $X$ 
defines a link $L$. If two pairs of pants are parts of the boundary of a 
trivial block in $S\times [0,n]$, they are collapsed to the same pair of pants in $Y$. Take the union $\mathcal{P}$
of pairs of pants in $Y$. The complement of $\mathcal{P}$ is the set of non-trivial blocks, which as noted above are pants blocks.
Since we assume  that there is no maximum annulus in $X$,
$Y$ is homeomorphic to $S\times [0,n]$. Thus we have constructed a pants-block decomposition for $S\times [0,n]$ defined by
the given edge path.
\end{proof}

If there is a maximum annulus $\ell$ in $X$ for a surface-cross-interval $S\times [0,n]$, then we can split $S$ into $S'$ and $S''$ along
$\ell$ and cut this surface-cross-interval into $S'\times[0,n]$ and $S''\times [0,n]$. If these two subsurface-cross-intervals both satisfy
the assumption of Lemma \ref{lem1}, then there exists two pants-block decompositions $(L_1,\mathcal{P}_{L_1}, \mathcal{B}_1)$ and $(L_2, 
\mathcal{P}_{L_2},\mathcal{B}_2)$ for the two subsurface-cross-intervals. Let $L=L_1\cup L_2\cup \{\ell\}$, $\mathcal{P}=\mathcal{P}_{L_1}
\cup \mathcal{P}_{L_2}$,
$\mathcal{B}=\mathcal{B}_1\cup \mathcal{B}_2$, then $(L,\mathcal{P}_L,\mathcal{B})$ is a pants-block decomposition for $S\times [0,n]$. If
there exists maximum annuli in the subsurface-cross-intervals, repeat the process above until all maximum annuli are gone. Since there 
are finitely many loops in a  pants decomposition of $S$, this process will stop eventually and we will obtain a pants-block decomposition
for $S\times [0,n]$.

Given a pants-block decomposition for a surface-cross-interval, we now explain how to generate a pants-block decomposition for a 3-manifold
with respect to it. 

\begin{Coro}\label{coro1}
 Every compact, connected, closed, orientable 3-manifold $M$ admits a pants-block decomposition.
\end{Coro}
\begin{proof}
By a theorem of Alexander ~\cite{Alexander}, $M$ admits an open book decomposition. The complement of the binding of the open book
decomposition is a surface bundle. Cut this complement along one page of this open book decomposition to obtain a surface-cross-interval.
Let $S$ be the base surface of this surface-cross-interval. Choose pants decompositions for the top surface and the bottom surface
respectively. Then by Lemma \ref{lem1}, this surface-cross-interval admits a pants-block decomposition. Let $h: S\to S$ be a monodromy
that maps the curves in the pants decomposition of the top surface to the curves in the pants decomposition of the bottom 
surface setwise. Glue the top and bottom surfaces under the monodromy $h$. This defines a pants-block decomposition for $M$.
\end{proof}

A 3-manifold may admit many different pants-block decompositions up to isotopy. We next define the relations between
different pants-block decompositions of a 3-manifold.

Given a surface-cross-interval $S\times [0,1]$ which contains only three different pants decompositions. Say $S\times \{0\}$ admits
$\mathcal{P}_0$, $S\times \{\frac{1}{2}\}$ admits $\mathcal{P}_1$ and $S\times \{1\}$ admits $\mathcal{P}_2$. If $\mathcal{P}_0$ and
$\mathcal{P}_1$ differ by an S-move (resp. A-move), $\mathcal{P}_1$ and  $\mathcal{P}_2$ differ by an S-move (resp. A-move), and
$\mathcal{P}_2$ is isotopic to $\mathcal{P}_0$, then these two S-moves (resp. A-moves) must be inverse of each other, and we call the two 
induced pants blocks an \textit{invertible pair}.

\begin{Def}\label{defpmove}
 Let $M$ be a 3-manifold, $PB_0=(L_0, \mathcal{P}_{L_0}, \mathcal{B}_0)$ and $PB_1=(L_1, \mathcal{P}_{L_1}, \mathcal{B}_1)$ are two 
 pants-block decompositions of $M$. We say  $PB_0$ and $PB_1$ differ by a P-move if one  of the following conditions holds:
 \begin{enumerate}
  \item $L_0\subset L_1$, $\mathcal{P}_{L_0}\subset \mathcal{P}_{L_1}$, $\mathcal{B}_0\subset \mathcal{B}_1$, and $PB_1$ contains only two
  more adjacent pants blocks such that they are an invertible pair. If this is the case, inserting or deleting an invertible pair is 
  a P-move turning one pants-block decomposition into the other.
  
  \item There  exists two collections of adjacent  pants blocks $A_0$ in $PB_0$ and $A_1$ in $PB_1$ respectively, with the same base surface. 
  The base surface, a subset of $\mathcal{P}_{L_i}$, could be $S_{1,1}, S_{0,4}, S_{0,5}$ or $S_{1,2}$.
  Outside  of $A_i$ in $PB_i$ are isotopic to each other for $i=0,1$. $A_0$ and $A_1$ can be viewed as surface-cross-intervals. Assume that
  their top surfaces with pants decompositions are isotopic to each other, and so are the bottom surfaces.
  Let $K$ be the collection of surfaces with different pants decompositions in $A_0$ and $A_1$. All of the pants decompositions in $K$ 
  together with the pants moves between them form a loop described as one of the HLS relations (except the commutativity). If this is the 
  case, replacing $A_0$ by $A_1$ (or the other way around) is a P-move between $PB_0$ and $PB_1$.
 \end{enumerate}
\end{Def}

Choose two edge paths in the pants complex such that they differ by a path move in the list. Lemma \ref{lem1} says each edge
path defines a pants-block decomposition. A path move slides one (part of) edge path across an HLS polygon or adds/deletes a cancelling
pair to obtain another edge path. This implies a move between the two corresponding pants-block decompositions. Thus by Definition \ref{defpmove},
each P-move is induced by a path move, though not every path move induces a P-move, so the list of P-moves is 
a sublist of path moves. In order to have a complete list of P-moves,  we need to do the following:
(1) find out which path move doesn't induce a P-move,
(2) discuss the conjugacy class of path moves, and keep only one representative for each conjugacy class in the list, 
(3) show that the rest of path moves in the list induce P-moves.

\begin{Lem}\label{lem3}
 All path moves generated by the commutativity relations do not contribute to any P-moves.
\end{Lem}

\begin{proof}
 A path move generated by a commutativity relation means there are two pants moves acting on a pants decomposition of 
 a surface one right after the other, and these two pants moves happen on different subsurfaces. 
 Take $as\leftrightarrow sa$ as an example. This implies that the order of the pants moves is important because they represent
 different edgepaths. However, $as$ and $sa$ both give the same collection of pants blocks in the pants-block decomposition because the two 
 pants moves happen on different subsurfaces so the two pants blocks do not overlap. Thus no P-move happens.
\end{proof}

Lemma \ref{lem3} rules out the possibilities that any 1-3 types or 2-2 types of path moves can be P-moves because they
are all generated by the commutativity relations. We next want to show that how to further narrow down the list of P-moves.

Two path moves $A$ and $B$ in the same category of the path moves list are \textit{conjugate} to each other if one can be obtained
by adding some cancelling pairs to another and then deleting same letters with same order from both sides of the move.
For example, let $A$ be the 2-3 type path move $a_2^{-1}a_1^{-1}\leftrightarrow a_3a_4a_5$, and $B$ be the 1-4 type path move
$a_1^{-1}\leftrightarrow a_2a_3a_4a_5$. Add $a_2a_2^{-1}$ to the left side of $B$, then delete $a_2$ from both 
sides since they are both at the beginning of the paths and don't affect the change. Thus we obtain the 2-3 type move $A$.
Similarly, $B$ can be obtained by adding a cancelling pair $a_2^{-1}a_2$ to the right side of $A$ and then deleting the same $a_2^{-1}$ 
from both sides. Note that we need to keep one representative
in the P-moves list, so we choose the 2-3 type path move. In general, we should keep one path move for each conjugacy class
from each category (not including
the commutativity category) of the path move list in the P-moves list. One category could have more than one conjugacy class.
One important thing we want to point out is that this action of conjugacy keeps the total number
of edges in a path move, therefore two path moves are not conjugate if they are in different categories.

The two path moves in the cancelling pair category cannot be obtained from each other by the above procedure,
so they are not conjugate. Therefore we have both of them in the P-moves list.
The two path moves of 1-2 types (in the triangle category) are both in the P-moves list for the same reason. As for the A-pentagon
category, by the above discussion, there is only one conjugacy class and we choose the 2-3 type as the representative in the P-moves list.
For the last category, Lemma \ref{lem4} shows that there is only one conjugacy class, so we choose a 3-3 type $a_1^{-1}s_2^{-1}a_4^{-1}\leftrightarrow
a_2s_1a_3$ to be the representative. 

\begin{Lem}\label{lem4}
 Any two path moves in the mixed hexagon category of the path moves list are conjugate. In particular, each path move can be obtained
 by adding no more than two cancelling pairs to another path move. 
\end{Lem}
\begin{proof}
 We first choose a 3-3 type $H_0:a_1^{-1}s_2^{-1}a_4^{-1}\leftrightarrow a_2s_1a_3$ and show that other moves can be 
 obtained by adding some cancelling pairs to this move. 
 
 (1) For $H_1:a_1a_2s_1\leftrightarrow s_2^{-1}a_4^{-1}a_3^{-1}$, first add $a_1^{-1}a_1$ to the head of the right side of $H_0$, delete 
 $a_1^{-1}$, then add $a_3^{-1}a_3$ to the tail of the left side of $H_0$, and delete $a_3$. Switch the left and right sides.
 
 (2) For $H_2:a_1^{-1}s_2^{-1}\leftrightarrow a_2s_1a_3a_4$, add $a_4a_4^{-1}$ to the tail of the right side of $H_0$, delete $a_4^{-1}$.
 
 (3) For $H_3:a_1^{-1}\leftrightarrow a_2s_1a_3a_4s_2$, it can be obtained by adding $s_2s_2^{-1}$ to the tail of the right side of $H_2$,
 and delete $s_2^{-1}$. Thus it is at most two steps from $H_0$.
     
 (4) For $H_4:a_1a_2\leftrightarrow s_2^{-1}a_4^{-1}a_3^{-1}s_1^{-1}$, it can be obtained by adding $s_1^{-1}s_1$ to the tail of the right side 
 of $H_1$, and delete $s_1$. Thus it is at most three steps from $H_0$.
     
 (5) For $H_5:s_2^{-1}a_4^{-1}\leftrightarrow a_1a_2s_1a_3$, it can be obtained by adding $a_1^{-1}a_1$ to the head of the right
 side of $H_0$, and delete $a_1^{-1}$.
 
 (6) For $H_6=a_1 \leftrightarrow s_2^{-1}a_4^{-1}a_3^{-1}s_1^{-1}a_2^{-1}$, it can be obtained by adding $a_2^{-1}a_2$ to
 the tail of the right side of $H_4$, and delete $a_2$. Thus it is at most four steps from $H_0$. 
 
 (7) For $H_7=s_2^{-1}\leftrightarrow a_1a_2s_1a_3a_4$, it can be obtained by adding $a_4a_4^{-1}$ to the tail of the right side of $H_5$, and 
 delete $a_4^{-1}$. Thus it is at most two steps from $H_0$.
 
 From the construction above we can see that each $H_i$ can be obtained by adding no more than two cancelling pairs
 to $H_j$ for some $j\neq i$.
\end{proof}

Therefore our list of P-moves is as follows:

\begin{itemize}
  \item Cancelling pair: $ss^{-1}\leftrightarrow \phi$,\ \ \ \ $aa^{-1}\leftrightarrow \phi$
  
  \item S-triangle  \ \ \ \ \ \ \ \ \ \ \ \ \ \ \ \ \ \ \ \ \ \ \ \ \ \   A-triangle

  1-2 type: $s_1^{-1}\leftrightarrow s_2s_3$\ \ \ \ \ \ \ \ \ \ \ \ \ \ \ \ \ \ \ $a_1^{-1}\leftrightarrow a_2a_3$

\item A-pentagon
\begin{itemize}

 \item 2-3 type: $a_2^{-1}a_1^{-1} \leftrightarrow a_3a_4a_5$
\end{itemize}
\item Mixed hexagon
\begin{itemize}

 \item 3-3 type:  $a_1^{-1}s_2^{-1}a_4^{-1}\leftrightarrow a_2s_1a_3$.
\end{itemize}

 \end{itemize}

\section{Morse 2-functions and the Reeb complexes}\label{secmorse}
The main argument of this paper will involve constructing pants block decompositions from  Morse 2-functions. 
This is a natural generalization of constructing pants decompositions from Morse functions on surfaces, which we will present in 
Section \ref{morsereeb}. This chapter is devoted to a brief introduction to the theory of Morse 2-functions
 and singularities, and their connections with Reeb complexes (Section \ref{secmorse2}) and the induced moves
 (Section \ref{singularities}). Gay and Kirby studied Morse 2-functions for $n$-manifolds in \cite{GK}.  Kobayashi and Saeki studied
 the theory for 3-manifolds in \cite{KS} with different terminology.
 
\subsection{Local behavior  of Morse functions}\label{locmorse}
 We first recall some background knowledge for Morse functions.
 Let $M$ be a smooth compact $n$-manifold. Here we only focus on $n=2,3$.
 Given a point $p$ in $M$ and a smooth function $f$, the \textit{gradient} of $f$ at $p$ is the vector defined by
 the partial derivatives of $f$ at $p$. We say $p$ is a \textit{critical} point of $f$ if the gradient of $f$ at $p$ is zero.
 The \textit{Hessian} of $f$ at $p$ is the matrix of second derivatives of $f$ at $p$. A critical point $p$ is 
 \textit{non-degenerate} if the determinant of the Hessian is non-zero. The image $f(p)$ is the \textit{critical value}
 of $p$ under $f$. If the Hessian of $f$ at $p$ is indefinite, then $p$ is a \textit{saddle point}.
 \begin{Def}
 A \textit{Morse function}\index{Morse function} is a smooth function $f:M\to \mathbb{R}$ that satisfies the following
 two conditions:
\begin{enumerate}
 \item Every critical point is non-degenerate,
 
 \item $f$ maps different critical points to different critical values.
\end{enumerate}

 If $f$ only satisfies condition (1) then we say $f$ is \textit{locally Morse} but not Morse.
\end{Def}

A Morse function is a special case of a more general notion called a \textit{stable map} (Morse
functions are stable maps on $\mathbb{R}$). We won't descibe stable maps in their full generality here. See ~\cite{BdR} for a definition. 

Given two Morse functions $f_0$ and $f_1$, consider a homotopy $\{f_t\}_{0\leq t\leq 1}$ from $f_0$ to $f_1$. The \textit{Cerf graphic} $G$ of this homotopy
is the subset
in $[0,1]\times [0,1]$, such that the intersection of $G$ and $t\times [0,1]$ are the critical values of $f_t$ for all $t\in [0,1]$.

The following two remarks recall some results of Cerf theory, which is useful for our description of singularities of Morse functions 
and Morse 2-functions.

\begin{Rmk}\label{rmkstra}
 In the space $\mathcal{F}$ of all smooth real-valued functions on a compact manifold $M$ (see Cerf ~\cite{Cerf}), we can stratify 
 $\mathcal{F}$ using the codimensions of functions $f\in \mathcal{F}$,  which  we will define below. Let $\mathcal{F}^i$ denote the set of functions of codimension $i$.
 For example, let $p$ be an isolated critical point of $f$. We 
 say $p$ is of \textit{codimension 0} if it has a neighbourhood such that the function has a canonical form:
 \begin{equation}\label{nondeg}
 -x_1^2-...-x_i^2+x_{i+1}^2+...x_n^2+x_{n+1}^2
\end{equation}
Such a point $p$ is non-degenerate, and we called it a \textit{fold point}. 
$p$ is of \textit{codimension 1} if it has a neighbourhood such that the function has a canonical form:
 \begin{equation}\label{eqcusp}
 -x_1^2-...-x_i^2+x_{i+1}^2+...x_n^2+x_{n+1}^3
\end{equation}
We call such a point a \textit{cusp point}.
$p$ is of \textit{codimension 2} if it has a neighbourhood such that the function has a canonical form:
 \begin{equation}\label{eqswallowtail}
 -x_1^2-...-x_i^2+x_{i+1}^2+...x_n^2\pm x_{n+1}^4
\end{equation}
We call such a point a \textit{swallowtail point}.

The \textit{codimension of a critical value} $b$ of $f$ is the number of critical points in $f^{-1}(b)$ minus 1. Let 
 $$v_1(f)=sum\ of\ codimensions\ of\ critical\ points,$$
 $$v_2(f)=sum\ of\ codimensions\ of\ critical\ values.$$
 The \textit{codimension} of $f$ is $v_1(f)+v_2(f)$. 
 We will further discuss the stratification of $\mathcal{F}$ in the following subsections.
\end{Rmk}

\begin{Rmk}

If $v_1(f)=0$, there are no cusp points or swallowtail points in the Cerf graphic, only non-degenerate critical points exist, which are called
fold points. A locus of fold points in the Cerf graphic is called a fold edge.

If $v_1(f)=1$, there is only one cusp point and no swallowtail points in the Cerf graphic. The rest are fold points. 

If $v_1(f)=2$, there are two possibilities in the Cerf graphic. 
\begin{enumerate}
 \item  There is a swallowtail point and the rest are fold points.
 
 \item  There are two cusp points and the rest are fold points.
\end{enumerate}
If $v_2(f)=1$, two critical points are mapped to the same critical value, this means two fold edges intersect in the Cerf graphic,
so we call the intersection point a \textit{crossing}.
\end{Rmk}

By the remarks above, the stratum $\mathcal{F}^0$  contains all functions of 
codimension 0, so we must have $v_1(f)=v_2(f)=0$. This implies $\mathcal{F}^0$ consists of functions with only 
fold points and distinct critical values, so they are all Morse functions.

\subsection{Reeb graphs}\label{morsereeb}

To connect pants decompositions on surfaces with Morse functions, we first need a tool called \textit{Reeb graphs}.
Named after Georges Reeb, there are  more general versions of the Reeb graphs, but here we only focus on the 
following definition:

\begin{Def}\label{reeb}
 Given a compact, orientable surface $S$, consider a Morse function $f:S\to \mathbb{R}$ such that $f$ is constant on each component
 of $\partial S$. A \textit{Reeb graph} 
 is the  quotient space $G=S/_{\sim_f}$ under the equivalence relation $x\sim_f y$ iff $x$ and $y$ are in the same component of a 
 level set of $f$.
\end{Def}
\begin{figure}[ht]
\centering
\includegraphics[width=.2\textwidth,height=.2 \textwidth]{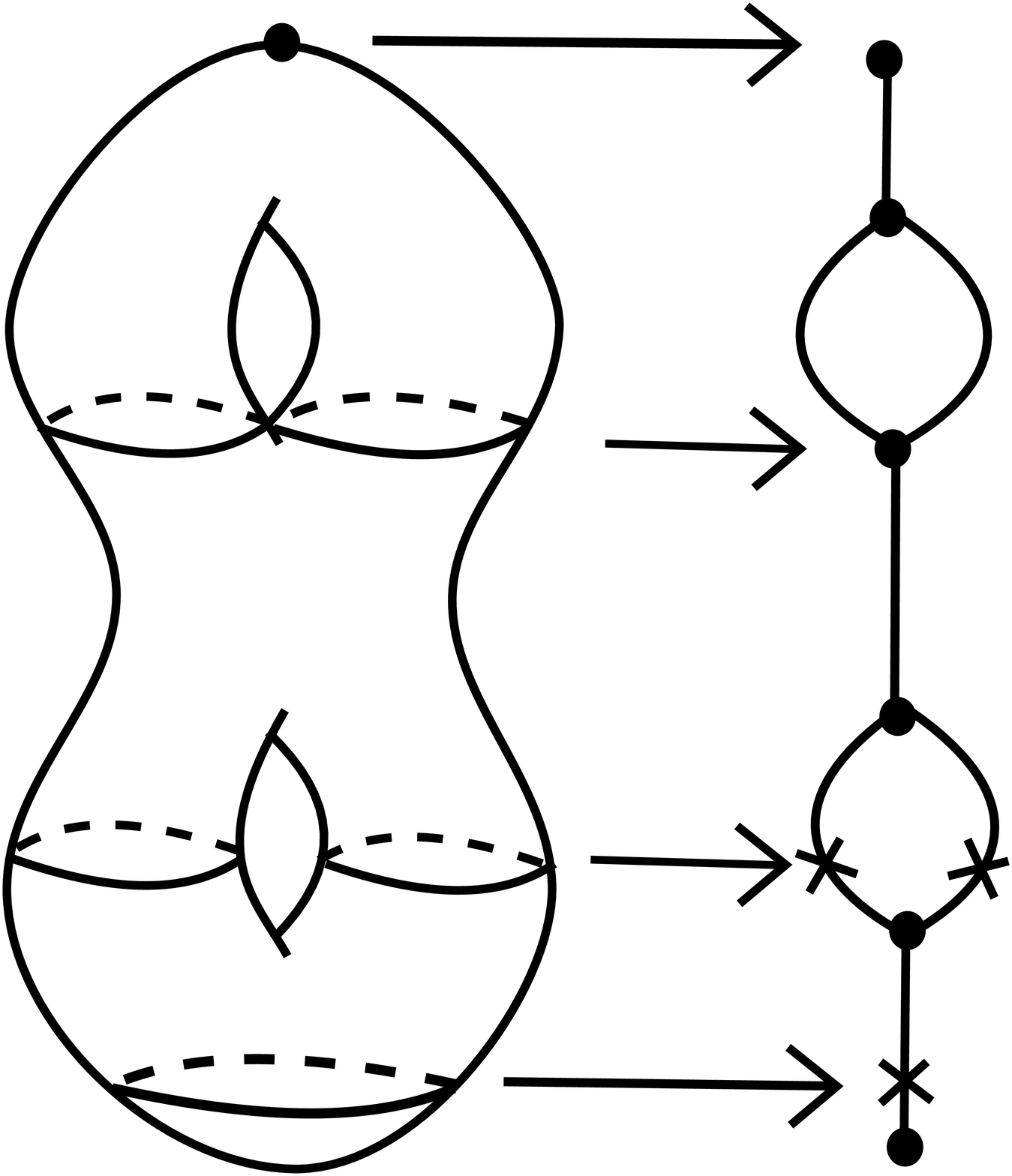}

\caption{A Reeb graph for a genus 2 closed surface under the height function. Level sets are mapped to corresponding points in the 
Reeb graph as shown in the figure.}
\label{reebg1}
\end{figure}

In a Reeb graph $G$, the vertices are of valence-one or valence-three. The preimage of a valence-three vertex is a figure-8 loop (wedge
sum of two circles),  corresponding to a saddle point, and the preimage of a valence-one vertex is either a component 
of $\partial S$, or a local maximum/minimum of $f$. Note that the original
 Morse function $f$ is the composition of the quotient map $q:S\to G$ and the induced map $G\to \mathbb{R}$.
In general, the method of defining a quotient space and writing 
the stable maps is  called a \textit{Stein factorization}. See ~\cite{BdR}, ~\cite{FT}, ~\cite{Levine} and ~\cite{MPS} for details.
In order to avoid any name conflicts, here we will call this quotient space $G$ a \textit{Reeb graph}, and the
procedure of writing a stable map as a composition of two maps a Stein factorization. 
 
 We now explain how to connect Morse functions with pants decompositions. By ~\cite{HLS}, we can associate 
 a given pants decomposition of surface $S$ to a Morse 
 function as follows:
For each Morse function $f:S\to \mathbb{R}$, by Definition \ref{reeb} we can associate a Reeb graph $G$.  Choose a midpoint for each edge in $G$ 
and let $A$ be the 
collection  of their preimages in $S$. Then $A$ is a set of curves in $S$.  Delete those curves in $A$ that bound disks in $S$
or are isotopic to boundary components of $S$. Replace mutually isotopic curves in $A$ by a single curve and denote the resulting collection by $A'$.
Then $A'$ is a pants decomposition of $S$ since elements of $A'$ are preimages of midpoints of edges in the modified  graph $G'$
(We will further discuss this graph $G'$ in Section \ref{sec4}). Thus every Morse function defines a pants decomposition. 

Conversely, one can define a Morse function from a pants decomposition: first define a function in the neighbourhoods of curves in the 
pants decompositions and neighbourhoods of boundary components so that all these curves are non-critical level curves,
then extend the function on $S$ and perturb it if it is only locally Morse to obtain a Morse function.

\subsection{Homotopies of Morse functions}\label{sechomo}
The first goal of this section is to construct a path between two vertices  representing
two pants decompositions in a pants complex. Note that we can associate a pants decomposition to a Morse 
function, thus this construction can be done by studying a homotopy between Morse functions. 
We first need the following remark.

\begin{Rmk}
In a Cerf graphic, a \textit{crossing} is the intersection of two fold edges, 
as indicated by $\mathcal{F}_{\beta}^1$ in Figure \ref{stratum1}, while a \textit{birth cusp} (resp. a \textit{death cusp})
is a shape with a cusp point at time $t_*$ and
two branches of critical values created after the moment $t_*$ (resp. cancelled after the moment $t_*$),
as indicated by $\mathcal{F}_{\alpha}^1$ in Figure \ref{stratum1}.
 Following the notations in Remark \ref{rmkstra}, we can decompose the stratum $\mathcal{F}^1$ into $\mathcal{F}_{\alpha}^1$ and 
 $\mathcal{F}_{\beta}^1$ as follows:
\begin{enumerate}
 \item $\mathcal{F}_{\alpha}^1$ is the set of functions for which $v_1(f)=1$ and $v_2(f)=0$. By definition, $v_2(f)=0$ means at every critical value $a$,
 $f^{-1}(a)$ has only one critical point. This rules out the possibility of a crossing. $v_1(f)=1$ means there is only 
 one birth cusp (or a death cusp) in the Cerf graphic. See Figure \ref{stratum1}. 
 
\item  $\mathcal{F}_{\beta}^1$ is the set of functions for which $v_1(f)=0$ and $v_2(f)=1$. $v_1(f)=0$ means all critical points are fold points. $v_2(f)=1$ means there
is only one crossing point, that is, a pair of critical points have the same critical value. See Figure \ref{stratum1}.
\end{enumerate}
\end{Rmk}

\begin{figure}[ht]
\centering
\includegraphics[width=.6\textwidth,height=.3 \textwidth]{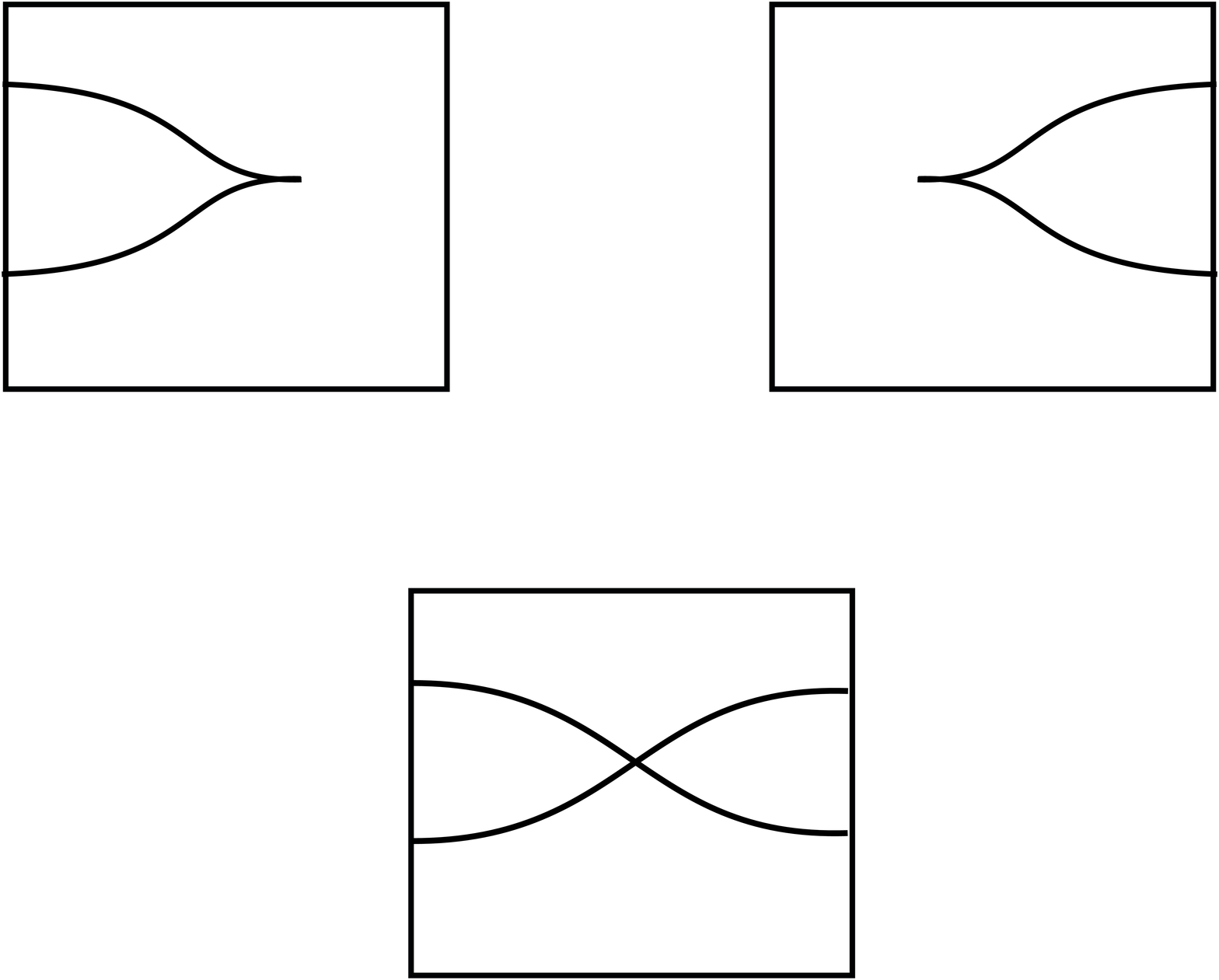}
\put(-235,90){$\mathcal{F}_{\alpha}^1$}
\put(-235,20){$\mathcal{F}_{\beta}^1$}
\put(-200,70){$i$}
\put(-200,97){$i+1$}
\put(-35,70){$i$}
\put(-45,97){$i+1$}
\put(-140,35){$i$}
\put(-140,5){$j$}
\put(-210,55){death cusp}
\put(-70,55){birth cusp}
\put(-125,-10){crossing}
\caption{Neighbourhoods of cusp points and a crossing in Cerf graphics.}
\label{stratum1}
\end{figure}

The following definition is based on the stratum of $\mathcal{F}^1$. 
\begin{Def}\label{def1parameter}
Given two Morse functions $g_0,g_1:M\to \mathbb{R}$, a homotopy $\{g_t\}_{0\leq t\leq 1}$ between $g_0$ and 
$g_1$ is called a \textit{generic homotopy between Morse functions} if it satisfies the following properties:
\begin{enumerate}
 \item The function $g_t$ is Morse for all but finitely many values of $t$ (codimension 0);
 
 \item For those values $t_i$ where $g_{t_i}$ is locally Morse but not Morse, the images of a neighbourhood around $t_i$ contain either a crossing
or a cusp in the Cerf graphic (codimension 1).
\end{enumerate}
\end{Def}

  A \textit{singularity} of a generic homotopy between Morse 
functions is either a crossing or a cusp. The number of singularities is finite by definition.
We use $t=t_s$ to represent the image of $g_{t_s}$ in the Cerf graphic.
More details can be found in~\cite{GK} and~\cite{HW}. 
An \textit{arc of Morse functions} is a homotopy $\{g_t\}$ where $g_t$ is Morse for all $t$.

As noted in Section \ref{morsereeb}, we can associate a Morse function to each pants decomposition of a surface.
Given two pants decompositions $\mathcal{P}_0$ and $\mathcal{P}_1$ of the same surface, associate two Morse functions $g_0$ and $g_1$. For each
generic homotopy of Morse functions $\{g_t\}_{0\leq t\leq 1}$, we want to construct an edge path in the pants complex
associated to $\{g_t\}$ which connects vertices whose representatives are these two pants decompositions. 
Below  we will prove Lemma \ref{rmkcross}, which states that there are finitely many pants moves between $\mathcal{P}_0$ and $\mathcal{P}_1$, which means there are finitely
many edges in the edge path between the two vertices representing $\mathcal{P}_0$ and $\mathcal{P}_1$. To 
finish this construction, we only need to prove Lemma \ref{rmkcross}.

There are two ways in which a pants move between two pants decompositions $\mathcal{P}_0$
and $\mathcal{P}_1$ may
correspond to a crossing in the Cerf graphic of the generic homotopy $\{g_t\}$.

\begin{figure}[ht]
\centering
\includegraphics[width=0.95\textwidth,height=.2 \textwidth]{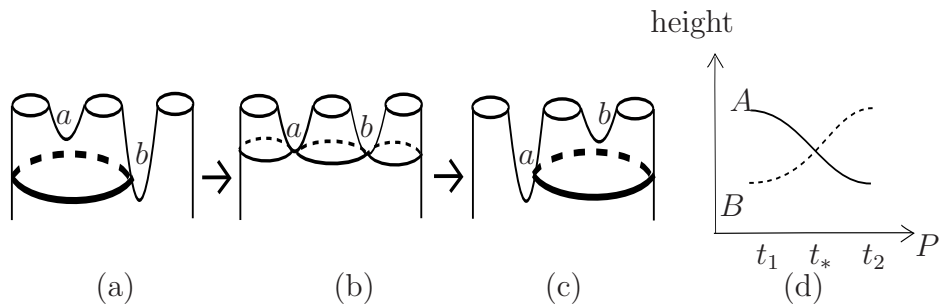}
\put(-325,45){$a$}
\put(-295,30){$b$}
\put(-70,50){$A$}
\put(-74,10){$B$}
\put(-210,40){$b$}
\put(-238,40){$a$}
\put(-120,45){$b$}
\put(-150,30){$a$}
\put(-100,80){height}
\put(0,-5){$P$}
\put(-20,-8){$t_2$}
\put(-60,-8){$t_1$}
\put(-40,-8){$t_*$}
\put(-310,-20){(a)}
\put(-220,-20){(b)}
\put(-140,-20){(c)}
\put(-50,-20){(d)}
\caption{The images of saddle points under the height function. The horizontal directions are named $P$-direction.}
\label{crossmove1}
\end{figure}

We first need some terms which may not be standard. Here we draw the Morse function as a height function.
Given a surface $S$ and $f:S\to \mathbb{R}$, let $p$ be a saddle point of $f$ in $S$, 
and $U_p$ be the component of $f^{-1}([f(p)-\epsilon, f(p)+\epsilon])$ that contains $p$. We choose $\epsilon$ to be a 
small enough positive number so that $U_p$ does not contain other critical points in $S$. Then $U_p$ is a pair of pants
with three boundary components. We say $p$ is an \textit{essential saddle point} if all three boundary components of $U_p$ 
are essential curves in $S$, i.e., they don't bound disks in $S$. A saddle point $p$ is said to be \textit{trivial}
if at least one boundary component of $U_p$ bounds a disk in $S$. An example of a trivial saddle point is the point $c$ in 
the left figure of Figure \ref{crossmove2}.

  Figure \ref{crossmove1} shows the 
A-move case. There are two different pants decompositions of a surface which are related by an A-move (subfigures (a) and (c)).
Subfigure (b) is the intermediate stage between them.
 We assume that $a$ and $b$ represent essential saddle points, $A$ and $B$ represent
the locus of $a$ and $b$ under the Morse function in a Cerf graphic. If we interchange the heights
of critical points
$a$ and $b$, this continuous process is reflected in the subfigure (d) as follows: critical points $a\in A$ and 
$b\in B$ are getting closer to each other (subfigure (a)), meet at the  crossing (subfigure (b)), then switch their heights and move away from
each other (subfigure (c)).  Therefore we have shown that an A-move corresponds to a crossing in a Cerf graphic. 

\begin{figure}[ht]
\centering
\includegraphics[width=0.7\textwidth,height=.3 \textwidth]{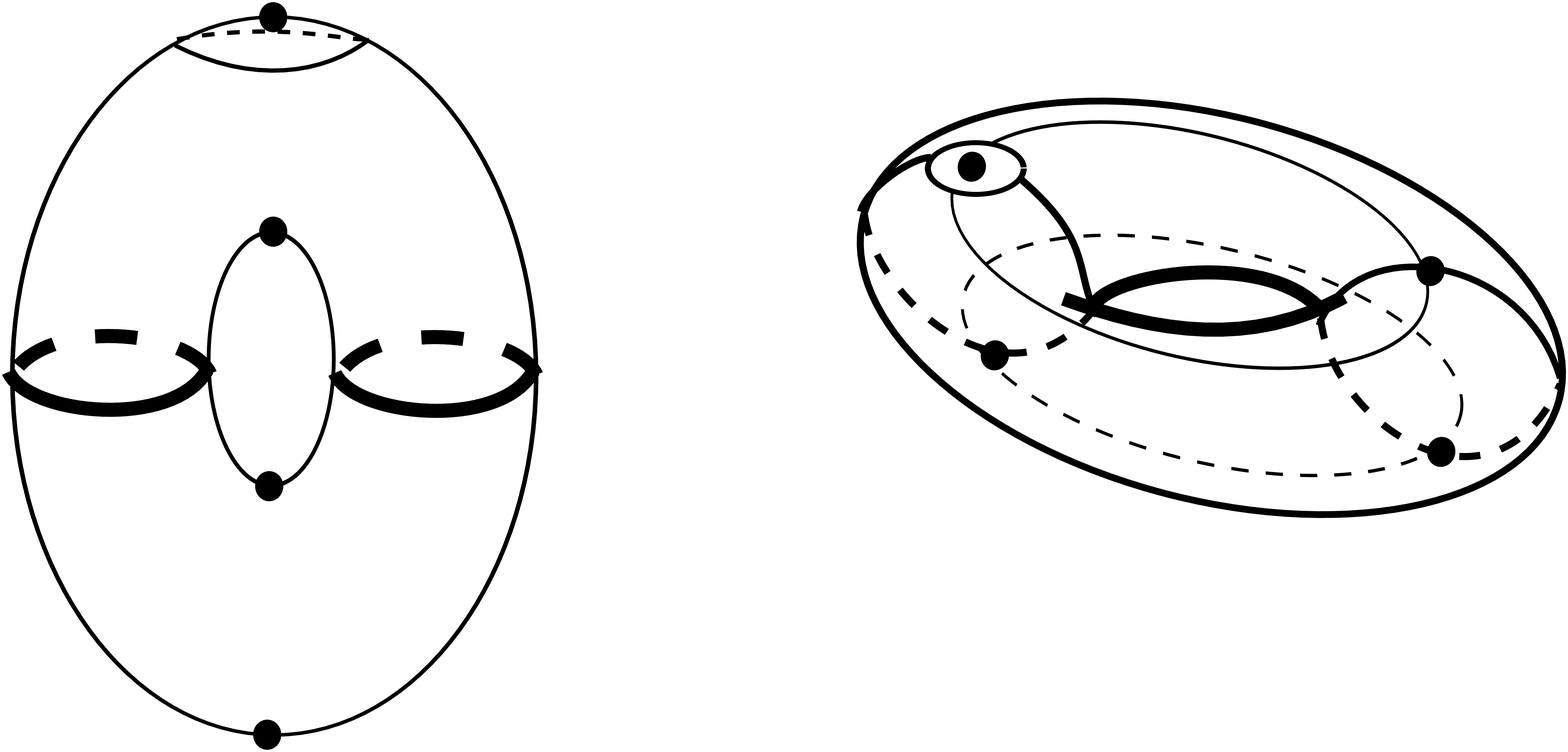}
\put(-210,110){$a$}
\put(-100,97){$a$}
\put(-210,80){$b$}
\put(-100,50){$b$}
\put(-210,30){$c$}
\put(-20,71){$c$}
\put(-210,5){$d$}
\put(-25,47){$d$}
\caption{Both figures are once-punctured tori. The neighbourhoods around $a$ in both figures are removed.}
\label{crossmove2}
\end{figure}
As for the S-move case, we look at Figure \ref{crossmove2}. As noted before, $c$ is a trivial saddle point while $b$ is
an essential saddle point in the left figure.
Imagine that this punctured torus is filled with some water
so that the thick horizontal circles in each figure are the boundaries of levels of water surfaces. One thick horizontal circle in the 
right figure is not shown, but it is mutually isotopic to the one we show.
Leaning the punctured torus changes the water surface, so one horizontal 
 circle on the left (meridian of this punctured torus) is related to the horizontal circle on the right (longitude
 of this punctured torus) by an S-move. At the beginning, we have heights $h(b)>h(c)$ in the left figure.
Lean the torus until the $h(b)<h(c)$, as the right figure shows. Here we assume $h(b)>h(d)$ and $h(a)>h(c)$
 so that $b$ and $c$ are still saddle points. This change corresponds to a crossing in a Cerf graphic for the same 
 reason as in the A-move case.  Therefore we have shown that an S-move corresponds to a crossing in a Cerf 
 graphic.

\begin{Lem}\label{rmkcross}
 Consider a generic homotopy between Morse functions $\{g_t\}$. Choose $t_0$ and $t_1$ so that $g_{t_0}$ and $g_{t_1}$
 are Morse functions, and there is a single crossing between
 $t=t_0$ and $t=t_1$ in the Cerf graphic for $\{g_t\}$. Let $\mathcal{P}_0$ and $\mathcal{P}_1$ be two pants decompositions associated to $g_{t_0}$
 and $g_{t_1}$ respectively. If $\mathcal{P}_0$ is not isotopic to $\mathcal{P}_1$, then $\mathcal{P}_0$ and 
 $\mathcal{P}_1$ differ by one or three pants moves.
\end{Lem}
\begin{proof}
We know pants moves define crossings from the discussions above. 
However, not every crossing defines a pants move. For example,
if we  cap off all top boundary components of surfaces in (a), (b) and (c) of Figure \ref{crossmove1}, then changing the heights of 
 $a$ and $b$ still gives a crossing in the Cerf graphic, but there is no pants move between (a) and (c). Therefore if a crossing between $t=t_1$ and $t=t_2$
 in a Cerf graphic doesn't define a pants move, then the two pants decompositions associated to $g_{t_1}$ and $g_{t_2}$ are 
 isotopic to each other. 
 Note that a crossing in a Cerf graphic $G$ is the intersection of two fold edges, thus the intersection of $G$ and 
 $t=t_s\in (t_1, t_2)$ contains at most two points, which involves at most two pairs of pants in the surface. Thus the 
 Euler characteristic of the surface is greater than or equal to $-2$. This means we only need to consider the cases that
 $S=S_{0,3}, S_{1,1}, S_{0,4}, S_{1,2}, S_{2,0}$. (1) If $S=S_{0,3}$, imagine the surface as in Figure \ref{crossmove1} with
 a bottom disk, thus a crossing defines no pants move. You can also imagine surface as in Figure \ref{crossmove1} with
 a bottom boundary circle and cap off one of the top boundary components, thus a crossing defines no pants move.
 (2) If $S=S_{1,1}$, this is exactly the S-move case above, so 
 a crossing defines an S-move. (3) If $S=S_{0,4}$, imagine the surface as in Figure \ref{crossmove1} with a bottom boundary
  circle, thus a crossing defines an A-move. (4) If $S=S_{1,2}$, imagine the surface as in Figure \ref{crossmove2} with 
  a neighbourhood of $d$ removed, Then there are three pants moves between these two pants decompositions based on 
  Figure \ref{fig8}: Start from the middle surface of the top row, it takes an A-move to get the right surface of the top row, then
  an S-move to get the right surface of the bottom row, and one more A-move before reaching the middle surface of the bottom row.
  (5) If $S=S_{2,0}$, then the crossing defines either an S-move as in (2) or an A-move as in (3). 
\end{proof}

Before introducing the generic homotopy of homotopies, we need the following remark.

\begin{Rmk}
 The stratum $\mathcal{F}^2$ contains all functions of codimension 2, and there are six types of functions in $\mathcal{F}^2$
in total. Here we list three types of them, which we will need in the rest of this section.  For a complete list, we 
recommand ~\cite{HW}.

Type 1: $\mathcal{F}_{\alpha}^2$ is the set of function for which $v_1(f)=2$ and $v_2(f)=0$.  $v_2(f)=0$ implies no crossing points, more precisely,
there are no two critical points mapped to the same critical value. $v_1(f)=2$ says there is a \textit{swallowtail point} (codimension 2) or there are two birth-death cusp points (each one is codimension 1). 
 There are two cases of swallowtail points, see Figure \ref{stratum2}.

\begin{figure}[ht]
\centering
\includegraphics[width=.8\textwidth,height=.3 \textwidth]{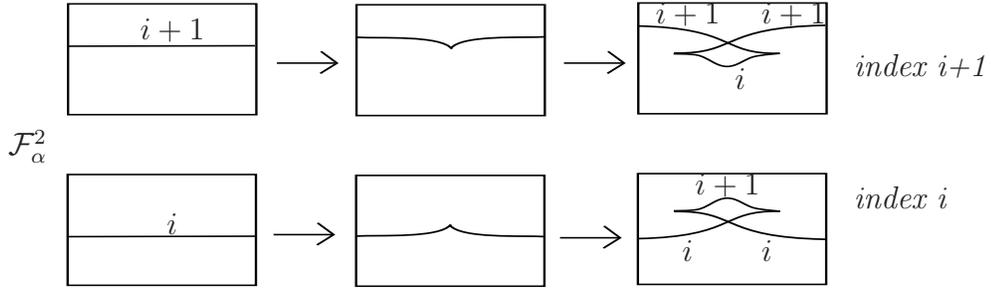}
\put(-310,50){$\mathcal{F}_{\alpha}^2$}
\put(-250,20){$i$}
\put(-260,94){$i+1$}
\put(-35,75){$i$}
\put(-25,100){$i+1$}
\put(-65,100){$i+1$}
\put(-25,10){$i$}
\put(-55,10){$i$}
\put(-50,35){$i+1$}
\put(10,80){index i+1}
\put(10,30){index i}
\caption{swallowtail singularity}
\label{stratum2}
\end{figure}

\begin{figure}[ht]
\centering
\includegraphics[width=.8\textwidth,height=.2 \textwidth]{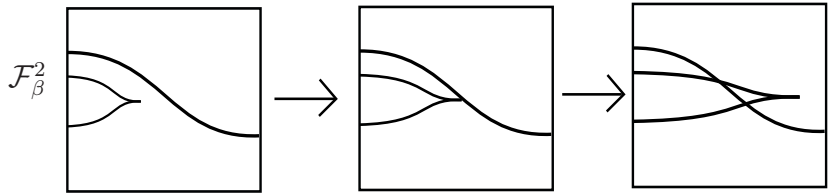}
\put(-310,40){$\mathcal{F}_{\beta}^2$}
\caption{cusp-fold}
\label{stratum5}
\end{figure}

Type 2: $\mathcal{F}_{\beta}^2$ is the set of functions for which $v_1(f)=1$ and $v_2(f)=1$. $v_1(f)=1$ implies there is only one cusp point,
either a birth cusp point or a death cusp point. Figure \ref{stratum5} uses a death cusp as an example. $v_2(f)=1$ implies this cusp point 
 and a non-degenerate point have the same critical value. We call the singularity in the middle figure of Figure \ref{stratum5} a 
 \textit{cusp-fold point} since it is the intersection of a cusp point and a fold edge.

Type 3: $\mathcal{F}_{\gamma}^2$ is the set of functions for which $v_1(f)=0$ and $v_2(f)=2$. Three non-degenerate points have the same critical
 value, thus the preimage of this critical value has 3 critical points, providing $v_2(f)=2$. $v_1(f)=0$
 implies no cusps nor swallowtail points.  See Figure \ref{stratum6}. We call the singularity in the middle figure 
 of Figure \ref{stratum6} a \textit{Reidemeister-III fold point} since the whole event looks like a Reidemeister-III type  
 move in knot theory.

\begin{figure}[ht]
\centering
\includegraphics[width=.8\textwidth,height=.2 \textwidth]{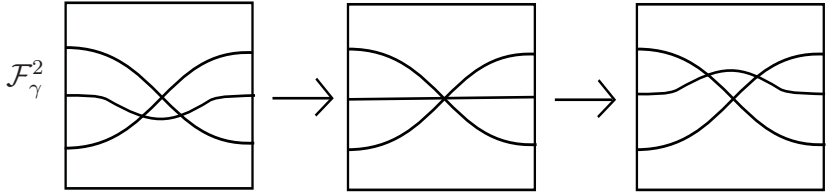}
\put(-310,40){$\mathcal{F}_{\gamma}^2$}
\caption{Reidemeister-III fold}
\label{stratum6}
\end{figure}
\end{Rmk}

Let $S$ be a surface.
We next introduce the generic homotopy of homotopies $g_{s,t}:S\to \mathbb{R}$. The definition and pictures
can be found in ~\cite{GK}. We will use the generic homotopy between homotopies of Morse functions to define the local behavior of a generic homotopy
between Morse 2-functions.

\begin{figure}[ht]
\centering
\includegraphics[width=.3\textwidth,height=.3 \textwidth]{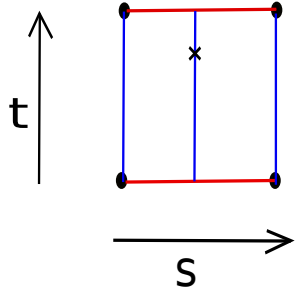}
\caption{A schematic picture for the image of a a generic homotopy between homotopies. This is an overall illustration of the following
four conditions.}
\label{m2}
\end{figure}
Given a generic homotopy $g_{0,t}:S\to \mathbb{R}$ between Morse functions $g_{0,0}$ and $g_{0,1}$, and another
generic homotopy $g_{1,t}:S\to \mathbb{R}$ between Morse functions $g_{1,0}$ and $g_{1,1}$,  a 2-parameter
family $g_{s,t}:S\to \mathbb{R}$ is called a \textit{generic homotopy between homotopies} if it
satisfies the following conditions (see Figure \ref{m2}):
\begin{enumerate}
 \item $g_{s,0}$ and $g_{s,1}$ are arcs of Morse functions for $0\leq s\leq 1$.
 
 \item $\{g_{s,t}\}_{0\leq t\leq 1}$ is a generic homotopy between Morse functions $g_{s,0}$ and $g_{s,1}$ for all but finitely many values of
$s$. In particular, we assume that $g_{0,t}$ and $g_{1,t}$ are generic homotopies.

 \item For those values $s_*$ where $\{g_{s_*,t}\}_{0\leq t\leq 1}$ is not a generic homotopy, there is a single value $t_*$ such that
$g_{s_*,t}$ is a generic homotopy for both $t\in [0,t_*)$ and $t\in (t_*,1]$.

 \item At each of these points $(s_*,t_*)$, exactly one of the following events occurs (see Figure \ref{2parameter}, \ref{3move}, 
\ref{birthpair}, \ref{swallowtail}, the middle levels in each graphic are $s=s_*$):
\begin{enumerate}[(a)]
 \item 2-parameter coincidence: Two of the events in Definition \ref{def1parameter} occur simultaneously at $t=t_*$.
 They are the three unshown cases in the stratum of $\mathcal{F}^2$. We won't need them in our discussion of Morse 2-functions but list them here
 for completeness. See Figure \ref{2parameter}.
\begin{figure}[ht]
\centering
\includegraphics[width=.3\textwidth,height=.4 \textwidth]{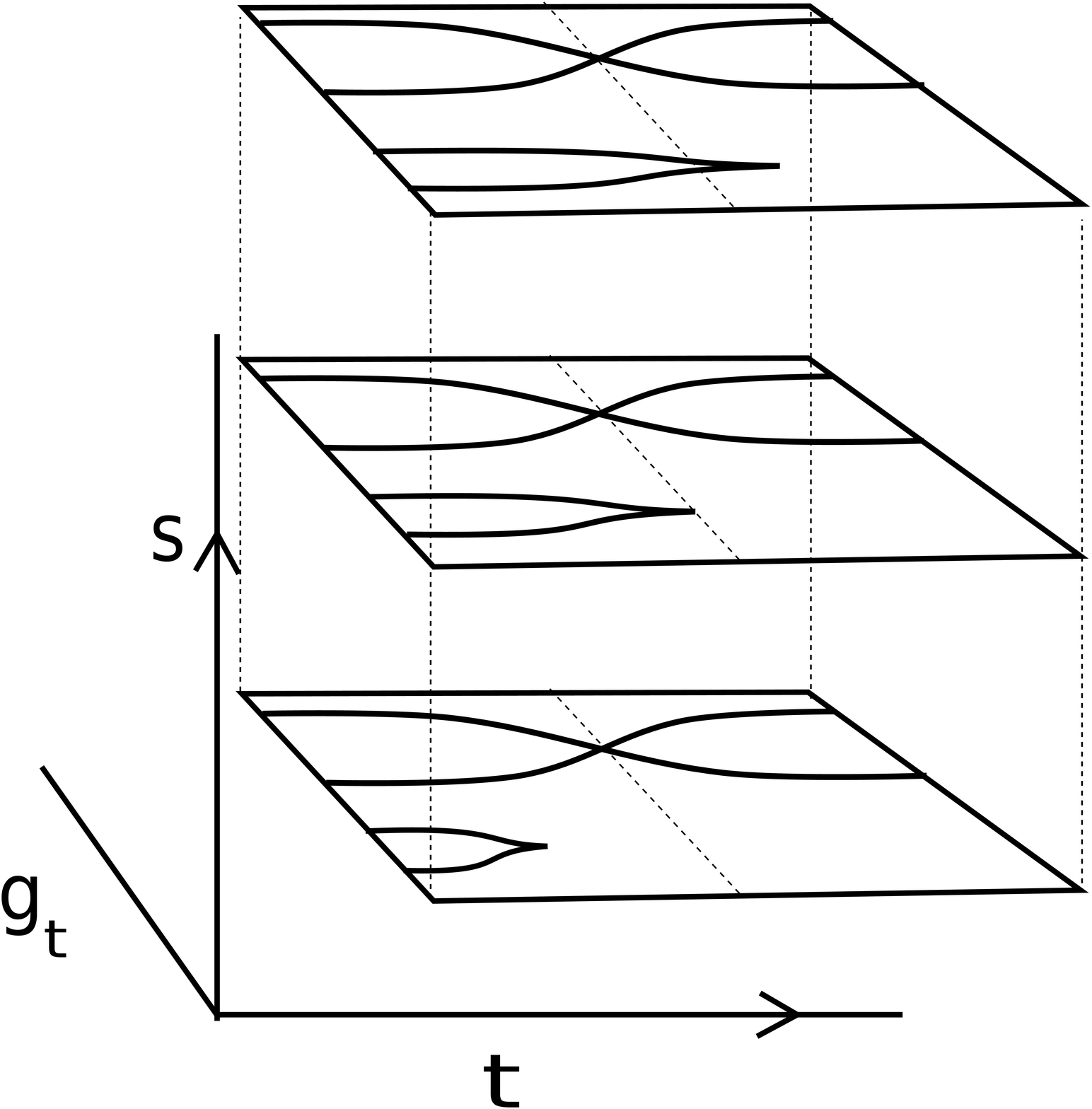}
\includegraphics[width=.3\textwidth,height=.4 \textwidth]{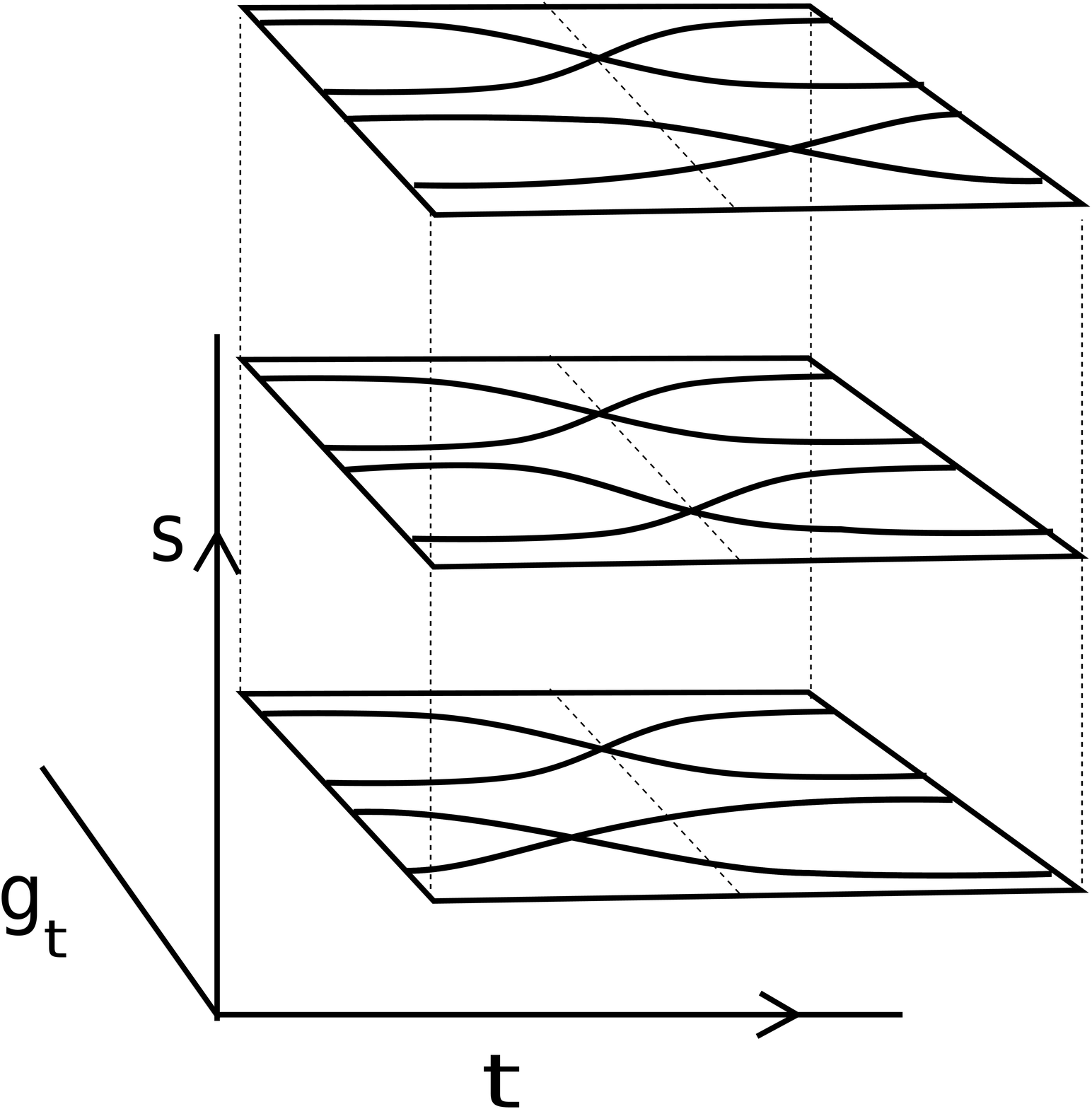}
\includegraphics[width=.3\textwidth,height=.4 \textwidth]{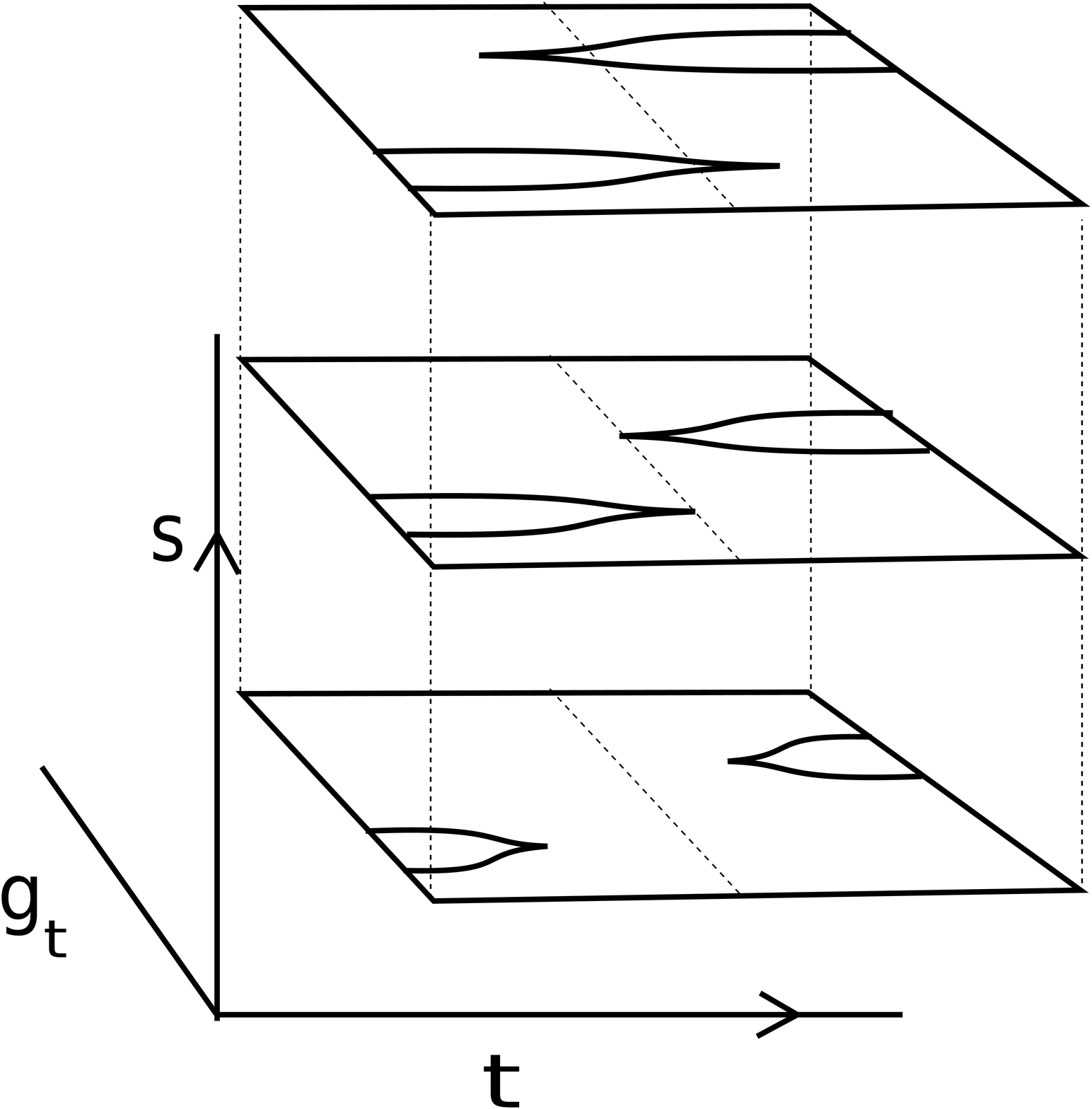}
\put(-55,28){\tiny$i$}
\put(-68,40){\tiny$i+1$}
\put(-25,35){\tiny$j$}
\put(-51,47){\tiny$j+1$}
\caption{2-parameter coincidence: Three different cases. The middle levels are 
$t=t_*$.The $i$'s and $j$'s are indices of folds.}
\label{2parameter}
\end{figure}

 \item Reidemeister-II fold crossing: In this event, two crossings are cancelled ($s<s_*$) or introduced ($s>s_*$). Note that 
 each pants move  corresponds to a crossing (Remark \ref{rmkcross}), so a cancelling pair corresponds to 
 the top level on the left of Figure \ref{3move}. From a level $s<s_*$ to a level $s>s_*$  a pants move and its inverse
 are created, thus an edge and its inverse are created in the corresponding edge path (in a pants complex). We will see that this event 
 induces a cancelling pair move in Section \ref{singularities}.
 
 \item  Reidemeister-III fold crossing: Also called a Reidemeister-III type singularity.
 Three folds intersect at $t=t_*$ simultaneously. This is case 3 in the stratum
 of $\mathcal{F}^2$. See the middle of Figure \ref{3move}. In section \ref{secpmove}, We will see that this event induces an 
 A-pentagon move, a move between pants-block decompositions of a manifold.

 \item Cusp-fold crossing: This is case 2 in the stratum of $\mathcal{F}^2$. The function $g_{s_*,t}$ fails to be a generic 
 homotopy because a cusp meets a fold point at $t=t_*$. In a neighbourhood where $|s-s_*|<\epsilon$ and $|t-t_*|<\epsilon$,
 the function $g_{s,t}$ is given by replacing $x_{n+1}^2$ in formula (\ref{nondeg}) by $x_{n+1}^3+(t-t_*)x_{n+1}$. The function $g_{s,t}$
 is Morse outside of this neighbourhood. See the right of Figure \ref{3move}.

 \begin{figure}[ht]
\centering
\includegraphics[width=.3\textwidth,height=.4 \textwidth]{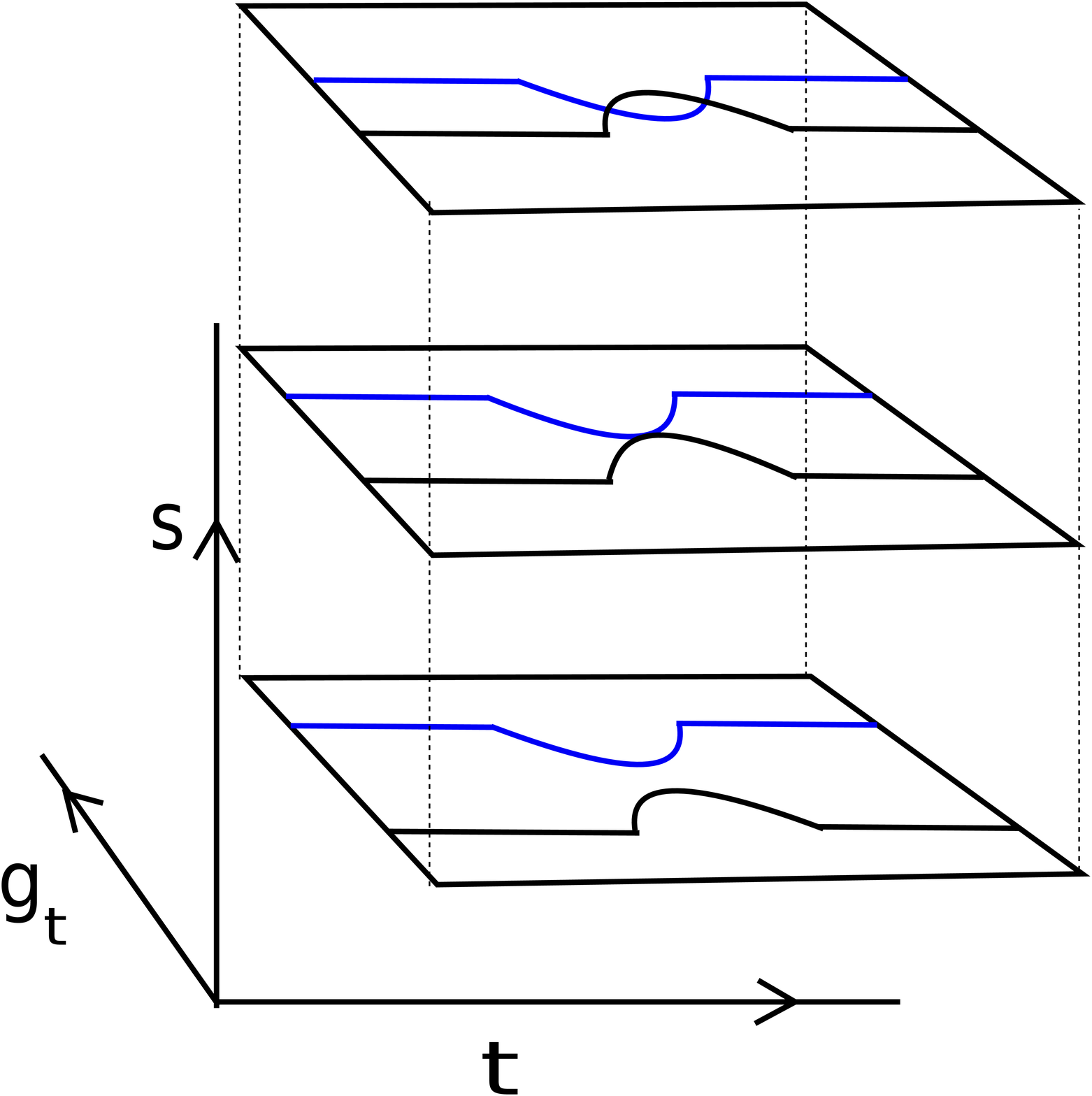}
\includegraphics[width=.3\textwidth,height=.4 \textwidth]{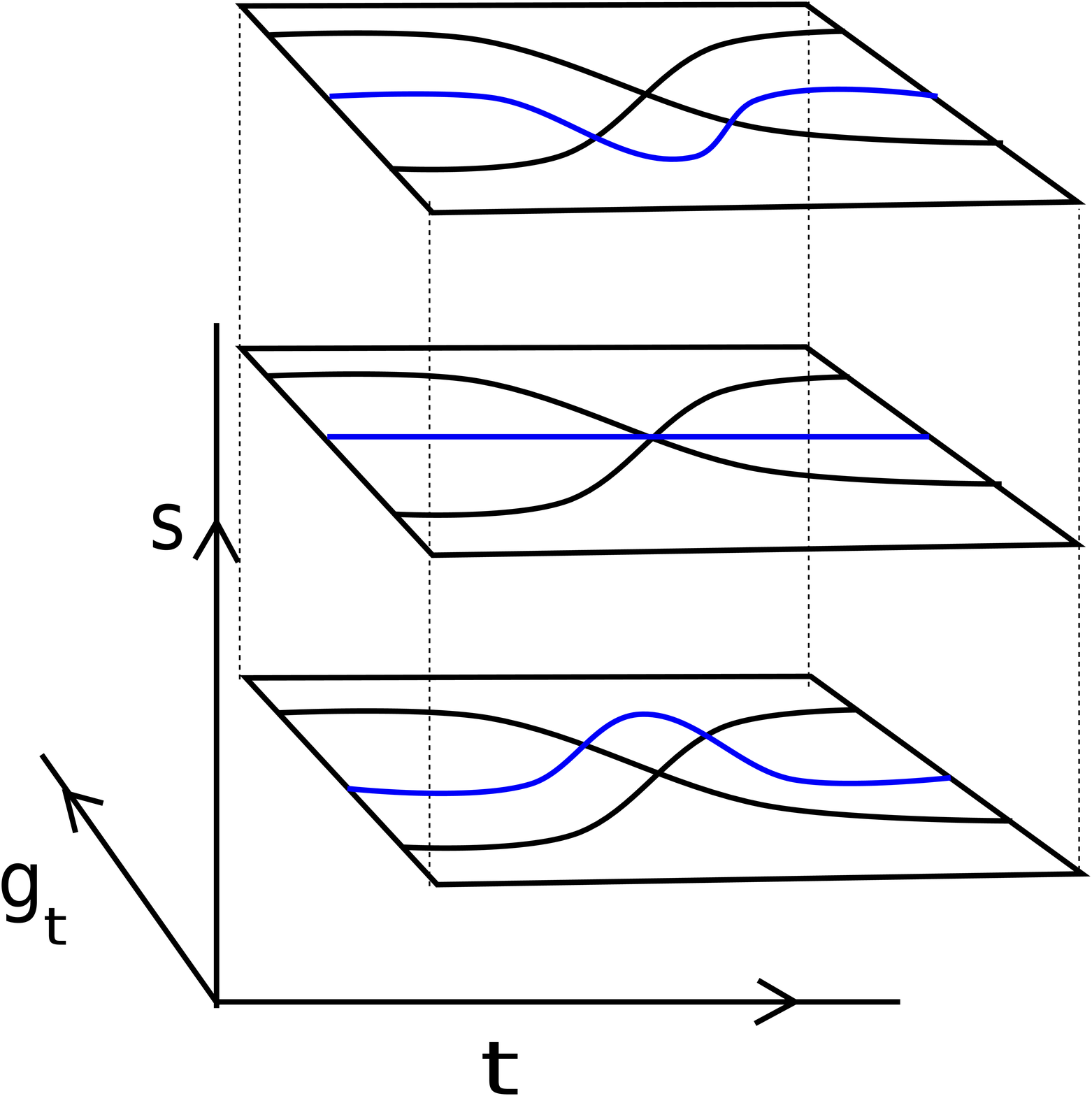}
\includegraphics[width=.3\textwidth,height=.4 \textwidth]{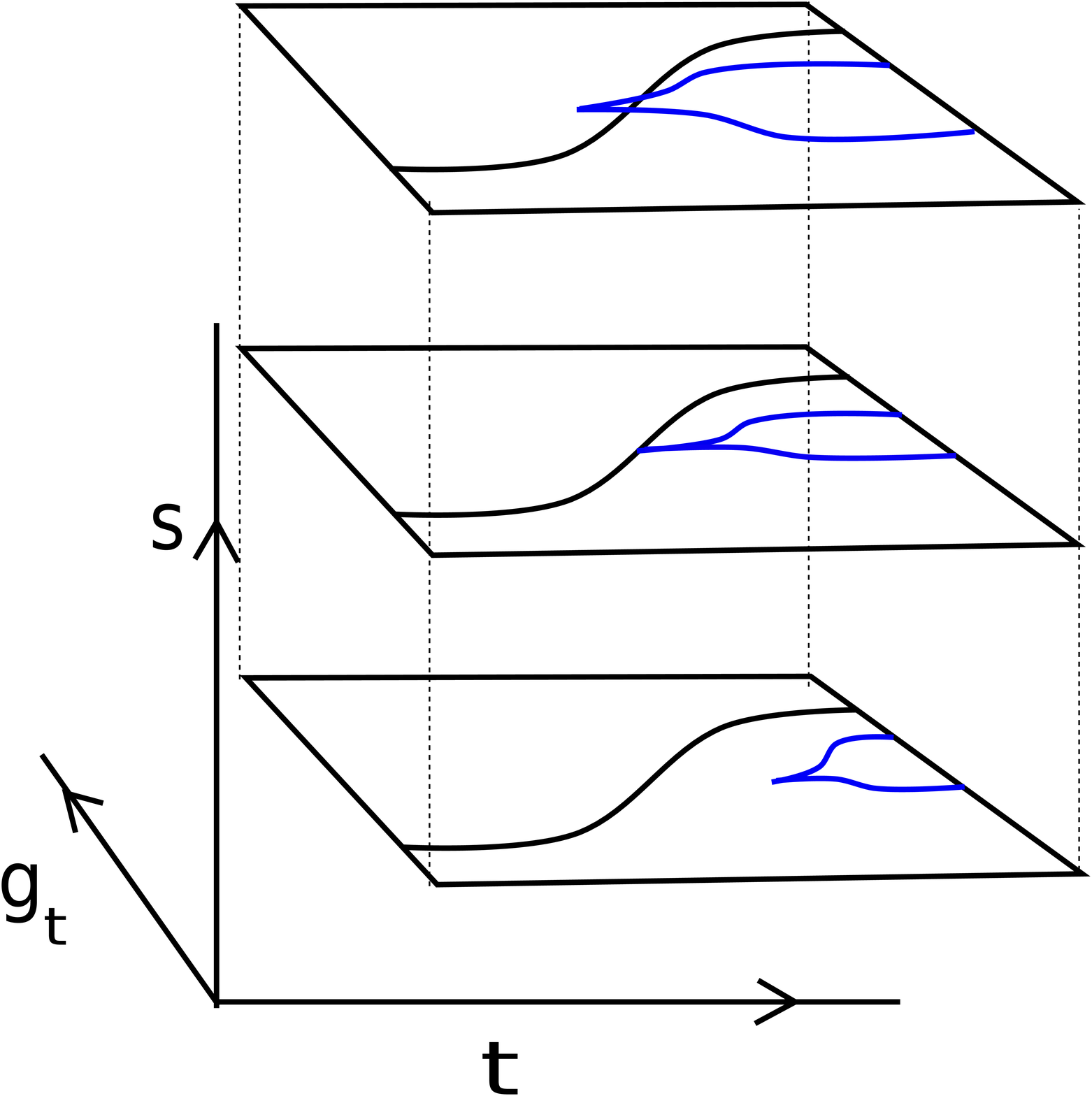}
\caption{3 cases of singularities. The middle horizontal levels are
$s=s_*$. They are from Figure 5, 6 and 7 in ~\cite{GK}.}
\label{3move}
\end{figure}

 \item Eye birth(death) singularity: In a neighbourhood where $|s-s_*|<\epsilon$ and $|t-t_*|<\epsilon$,
 the function $g_{s,t}$ is given by replacing $x_{n+1}^2$ in formula (\ref{nondeg}) by $x_{n+1}^3+(t-t_*)^2x_{n+1}+(s-s_*)x_{n+1}$. The 
 function $g_{s,t}$
 is Morse outside of this neighbourhood. Geometrically, this singularity introduces(cancels) an ``eye'' shape that is joined by 
 a pair of cusps in Figure \ref{birthpair} as $s$ increases(decreases). Figure \ref{eye} gives another way to view this singularity (suggested 
 by David Gay).
 Use vertical planes to cut the shape, the slices are empty sets before $s_1$ and after $s_2$. the intersection of
 the shape with the planes $s=s_1$ and $s=s_2$ 
 are the eye birth and eye death singularities. When $s_1<s<s_2$, each slice looks like an eye. The singularity has codimension 1
 since it is equivalent to one cusp point in an eye-shaped slice. Therefore this singularity is in $\mathcal{F}_{\beta}^1$.
 
 \item Merge singularity: In a neighbourhood where $|s-s_*|<\epsilon$ and $|t-t_*|<\epsilon$,
 the function $g_{s,t}$ is given by replacing $x_{n+1}^2$ in formula (\ref{nondeg}) by $x_{n+1}^3-(t-t_*)^2x_{n+1}+(s-s_*)x_{n+1}$. $g_{s,t}$
 is Morse outside of this neighbourhood. See the right of Figure \ref{birthpair}.
 
 \begin{figure}[ht]
\centering
\includegraphics[width=.6\textwidth,height=.3 \textwidth]{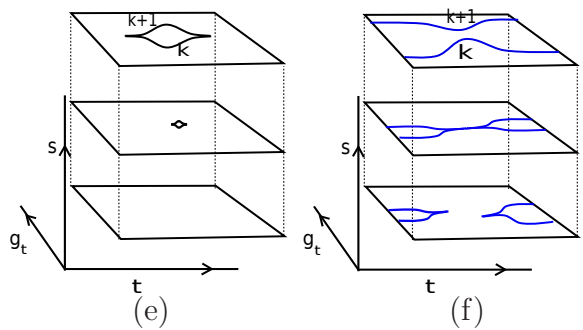}
\put(-170,-10){(e)}
\put(-50,-10){(f)}
\caption{Left: Introduce an eye-shaped;  Right: Merge of two cusps. They are from
Figure 8 and 9 in ~\cite{GK}.}
\label{birthpair}
\end{figure}

\begin{figure}[ht]
\centering
\includegraphics[width=.6\textwidth,height=.3 \textwidth]{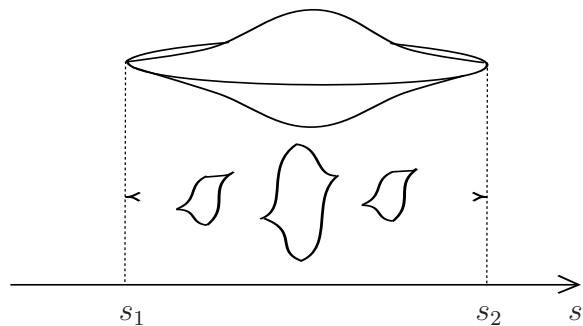}
\put(-175,-10){$s_1$}
\put(-40,-10){$s_2$}
\put(-5,-10){$s$}
\caption{Another way to view the eye birth singularity.}
\label{eye}
\end{figure}

\begin{figure}[ht]
\centering
\includegraphics[width=.6\textwidth,height=.3 \textwidth]{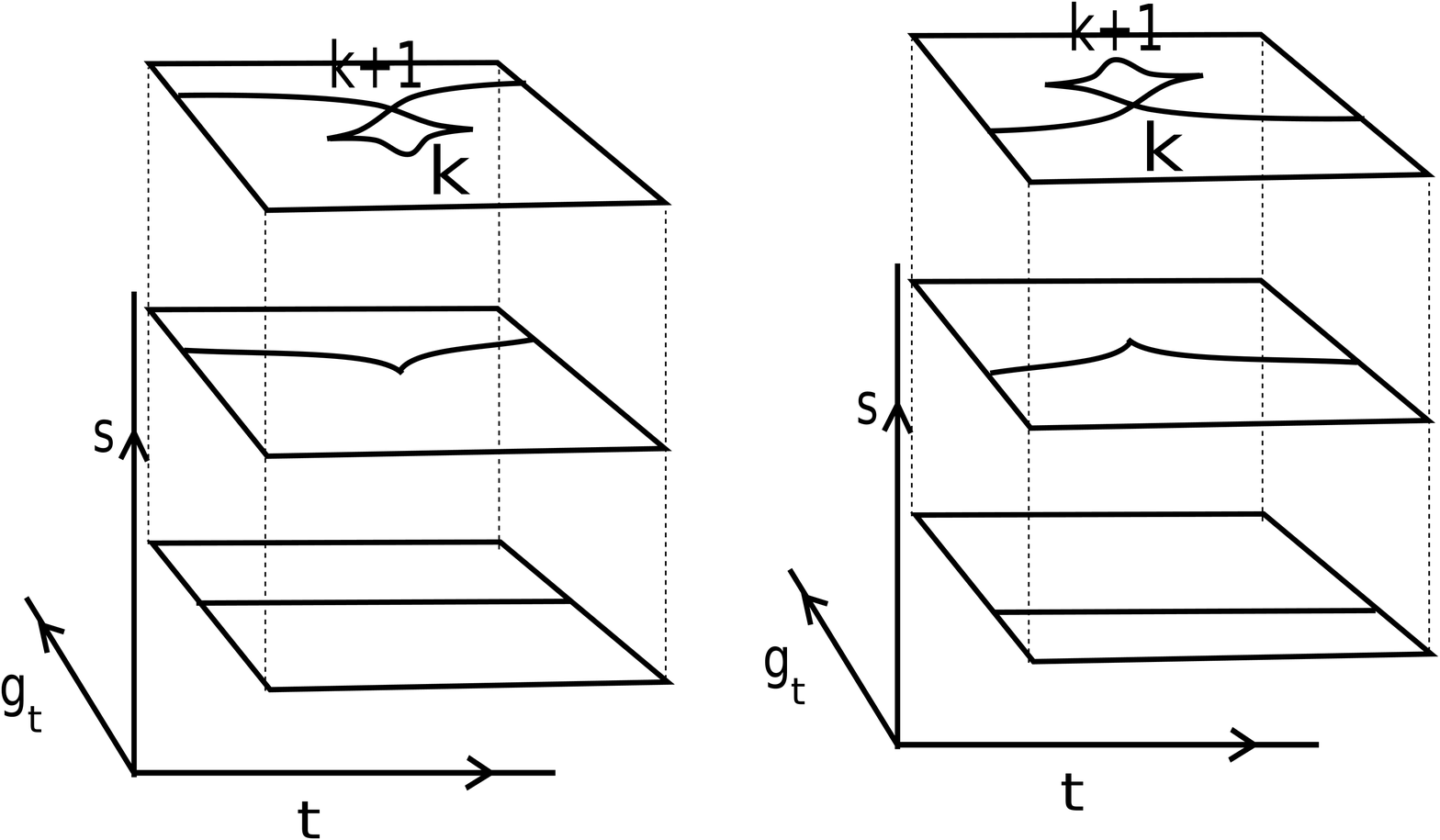}
\caption{Birth of a swallowtail. From Figure 10 in ~\cite{GK}.}
\label{swallowtail}
\end{figure}
 \item Swallowtail birth singularity: This is case 1 in the stratum of $\mathcal{F}^2$. In a neighbourhood where $|s-s_*|<\epsilon$ 
 and $|t-t_*|<\epsilon$, the function $g_{s,t}$ is given by replacing $x_{n+1}^2$ in formula (\ref{nondeg}) by 
 $x_{n+1}^4+(s-s_*)x_{n+1}^2+(t-t_*)x_{n+1}$. The function $g_{s,t}$ is Morse outside of this neighbourhood. 
See Figure \ref{swallowtail}.

\end{enumerate}
\end{enumerate}

The last four events in the list can also be found in other references. In Chapter 5 of ~\cite{HW}, Hatcher and 
Wagoner listed them as changes in the Cerf graphic. 
In  Section A of ~\cite{Lekili}, Lekili used them for the classification of (1,1)-stable unfoldings. We will describe 
induced moves corresponding to these six events (not include the first one) in Section \ref{singularities}.

\subsection{Morse 2-functions and Reeb complexes}\label{secmorse2}
Throughout this subsection, we let $M$ be a smooth, closed, connected, oriented 3-manifold unless otherwise specified. 
A Morse 2-function, is a smooth, stable map from $M$ to $\mathbb{R}^2$, which can be locally modelled by some generic
homotopy between Morse functions. The definition of Morse 2-function can be generalized to arbitrary $M^n$
and a general 2-manifold, but in this paper we will only focus on $n=3$ and $\mathbb{R}^2$. Among all of the cases, we are 
interested in the one that Gay and Kirby called \textit{indefinite Morse 2-function}.  They gave the existence
and uniqueness results for indefinite Morse 2-functions in ~\cite{GK}, and used Morse 2-functions to reconstruct 4-manifolds in ~\cite{GK2}. Here we
will introduce  Morse 2-functions (Definition \ref{morse2}) and the Reeb complexes, as well as the generic homotopy between two 
Morse 2-functions, which we will use in eliminating some singularities of Reeb complexes (Section \ref{singularities}).

The following is Definition 2.7 of ~\cite{GK}, except we restrict arbitrary 2-manifolds to $\mathbb{R}^2$.
\begin{Def}\label{morse2}
 Given an $n$-manifold $M$, a smooth proper map $G:M\to \mathbb{R}^2$ is a \textit{Morse 2-function} if for
 each $q\in \mathbb{R}^2$ there is a compact neighbourhood $S$ of $q$ with a diffeomorphism $\psi: S\to I\times I$ and a diffeomorphism
 $\phi: G^{-1}(S)\to I\times N$, for an $(n-1)$-manifold $N$ where $I\times N\subset M$, such that $\psi\circ G\circ \phi^{-1}: I\times N\to I\times I$
 is of the form $(t,p)\to (t,g_t(p))$ for some generic homotopy $g_t:N\to I$ between Morse functions .
\end{Def}

The set of critical points of $G$, denoted by $C(G)$, can be classified as follows (see ~\cite{Levine}):
\begin{Def}
 Let $G:M\to \mathbb{R}^2$ be a Morse 2-function. There exists neighbourhoods near each point $p\in C(G)$ and $G(p)\in \mathbb{R}^2$, so that
 $G$ is one of the following forms:
 
 \begin{enumerate}
  \item $(u,x,y)\to (u, x^2+y^2)$, when $p$ is a definite fold point;
  
  \item $(u,x,y)\to (u, x^2-y^2)$, when $p$ is an indefinite fold point;
  
  \item $(u,x,y)\to (u, y^2+ux-\frac{x^3}{3})$, when $p$ is a cusp point.
 \end{enumerate}
In addition the following global conditions are satisfied:

\begin{enumerate}[(a)]
 \item If $p$ is a cusp point, then $\{p\}=G^{-1}(G(p))\cap C(G)$;
 
 \item $G|_{C(G)-\{cusps\}}$ is an immersion with normal crossings.
\end{enumerate}

\end{Def}

The sets of singularities of Morse 2-functions are \textit{folds} and \textit{cusps}. There are two types of folds, a \textit{definite
fold} is an arc of definite fold points, while an \textit{indefinite fold} is an arc of indefinite fold points. A \textit{cusp} contains a cusp point with two
branches of folds, one branch is definite and the other is indefinite, see Figure \ref{morsecri}. A \textit{crossing}
is the intersection point of images of two indefinite folds in $\mathbb{R}^2$. We will give 
 details about images of each type of critical points in Reeb complexes.
To know more about folds and cusps, we recommand ~\cite{KS},~\cite{Levine} and  ~\cite{Saeki}.  
A Morse 2-function is \textit{indefinite} if it has no definite folds.

\begin{figure}[ht]
\centering
\includegraphics[width=.4\textwidth,height=.3 \textwidth]{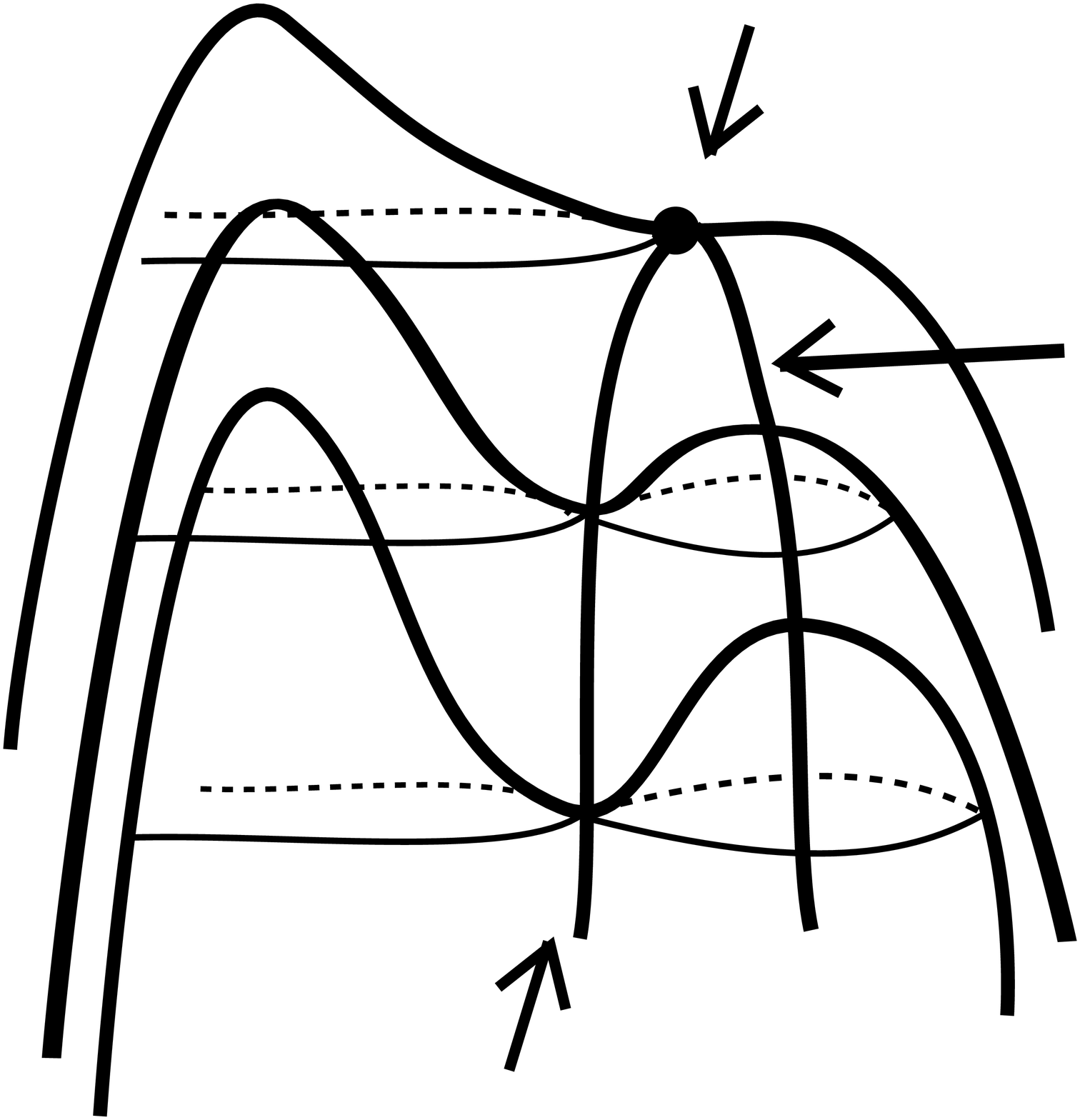}
\put(-50,110){cusp point}
\put(0,70){definite fold}
\put(-100,-10){indefinite fold}
\caption{An example of singularity sets of a Morse 2-function. The definite fold is a collection of local maxima of level
surfaces, the indefinite fold is a collection of saddle points of level surfaces.}
\label{morsecri}
\end{figure}

Similar to the generic homotopy of Morse functions, we can now define a generic homotopy between Morse 2-functions.
We will need the generic homotopies between Morse 2-functions in the proof of the main theorem. 
\begin{Def}
 A homotopy $F_s: M\to \mathbb{R}^2$ is a \textit{generic homotopy between Morse 2-functions} $F_0$ and $F_1$ if,
 for each $q\in \mathbb{R}^2$ and each $s_*\in I$ ($I=[0,1]$), there is an $\epsilon>0$ and a compact neighbourhood
 $U$ of $q$ with a diffeomorphism $\psi: U\to I\times I$ and a 1-parameter family of diffeomorphism $\phi_s: F_s^{-1}(U) \to
 I\times S$, for a surface $S$ and for $|s-s_*|<\epsilon$, such that $\psi\circ F_s\circ \phi_s^{-1}: I\times S\to I\times I$
 is of the form $(t,p)\to (t,g_{s,t}(p))$ for some generic homotopy of homotopies $g_{s,t}: S\to I$.
\end{Def}
 Note that a Morse 2-function is determined by a generic homotopy of Morse functions, thus a generic homotopy of 
 Morse 2-functions is determined by a generic homotopy between homotopies of Morse functions. as descibed in the previous section.

A \textit{Reeb complex} is the quotient space  of a manifold under some Morse 2-function. We can visualize the 
singularities in the generic homotopy of Morse 2-functions by understanding Reeb complexes.
Since all but finitely many slices of a Reeb complex are locally Reeb graphs, the definition of a Reeb complex is similar to the definition
of a Reeb graph. 
\begin{Def}\label{rc}
Given a compact, closed, orientable 3-manifold $M$ and a Morse 2-function $F:M\to \mathbb{R}^2,$ define the equivalence
relation $\sim$ on M by $x\sim y$ whenever $x,y\in M$ are in the same component of a preimage of a point in $\mathbb{R}^2$.
Similar to the two-dimensional case, there is a Stein factorization composing maps from $M$ to the \textit{Reeb complex} $\mathcal{RC}=M/\sim$ and from
$\mathcal{RC}$ to $\mathbb{R}^2$ such that the composition $M\to \mathcal{RC}\to \mathbb{R}^2$ is $F$.
\end{Def}

In general, a Reeb complex is not a manifold, but is homeomorphic to a 2-dimensional finite CW complex.
See~\cite{Jesse} for a discussion  of each type of critical point. For a better understanding of
the local structure of a Reeb complex and the upcoming discussion of the P-complex, we summarize that discussion here. In the rest of 
this section, we always assume that the smooth arcs are locally reparametrized as vertical arcs.

 \begin{figure}[ht]
\centering
\includegraphics[width=.6\textwidth,height=.2 \textwidth]{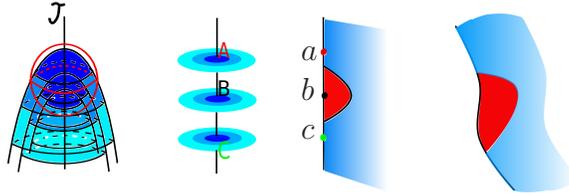}
\put(-105,50){$a$}
\put(-105,35){$b$}
\put(-105,20){$c$}
\caption{A definite fold edge, the second left figure is a reparametrized figure of the  left one, while the 
 second right figure is a reparametrized figure of the right one.}
\label{fig14}
\end{figure}

Case 1: definite fold points, see Figure \ref{fig14}. Let $a,b,c$ be the images of definite points $A,B,C$. The preimage of a neighbourhood $U_b$ (the disk in the 
second right figure) is shown on the left (inside the circle). The second left figure is a reparametrized 
result of the left figure, i.e., flatten each level surface and make each level curve a perfect circle. One can imagine 
that the preimage of $U_b$ in the second left figure is a shape of an American football.

 \begin{figure}[ht]
\centering
\includegraphics[width=.4\textwidth,height=.2 \textwidth]{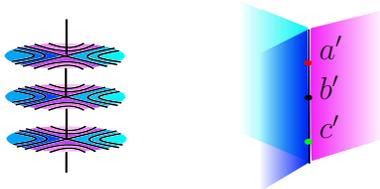}
\put(-25,50){$a'$}
\put(-25,35){$b'$}
\put(-25,20){$c'$}
\caption{An indefinite fold edge, the left figure is a reparametrized figure of a collection of saddles surfaces.}
\label{fig15}
\end{figure}

Case 2: indefinite fold points, see Figure \ref{fig15}. Let $a',b',c'$ be the images of indefinite fold points in the Reeb complex. If
$b'$ is not a crossing (i.e.,the intersection of two fold edges), the preimage of $b'$ is a figure 8, i.e., the wedge sum of two circles. 
The neighbourhood $U_{b'}$ is of the shape that
three disks intersect at a common boundary arc. Let $T_{b'}$ be the intersection of a level set at $b'$ with $U_{b'}$, 
which is $b'$ with three edges associated to it. The preimage of $T_{b'}$ is a level surface on the left of
Figure \ref{fig15}.

\begin{figure}[ht]
\centering
\includegraphics[width=.3\textwidth,height=.2 \textwidth]{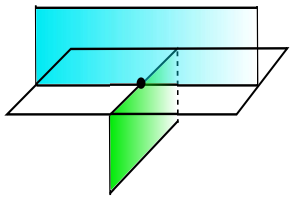}

\caption{The neighbourhood of an entangled crossing in the $\mathcal{RC}$.}
\label{figcross}
\end{figure}
Case 3: crossing. If $b'$ as in the previous paragraph is a crossing, we can choose a neighbourhood $U_{b'}$ so that it contains only one crossing since the number of crossings is finite.
Note that by the previous paragraph, the preimage of a non-crossing indefinite fold point is a graph of one valence-four vertex with two
edges, so the preimage of a crossing is a graph of two valence-four vertices with four edges. There are two cases. If the graph is 
disconnected, i.e., it contains two wedge sums of  two circles, then the local behavior of the $\mathcal{RC}$ is as in the non-crossing case.
Bachman and Schleimer ~\cite{BS} call this an \textit{unentangled crossing}.
If the graph is connected, then the two valence-four vertices are mapped to the same point $b'$ in the $\mathcal{RC}$, and we get a
valence-four vertex $b'$ and six two cells adjacent to it,  as in Figure \ref{figcross}. This $b'$ is called an \textit{entangled crossing}.

 \begin{figure}[ht]
\centering
\includegraphics[width=.8\textwidth,height=.2 \textwidth]{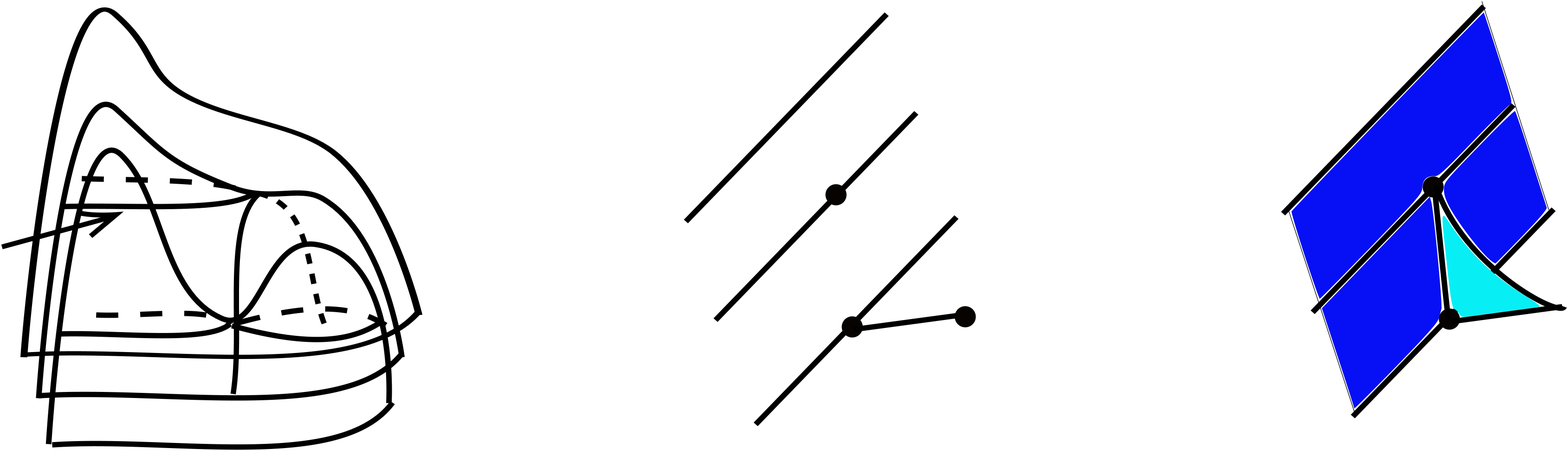}
\put(-143,45){$b''$}
\put(-140,20){$c''$}
\put(-330,40){preimage}
\put(-320,30){of $b''$}
\caption{A neighbourhood of the image of a cusp point. The middle figure shows the Reeb graphs of three surfaces on the left,
the right figure is a collection of Reeb graphs in the Reeb complex.}
\label{fig16}
\end{figure}
Case 4: cusp points, see Figure \ref{fig16}. The preimage of middle edge in the middle figure is the middle level surface on the left.
The  preimage of $b''$, as suggested in the figure, contains a cusp point. The neighbourhood $U_{b''}$ of $b''$ is the right of 
Figure \ref{fig16}.

\subsection{Induced moves}\label{singularities}
The singularities of Morse 2-functions induce singularities of the Reeb complexes. In this section we will explain
how to use some of the singularities in a generic homotopy between Morse 2-functions listed in Section \ref{secmorse2} to eliminate some singularities of Reeb complexes. Namely, we want
to eliminate those definite fold edges and cusp points, and reduce the intersections of indefinite fold edges,
so that the Reeb complex becomes an object called a P-complex, which we will define
in Section \ref{sec3}. All of these moves are called \textit{induced moves} of Morse 2-functions.

Move 1: Add or eliminate two crossings that represent a cancelling pair, as in Figure \ref{fig2cancell}. Two indefinite
fold edges in an $\mathcal{RC}$ intersect at two points, and no other crossings appear in a neighbourhood of figure $(a)$
in Figure \ref{fig2cancell}. These two crossings can be eliminated by pulling the two folds away from each other 
so that the crossings move towards each other first, merge, then disappear. This process can be realized by 
adjusting the Morse 2-function so that the Cerf graphics change as shown in $(a)\to(c)$. 
Figure $(a')$ to $(c')$ are the preimages of $(a)$ to $(c)$ in the $\mathcal{RC}$.
\begin{figure}[ht]
\centering
\includegraphics[width=.8\textwidth,height=.3 \textwidth]{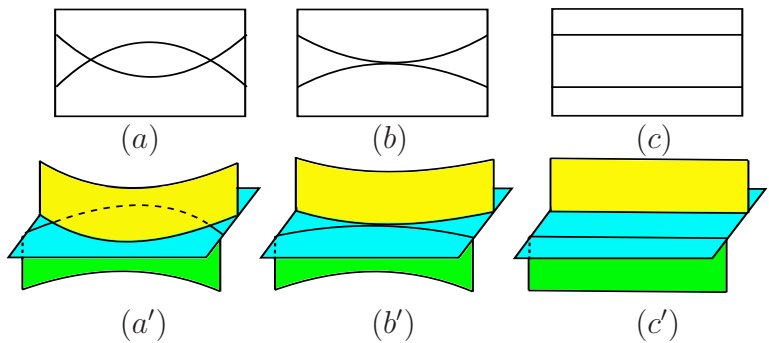}
\put(-245,55){$(a)$}
\put(-150,55){$(b)$}
\put(-50,55){$(c)$}
\put(-245,-15){$(a')$}
\put(-150,-15){$(b')$}
\put(-50,-15){$(c')$}
\caption{Eliminate a cancelling pair.}
\label{fig2cancell}
\end{figure}

\begin{figure}[ht]
\centering
\includegraphics[width=.95\textwidth,height=.4 \textwidth]{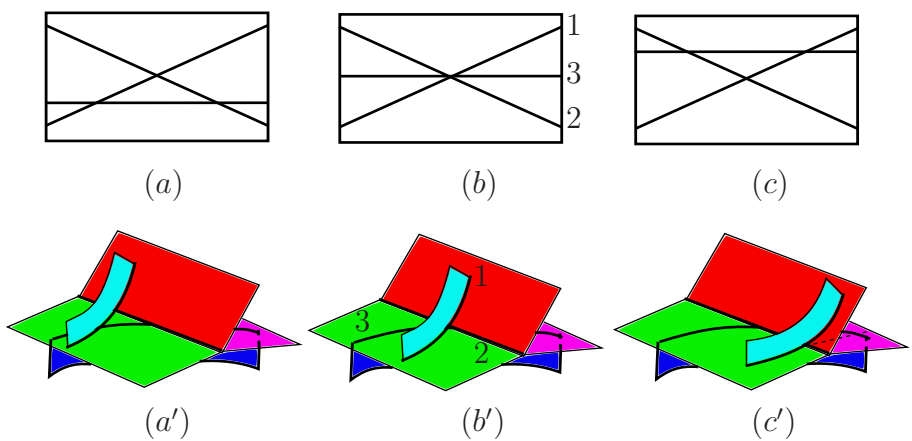}
\put(-290,75){$(a)$}
\put(-170,75){$(b)$}
\put(-60,75){$(c)$}
\put(-290,-15){$(a')$}
\put(-170,-15){$(b')$}
\put(-60,-15){$(c')$}
\put(-165,40){1}
\put(-165,11){2}
\put(-210,23){3}
\put(-130,135){1}
\put(-130,100){2}
\put(-130,117){3}
\caption{Swipe a fold edge across a crossing. In order to see the intersection points clearly, we draw a real line for the intersection
between the dark blue 2-cell and other 2-cells. We label the three fold edges with numbers in two of the subfigures, other subfigures
have the same labels.}
\label{figr3}
\end{figure}

Move 2: Swipe a fold edge across a crossing, as in Figure \ref{figr3}. Three indefinite fold edges pairwisely intersect
in a Reeb complex as shown in ($a'$). Their images under a Morse 2-function is shown in ($a$). In order to describe things clearly,
we label the three fold edges as follows: fold edge 1 is the one with positive slope in Figure \ref{figr3},
fold edge 2 is the one with negative slope, fold edge 3 is the horizontal one.
When fold edge 3 in ($a$) moves up to the position in ($c$), passing through the crossing as shown in ($b$), we can 
push the dark blue 2-cell across the crossing of the other two fold edges in the Reeb complex. Note that one of the three intersection
points in ($c'$) is ``unentangled'': fold edge 1 and fold edge 3 don't intersect in the Reeb complex, but the projection of fold edge 1
on the pink 2-cell intersect fold edge 3 at some point. Thus the intersection of fold edge 1 and 3 in $(c)$ is an unentangled crossing. 
This move corresponds to the P-move of 2-3 type, which we will further discuss in Section \ref{secpmove}.

 \begin{figure}[ht]
\centering
\includegraphics[width=.5\textwidth,height=.3 \textwidth]{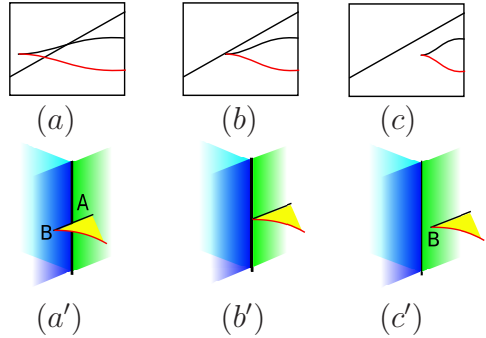}
\put(-170,60){$(a)$}
\put(-100,60){$(b)$}
\put(-40,60){$(c)$}
\put(-170,-15){$(a')$}
\put(-100,-15){$(b')$}
\put(-40,-15){$(c')$}
\caption{The red edges are definite fold edges while the black edges are indefinite fold edges, point $A$ is the intersection
of two indefinite fold edges while $B$ is a cusp point.}
\label{fig17}
\end{figure}

Move 3: Eliminate the intersection of a cusp and an indefinite fold edge, as in Figure \ref{fig17}. Two indefinite fold edges in a Reeb complex
intersect at a point $A$, and one of them is adjacent to a definite fold edge at a cusp point $B$. This local picture corresponds to
Figure \ref{3move}(c). Point $A$ can be eliminated if we adjust the Morse 2-function by a family of homotopies
(i.e., from $(a)$ to (c) in Figure \ref{fig17}) so that the cusp is moved away from the indefinite fold edge.
Figures $(a')$ to $(c')$ are preimages of $(a)$ to $(c)$ in a Reeb complex.
\begin{figure}[ht]
\centering
\includegraphics[width=.6\textwidth,height=.2 \textwidth]{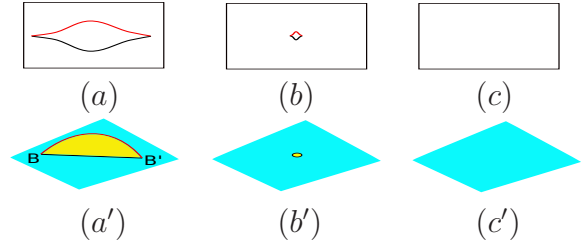}
\put(-190,33){$(a)$}
\put(-115,33){$(b)$}
\put(-40,33){$(c)$}
\put(-190,-15){$(a')$}
\put(-115,-15){$(b')$}
\put(-40,-15){$(c')$}
\caption{The red edges are definite fold edges while the black edges are indefinite fold edges,
$B$ and $B'$ are the intersections of definite and indefinite fold edges.}
\label{fig18}
\end{figure}

Move 4: Eliminate two cusp points by an eye shaped cancellation, see Figure \ref{fig18}. Two cusp points $B$ and $B'$ in
$\mathcal{RC}$ are connected as shown in ($a'$), connected by a definite fold edge and an indefinite fold edge. There is
no other indefinite fold edge crossing the indefinite fold edge connecting $B$ and $B'$. (If there is one,  we can eliminate 
the intersection using Move 1.) The corresponding Cerf graphics show that the eye shape singularities are elimiated
by adjusting the Morse 2-function, which means the two edges connecting $B$ and $B'$ together with the region 
they cobound will disappear. This move corresponds to eliminating
the valence-one vertices and their adjacent valence-three vertices in the Reeb graph slices, which we will discuss in Section \ref{sec4}.

Move 5: Eliminate two cusp points by  merging a death-birth pair, as in Figure \ref{fig19}. Two cusps pointing to each other
can be eliminated by merging their cusp points, which can be realized by adjusting the Morse 2-function to cancel
a death-birth pair of cusp points. After this move we obtain $(c')$, which contains definite fold edges.
\begin{figure}[ht]
\centering
\includegraphics[width=.6\textwidth,height=.2 \textwidth]{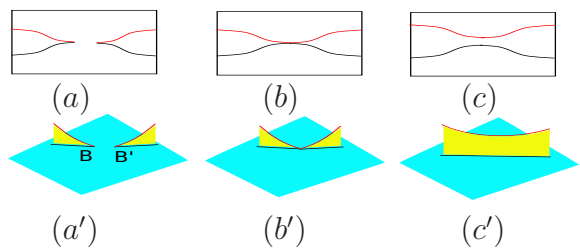}
\put(-200,35){$(a)$}
\put(-120,35){$(b)$}
\put(-45,35){$(c)$}
\put(-200,-15){$(a')$}
\put(-120,-15){$(b')$}
\put(-45,-15){$(c')$}
\caption{Cancel a pair of cusp points}
\label{fig19}
\end{figure}
\begin{figure}[ht]
\centering
\includegraphics[width=.6\textwidth,height=.4 \textwidth]{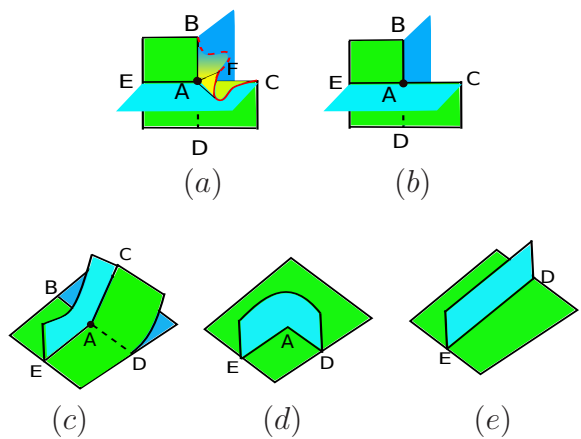}
\put(-150,75){$(a)$}
\put(-70,75){$(b)$}
\put(-200,-15){$(c)$}
\put(-120,-15){$(d)$}
\put(-40,-15){$(e)$}
\caption{Eliminate a swallowtail singularity, the corresponding Cerf graphics are shown in Figure ~\ref{swallowtail},
with $s$ direction reversed.}
\label{fig20}
\end{figure}

Move 6: Create or eliminate a swallowtail singularity, as in Figure \ref{swallowtail} and Figure \ref{fig20}. Figure \ref{fig20}(a)
corresponds to the highest level in Figure \ref{swallowtail}, where point $A$ is a vertex of valence four. As the level
goes down in Figure \ref{swallowtail}, the 2-cell bounded by the definite fold edge and two branches of $A$ in (a) is getting
smaller and smaller, it disappears until the middle level in Figure \ref{swallowtail}, we obtain (b). Drawings (c) and (d)  are
simplifications of (b) since two branches $AB$ and $AC$ contain only valence-two points, they are absorbed by the 2-cells
containing them. We can realize from (d) to (e) by moving the level in Figure \ref{swallowtail} from middle to bottom,
which means erasing the inessential valence-two vertex.

The uniqueness theorem of Morse 2-functions given by Gay-Kirby ~\cite{GK} implies the following lemma:

\begin{Lem}
 Any two Reeb complexes of a Manifold defined by two Morse 2-functions are related to each other by a finite sequence of the induced moves
 descibed above.
\end{Lem}

The proof of this lemma is straight forward, the main idea is that any two Morse 2-functions are related by a generic homotopy and each
induced move is obtained from one type of singularities in a generic homotopy. We leave the details to the readers. This lemma gives us 
a method to eliminate some of the singularities, thus change the Reeb complex, which leads to the discussion of the next section.

\section{P-graph and P-complex}\label{secpcomplex}

\subsection{P-graph}\label{sec4}
Hatcher and Thurston discussed combinatorial graphs corresponding to pants decompositions of surfaces in 
the appendix of~\cite{HT}. They defined the graph in a way that is relatively easy to understand: the preimage
of a vertex in the graph is a pair of pants and the preimage of an edge is a circle in the pants decomposition. 
However, they cannot specify the preimage of each point in the edges from this definition. In order to
obtain a similar quotient map as in the definition of Reeb graph, we need to define this graph and the map from a surface
to such a graph in a slightly different
way, and call it \textit{P-graph}. This graph was mentioned as $G'$ in Section \ref{morsereeb}. As in Figure \ref{fig10}, we first choose two different points in the interior of
a pair of pants, and choose three pairwise non-homotopic simple arcs joining them, forming a \textit{$\theta$-graph}. The 
choice of the $\theta$-graph is unique up to isotopy. Let $\Theta$ denote the $\theta$-graph.

\begin{figure}[ht]
\centering
\includegraphics[width=.6\textwidth,height=.3 \textwidth]{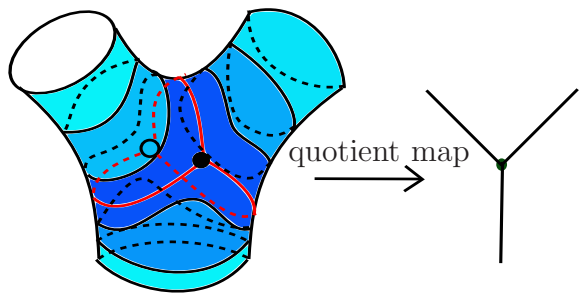}
\put(-110,50){quotient map}
\caption{The preimage of the vertex is the $\theta$-graph.}
\label{fig10}
\end{figure}
\begin{Def}
  Let $S$ be a compact, orientable surface and $\mathcal{P}$ a pants decomposition of $S$. Let 
  $S\backslash \mathcal{P}=\{Y_1,...,Y_N\}$. For each pair of pants $Y_i$, choose a $\theta$-graph $\Theta_i$ in $Y_i$, 
  then $Y_i\backslash \Theta_i$ contains three open annuli. For each open annulus in $S$, we reparameterize it as $C\times I$ where $C$
  is an essential loop and $I=(0,1)$. Define the equivalence relation on points in $S$ by $x\sim_P y$ whenever $x,y\in S$ are in the same $\Theta_i$
  or $x$ and $y$ are in $C\times \{t\}$ for some $t$ in an annulus. The \textit{P-graph} of $S$ is the quotient of $S$ under the 
  relation $\sim_P$. 
 \end{Def}

The preimage in $S$ of a vertex in a P-graph is a $\theta$-graph, and the preimage in $S$ of a point on each edge of a P-graph is a loop as in 
Figure \ref{fig10}. Each loop in $\mathcal{P}$ is mapped to a point under this quotient.
The image of a pair of pants under the relation $\sim_P$ is a shape of Y. 
By definition, the pair-of-pants are  in one-to-one correspondence with  Y-shapes, therefore a compact orientable surface with a pants 
decomposition defines a unique P-graph. We can form a P-graph for a surface by gluing together Y-shapes while gluing together the corresponding
pairs of pants. Here we need to pay attention to the following:
1)If we glue one boundary component of a pair of pants to one boundary component of another pair of pants, we need to connect one branch
of their Y-shapes together;
2)If we glue two boundary components of a pair of pants together, we need to connect two branches of the Y-shape. 
For closed surfaces, their P-graphs contain only valence-three vertices, for 
surfaces with boundaries, their P-graphs also contain some valence-one vertices. The preimages of the valence-one 
vertices are the boundary components of the surface. Figure \ref{fig12} shows examples
of P-graphs corresponding to given pants decompositions on a 
genus 4 closed surface.

\begin{figure}[ht]
\centering
\includegraphics[width=.6\textwidth,height=.3 \textwidth]{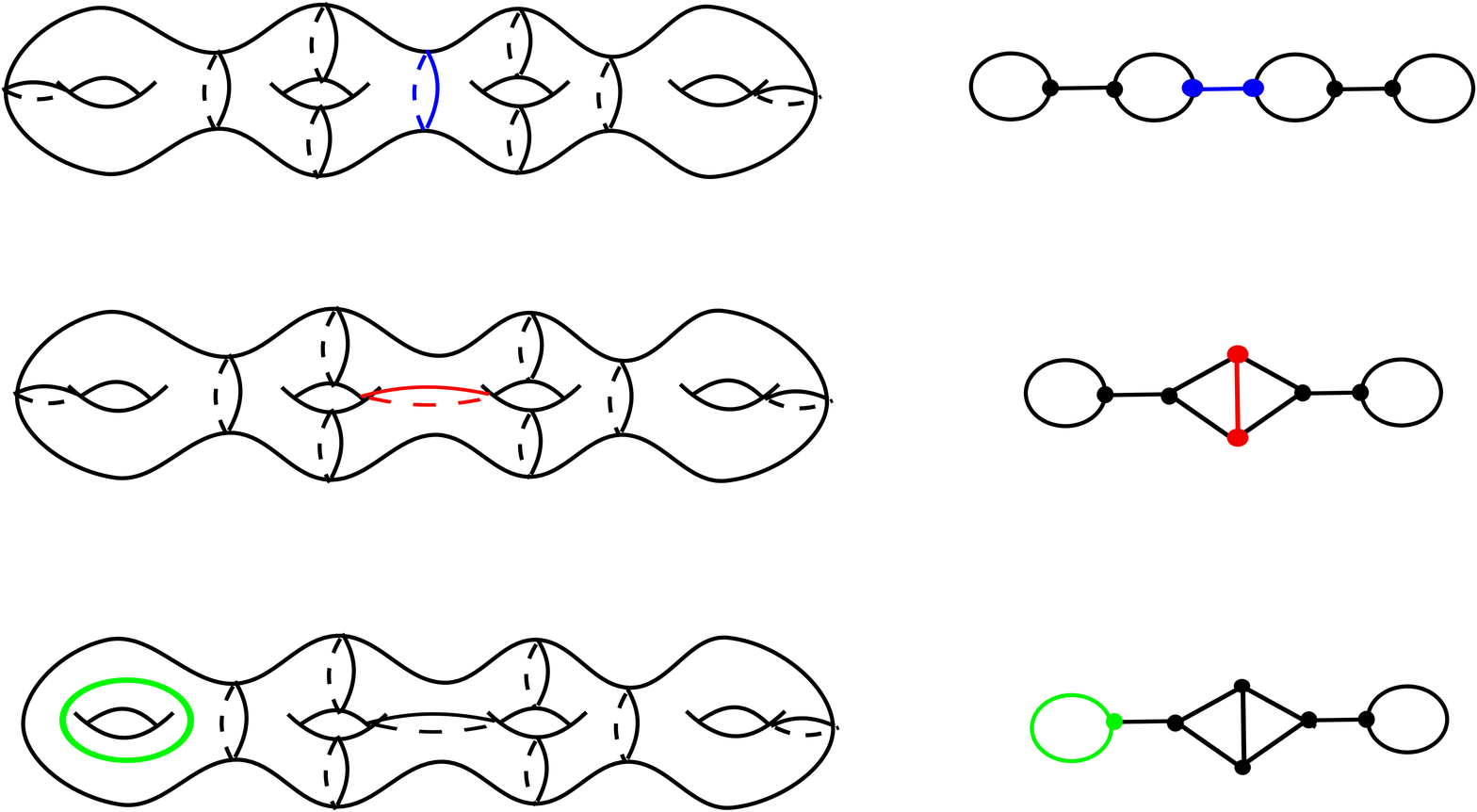}

\caption{Pants decompositions of a (4,0)-surface and their P-graphs,}
\label{fig12}
\end{figure}

We next want to discuss how to convert a Reeb graph into a P-graph for the same surface. All we need to do is to discuss the relations
between vertices and edges. Let's look at the vertices of Reeb graphs first. As we noted after Definition ~\ref{reeb}, there are two types of
valence-one vertices in the Reeb graph: images of boundary components and local extrema. We call them \textit{boundary vertices} and \textit{extremal vertices} respectively.
There are also two types of valence-three vertices: the \textit{trivial-saddle vertices} are those whose preimages in $S$ are trivial saddles, others 
are \textit{inner vertices}. A P-graph also has valence-one and valence-three vertices, but each of them has only 
one type. By the arguments in Section \ref{morsereeb}, we can show that from a Reeb graph to a P-graph, boundary vertices and 
inner vertices in the Reeb graph are in one-to-one correspondence with valence-one and valence-three vertices in the 
P-graph, while all extremal vertices and trivial-saddle vertices disappear. For a better understanding, we present a
combinatorial way to show this correspondence. 

\begin{figure}[ht]
\centering
\includegraphics[width=.6\textwidth,height=.2\textwidth]{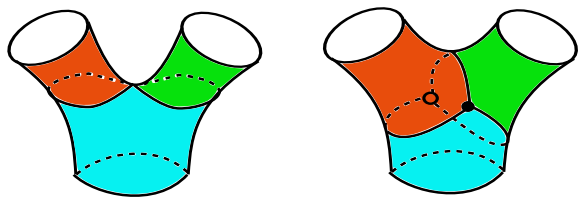}

\caption{An illustration of the preimages of the one-to-one correspondence between inner vertices of Reeb graphs and valence-three vertices of P-graphs}
\label{fig13}
\end{figure}

If we glue a disk to one of the boundary components of a pair of pants, the $\theta$-graph can be homotoped
to a circle, and we need to elimiate one edge in the corresponding shape of $Y$. This explains why the extremal vertices and trivial-saddle
vertices of Reeb graphs don't exist in P-graphs. The correspondence between boundary vertices is straightforward, 
hence we only need to descibe
the one-to-one correspondence between inner vertices in a Reeb graph and the valence-three vertices in the corresponding P-graph (see Figure \ref{fig13}): 
 Take a neighbourhood (i.e., a pair of pants) on the surface for the preimages of both vertices (a figure-eight loop and a $\theta-$graph) 
 respectively.
 For one direction, slide the two points in a $\theta$-graph along one of the three arcs until
 they merge, then this arc disappears and the other two arcs together with the merged vertex form a figure-eight loop, 
 which is (a component of) a level set that contains a saddle point.
  For the reverse direction, take a smaller neighbourhood of the 
 figure-eight loop (smaller than the one shown in Figure \ref{fig13}), choose the boundary component which is sitting on one side of the figure-eight
 loop by itself, take two points from this component and connect them by a simple arc through the vertex of the 
 figure-eight loop. Then we produce a $\theta$-graph. The choice of the arc is unique up to isotopy.
 As for the remaining parts,  both the figure-eight loop and the $\theta$-graph cut the pairs of pants into three
 annuli. This shows the inner points of the edges of a Reeb graph are naturally in
 one-to-one correspondence with the inner points of the edges of the corresponding P-graph 
for the same surface since their preimages are both circles.

The P-graph move
between the first and second rows is induced  by an A-move, so we called it a \textit{P-graph A-move}\index{P-graph! A-move}. It is achieved by
removing the horizontal edge in the middle of the first row P-graph and its endpoints, then adding the vertical middle
edge with its endpoints to obtain the second P-graph. The other P-graph move between second and third rows is induced
by an S-move, so we call it a \textit{P-graph S-move}\index{P-graph! S-move}. It is achieved by removing the left one-endpoint edge
with its vertex, then adding a new one-endpoint edge with a new vertex.

\subsection{Local models of P-complexes}\label{sec3}
In Section \ref{morsereeb}, we showed a way to associate a Morse function to a pants decomposition of a surface, so we want to 
find a similar way to associate a Morse 2-function to a pants-block decomposition for a 3-manifold. 
The Reeb graphs and P-graphs connect Morse functions to pants decompositions, 
and in Section \ref{secmorse2}, we saw that a Reeb complex is locally a stack of Reeb graphs, so the Reeb complex could be a useful tool.
However, a Reeb complex may have cusp singularities (see Figure \ref{fig16}) and other singularities which are not what we want. 
So besides Reeb complexes, we still need to construct another complex, called \textit{P-complex},
based on P-graphs. Section \ref{singularities} gives a list of induced moves which we can use to turn a Reeb complex into a P-complex.
In order to explain things naturally, we first consider the cases of surface-cross-interval. Then we will introduce the general definition
of a P-complex.

Consider a surface-cross-interval $S\times I$ where $I=[0,1]$. If the surface $S\times \{t\}$ admits the same pants decomposition for all $t$, then
this surface-cross-interval contains only trivial pants blocks. Let's consider the nontrivial cases. Without loss of generality, we assume 
$S\times I$ is one of the two fundamental blocks defined in Definition \ref{defpantsblock}. There exists $0< t_*<1$ so that 
the surface $S\times \{t\}$ admits a pants decomposition $\mathcal{P}_0$ for $t\in [0, t_*)$ and the surface $S\times \{t\}$ admits a pants 
decomposition $\mathcal{P}_1$  for $t\in (t_*,1]$. By definition, $\mathcal{P}_0$ and $\mathcal{P}_1$ differ by a pants move. For each $t\neq t_*$,
$S\times \{t\}$ has a P-graph $P\times \{t\}$. 
For $t=t_*$, $S\times \{t_*\}$ has an \textit{almost P-graph} $P\times \{t_*\}$: $P\times \{t_*\}$ is the same as $P\times \{0\}$
or $P\times \{1\}$ elsewhere except replacing one edge with endpoints by a vertex. We will explain more about almost P-graphs and their
preimages in the next paragraph.
A \textit{stack} of P-graphs is a union of these P-graphs and an almost P-graph such that they are piled together with
respect to $t$, see Figure \ref{fig21} for a local point of view. We also say $P\times \{t\}$ is a \textit{level} of the stack.  In a stack of P-graphs, vertices of P-graphs form edges of the stack, 
edges of P-graphs are contained in the 2-cells of the stack. 
The \textit{vertices} of the stack are the intersections of two or four edges, contained in the almost P-graph. 
Given  two stacks of P-graphs so that the top level of one stack is the same as the bottom level of the other, we can glue these two
stacks together combinatorially along these levels, and still call the resulting object a stack of P-graphs.

\begin{figure}[ht]
\centering
\includegraphics[width=.2\textwidth,height=.2\textwidth]{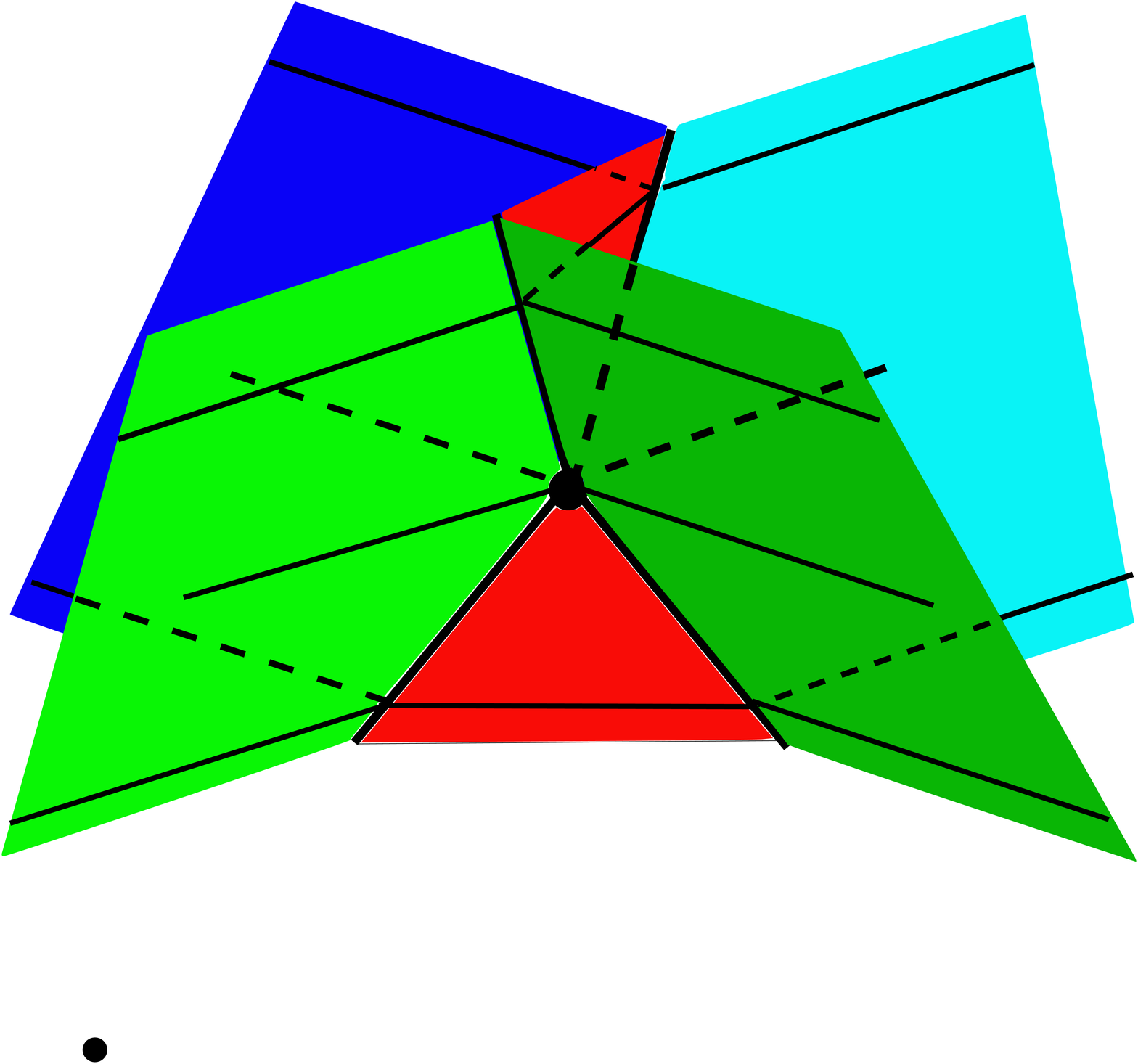}\hspace{2.5cm}
\includegraphics[width=.2\textwidth,height=.2\textwidth]{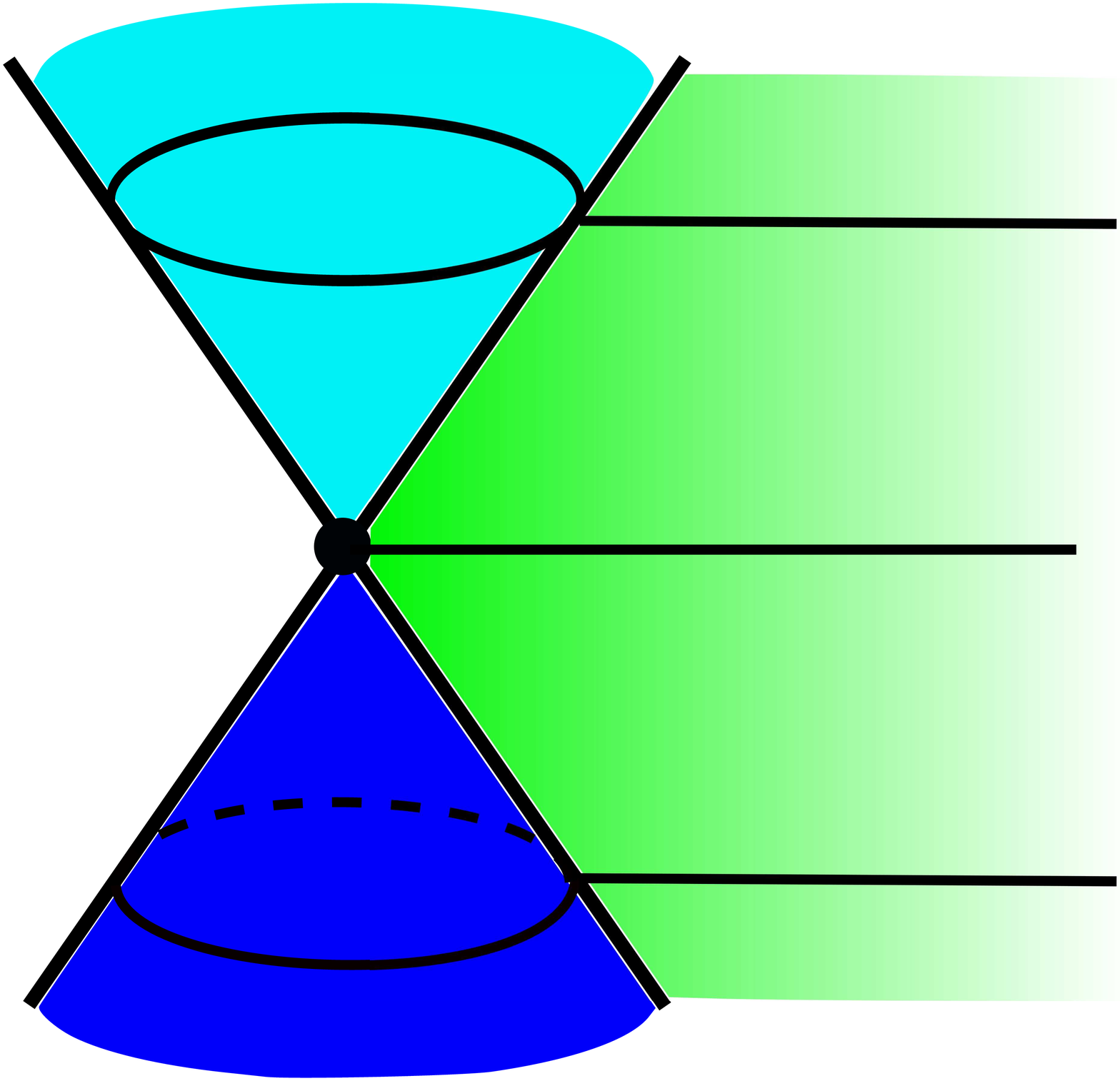}
\put(-200,-10){(a)}
\put(-60,-10){(b)}
\caption{Two stacks of P-graphs, which we will call local P-models below.}
\label{fig21}
\end{figure}

\begin{figure}[ht]
\centering
\includegraphics[width=.8\textwidth,height=.3\textwidth]{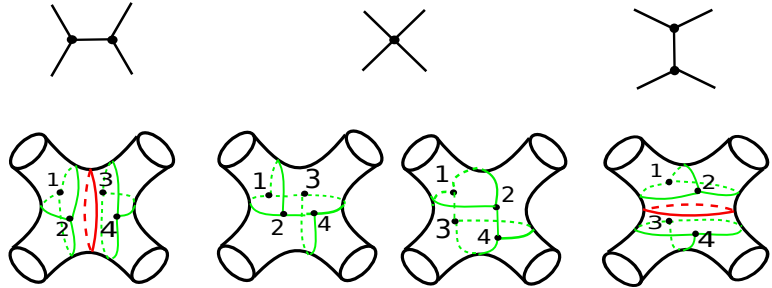}

\caption{An H-I move and the change of the preimages of vertices of P-graphs on the surface.}
\label{himove2}
\end{figure}

Each P-graph is induced by a pants decomposition of a surface, thus pants moves between pants decompositions induce
moves between P-graphs.  Since there are two types of pants moves, we have two types of  moves between P-graphs. 
Given two pants decompositions of a (0,4)-surface differ by an A-move, as the left and right subfigures in the bottom row 
of Figure \ref{himove2}, both have two $\theta$-graphs with labelled vertices, they induce two P-graphs as shown in the figure. 
Suppose the left P-graph in Figure \ref{himove2} is the bottom level as in Figure \ref{fig21}(a), and the right P-graph in Figure \ref{himove2}
is the top level as in Figure \ref{fig21}(a), then the preimage of the almost P-graph is one of the two subfigures in the middle.
These two subfigures are the same except changing the relative positions of vertices 2 and 3. One can imagine that as $t$ goes from 0 to 
$t_*$, the lower arc connecting vertices 1 and 2 and the upper arc connecting vertices 3 and 4 are moving towards each other, and they 
merge as shown in the figure when $t=t_*$. When $t>t_*$, these two arcs appear again and move away from each other, as in right subfigure.
Since the two P-praghs look like a letter H and a letter I, so we call this move between P-graphs an \textit{H-I move}.

As for the other type, given two pants decompositions of a (1,1)-surface differ by an S-move,  as the left and right subfigures in the second row 
of Figure \ref{hourglass2}, both have a $\theta$-graph, they induce two P-graphs as shown in the figure. Start from the left P-graph, 
the preimage of its vertex is the $\theta$-graph on the left pants decomposition.  In order to visualize the process of the move better, 
we again use a loop to represent the boundary of the water surface, similar to the idea in Figure \ref{crossmove2}. Leaning the punctured 
torus until second left figure, at this moment the preimage of the vertex becomes a graph $G_0$  with four vertices connected by seven arcs. At the
same time the water surface becomes the shadowed area, so its boundary overlaped with $G_0$.  This means that the points in the boudary of 
the shadowed area are also the preimages of the vertex in the almost P-graph. Furthermore, every point under the graph $G_0$ in the 
punctured torus is in some
level set which overlaped with $G_0$, thus is a preimage of the vertex in the almost P-graph. This explains why there is only one vertex 
in the almost P-graph. After this moment, continue leaning the punctured torus, the boundaries of the water surface break into two components
as described in Figure \ref{crossmove2}. This process resolves the vertices 3 and 4 in graph $G_0$ and turns the graph as in the second right figure.
The right figure is homotopic to the second right with one of the boundaries of water surface  is a loop in the pants decomposition.
This move between P-graphs,
induced by an S-move, is called an \textit{hourglass move}, since the shape of the stack looks like an hourglass.

\begin{figure}[ht]
\centering
\includegraphics[width=.8\textwidth,height=.3\textwidth]{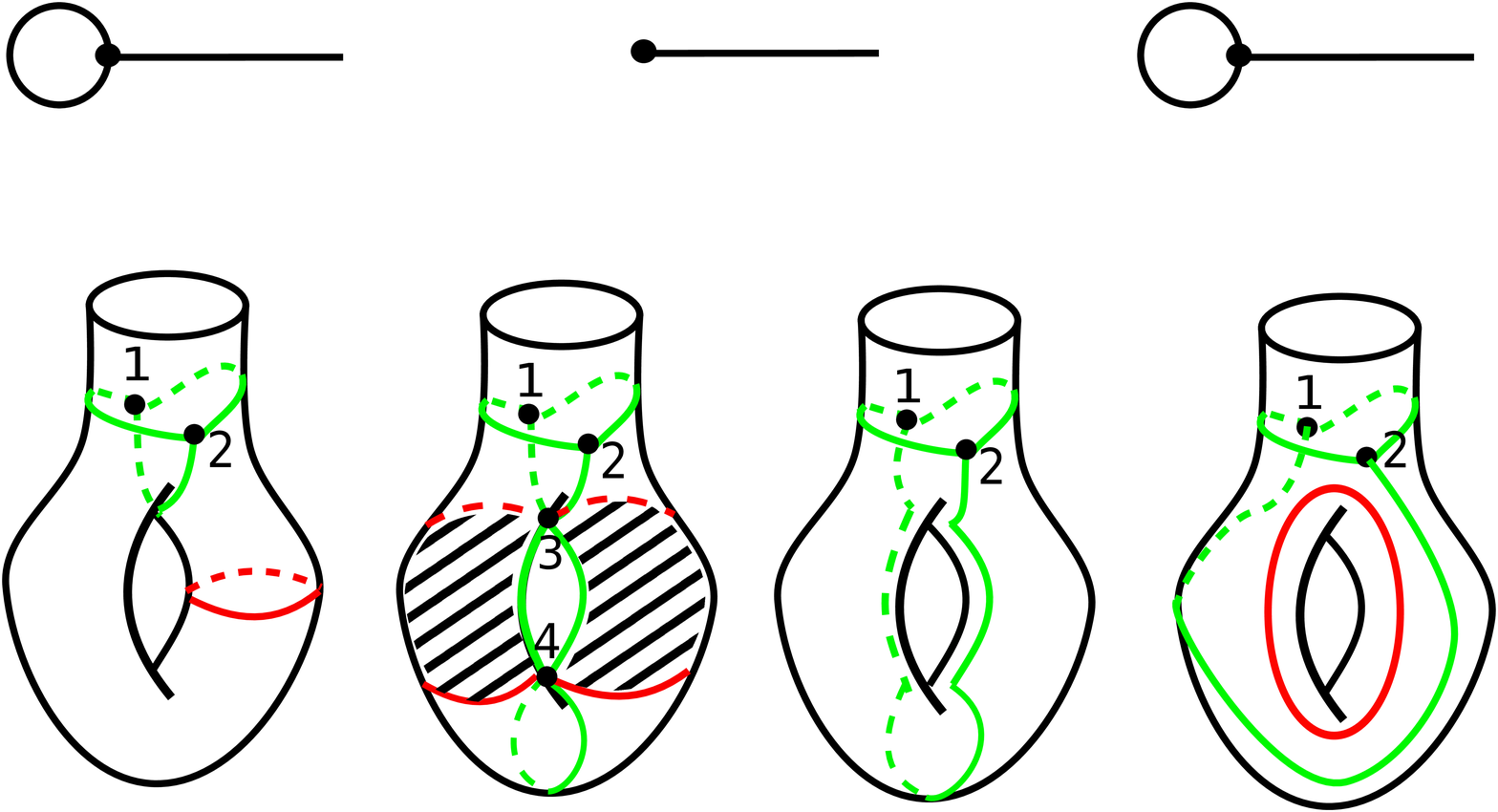}

\caption{An hourglass move and the change of the preimages of vertices of P-graphs on the surface.}
\label{hourglass2}
\end{figure}

\begin{Def}\label{def2}
 We say a stack of P-graphs is a \textit{local P-model} if

 (1) All but finitely many levels of the stack are P-graphs and

 (2) For those levels which are not P-graphs, they are almost P-graphs, i.e., they contain vertices of the stack.
\end{Def}

We say those levels which contain vertices of a local P-model are \textit{critical levels},
others are \textit{regular levels}. Each critical level contains only one vertex. We take small neighbourhoods of 
each critical levels so that they don't pairwise intersect.
As shown in Figure \ref{fig21} (a),  a critical level contains a 
valence-four vertex. This level separates its neighbourhood into two parts so that two regular
levels in different parts differ by an H-I move. As for Figure \ref{fig21}(b),
a critical level contains a valence-two vertex,  separates
its neighbourhood into two parts so that the regular levels in different parts differ by an hourglass move. Two
regular levels $P\times \{t_i\}$ and $P\times \{t_j\}$ are equivalent if they are in between two adjacent critical levels.

Using the two models in Definition \ref{def2}, we now want to discuss the local models that we will use below to define
a P-complex, i.e., the preimage in $M$ of each point in a local P-model under the
quotient map. There are three different types of points: the points in the interior of 2-cells, the points in the interior
of edges, and the vertices. We will call them type-I, type-II and type-III points, respectively.

For a type-I point $p_1$, let $V_1=U_1\times S^1$, where $U_1$ is an open disk. Let $f_1:V_1\to U_1$ be a projection map.  $p_1$  is an 
interior point of $U_1$, thus the preimage under $f$ of $p_1$ in $V_1$ is a circle. 

\begin{figure}[ht]
\centering
\includegraphics[width=.3\textwidth,height=.2\textwidth]{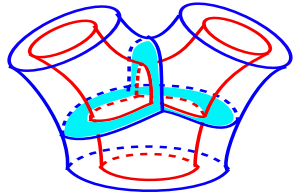}

\caption{Preimage of a neighbourhood of a type-II point. The thicken $\theta$-graph is the preimage of $U_2'$
in $M$.}
\label{fig22}
\end{figure}

For a type-II point $p_2$, let  $U_2$ be a valence-three graph$\times I$, and $V_2$ be a pair-of-pants$\times I$, as shown in  Figure \ref{fig22}.
Define a map $f_2:V_2\to U_2$ as follows: $p_2$ is an interior point in the edge of $U_2$ whose preimage under $f_2$ is a $\theta$-graph$\times \{i\}$
for some $i\in I$. The preimage under $f_2$ of an interior point of $U_2$ is an essential circle in the complement of $\theta$-graph$\times I$ 
in $V_2$. This matches the definition of type-I points and the projection map. Thus the interior points of $U_2$ are type-I points while the 
edge points are type-II points.

A type-III point $p_3$ is defined to be a vertex. A neighbourhood of a vertex in  a local P-model is a stack of P-graphs as in Definition \ref{def2}. There are two cases:
(1) If the vertex is of valence two, take a neighbourhood $U_3$ in the local P-model of $p_3$, as in Figure \ref{fig21}(b).
The lower boundary of $U_3$ in the figure is a part of a P-graph, whose preimage in $M$ is a part of a surface $S_0$ with a
pants decomposition $\mathcal{P}_{S_0}$. The upper boundary of $U_3$ in the  figure  is a part of a P-graph on a surface $S_1$ isotopic to $S_0$.
The preimage of the upper boundary of $U_3$ in $M$ is a subsurface of $S_1$ which is isotopic to the preimage of the upper part in $S_0$ but with a 
different pants decomposition $\mathcal{P}_{S_1}$. Moreover, $\mathcal{P}_{S_0}$ and $\mathcal{P}_{S_1}$ differ by one S-move.
Define a map $f_3$ from a type (1,1)-block to $U_3$, such that the preimage of $p_3$ is the middle figure in the second row of Figure \ref{hourglass2}.
The preimages of edge points and interior points in the 2-cells are $\theta$-graphs and circles,  which match the definition of $f_2$.
Thus the preimage $V_3$ of $U_3$ is a pants block of type (1,1). We call $p_3$ a III-S point. (2) Similarly, 
if $p_3$ is of valence four, the preimage of a neighbourhood of $p_3$, as in Figure \ref{fig21}(a),
is a pants block of type (0,4). We call it a III-A point.

We call $(f_i, V_i, U_i)$ a triple for a type-I, -II or -III point.

The following is the definition of a P-complex:

\begin{Def}\label{defpcomplex}
 Let $M$ be a closed, orientable 3-manifold.Consider a 2-dimensional CW complex $X$ and a map $F:M\to X$. 
We say $X$ is a \textit{P-complex} if for each point $p\in X$, there is a neighbourhood $U$ of $p$ in $X$, with a subset $V\subset M$ such that
$F(V)=U$, and a map $G=F|_V:V\to U$, such that $(G, V, U)$ is a triple of type-I, -II and -III points, up to a homeomorphism. 
We denote a P-complex by $\mathcal{PC}$.
\end{Def}

\begin{Rmk}
Note that this definition explicitly excludes local homeomorphism from $\mathcal{PC}$ to $\mathcal{RC}$ which maps a neighbourhood of a vertex 
in $\mathcal{PC}$ to a local model descibed in Figures \ref{fig2cancell}, \ref{fig17}, \ref{fig18}, \ref{fig19}, \ref{fig20}. 
\end{Rmk}

The following two lemmas reveal the relation between pants-block decompositions and P-complexes of a 3-manifold. Lemma \ref{lem2} also shows the
existence of P-complexes.

\begin{Lem}\label{lem2}
 Every pants-block decomposition of a 3-manifold $M$ defines a unique P-complex.
\end{Lem}
\begin{proof}
 Let $(L,\mathcal{P}_L,\mathcal{B})$ be a pants-block decomposition of $M$. Let $B_1\in \mathcal{B}$ be a fundamental block.  
 Let $S_1$ and $S_1'$ be the top and bottom surfaces of $B_1$. By definition, the pants decompositions on $S_1$ and $S_1'$ differ
 by a pants move, and each of them defines a P-graph. The two P-graphs differ by a P-graph move.
 By the discussion of local behaviors of P-complex above, $B_1$ is the preimage of a neighbourhood $U_1$ of a type-III point, so that
 the upper boundary of $U_1$ is one P-graph and the lower boundary of $U_1$ is the other P-graph. If $\mathcal{B}$ contains only one block,
 then $U_1$ is the P-complex. If $B_1$ is glued to itself in $M$ along $S_1$ and $S_1'$, then glue $U_1$ to itself along the upper and lower
 boundaries correspondently.
 Assume $\mathcal{B}$ contains more than one block. Let $B_2\in \mathcal{B}$ be a fundamental block that is adjacent to $B_1$ in the pants-block decomposition.
 Let $S_2$ and $S_2'$ be the top and bottom surfaces of $B_2$. Similarly, $B_2$ is the preimage of a neighbourhood $U_2$ of another
 type-III point. Since $B_2$ is adjacent to $B_1$, without loss of generality, we say
 $S_1'\cap S_2$ is either one or two pairs of pants. Let $W$ be the result of gluing $U_1$ to $U_2$ along their boundaries corresponding
 to $S_1'\cap S_2$, denoted by $W=U_1\cup U_2$. $W$ is a P-complex since both $U_1$ and $U_2$ are P-complexes.
 Continue this process, letting $U_i$ be the neighbourhood of a type-III point that
 is uniquely determined by a fundamental block $B_i\in \mathcal{B}$, we have 
 $$U=\displaystyle\bigcup_{B_i\in \mathcal{B}} U_i.$$ Thus the P-complex $U$ is uniquely determined by the pants-block decomposition $(L,\mathcal{P}_L,
 \mathcal{B})$.
\end{proof}

\begin{Lem}\label{lempcomplex3}
 Every P-complex of a 3-manifold $M$  defines a unique pants-block decomposition for $M$.
\end{Lem}
\begin{proof}
 Let $\mathcal{PC}$ be a P-complex of $M$. Choose an interior point from each 2-cell in $\mathcal{PC}$, and denote this collection by $V$.
 All points in $V$ are type-I points in $\mathcal{PC}$ as classified above. The preimages in $M$ of $v\in V$ are loops which form a link
 $L$. Consider an index-three edge $E$ in $\mathcal{PC}$. Let $U_1$, $U_2$ and $U_3$ be the three 2-cells in $\mathcal{PC}$ adjacent to $E$.
 Let $v_1,v_2,v_3\in V$ such that $v_i\in U_i$. Let $e\in E$ be a type-II point. Connect $v_i$ to $e$ by an arc in $U_i$. This forms a 
 Y-shape. Note that the preimage in $M$ of $e$ is a $\theta$-graph, so the  preimage of this Y-shape is a  pair of pants such that 
 its boundary components are in $L$. As we continue this process, we obtain a collection $\mathcal{P}_L$ of pairs of pants that cuts the manifold
 $M$ into  a collection $\mathcal{B}$ of pants blocks. This is true because of the discussion of type-III points. Thus $(L,\mathcal{P}_L,\mathcal{B})$
 is a pants-block decomposition of $M$.  As mentioned above,  the
 preimages of two type-I points in the same 2-cell of $\mathcal{PC}$ cobound an annulus in $M$, thus the choice of the interior point
 in $V$ for this 2-cell is unique up to homotopy. Therefore the choices of the link and pairs of pants are unique up to homotopy.
\end{proof}

In the statement of the main theorem, we require the manifold to be hyperbolic. This is because we need our P-complex
to be nice enough, i.e., every 2-cell in a P-complex is a disk, see Lemma \ref{lemdisk}. By this lemma, we can 
guarantee that there is no valence-two edges in the P-complex since the 2-cells like annuli or even higher genus 
doens't exist. See Figure \ref{v2edge} for an illustration. Before proving Lemma \ref{lemdisk}, we need the following discussions.

 \begin{figure}[ht]
\centering
\includegraphics[width=.3\textwidth,height=.15 \textwidth]{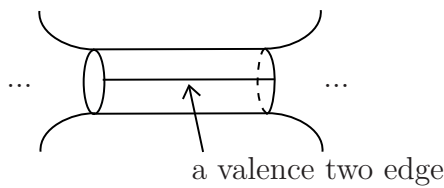}
\put(-50,-10){a valence two edge}
\put(-120, 25){...}
\put(0, 25){...}
\caption{A valence-two edge in an annulus. Both sides of the annulus are not disks.
But this cannot happen in the P-complexes.}
\label{v2edge}
\end{figure}

\begin{Lem}\label{lemcomsep}
 Let $M$ be a compact, closed, hyperbolic 3-manifold. Then every embedded torus in $M$ is both compressible and separating.
\end{Lem}

\begin{proof}
 Every hyperbolic 3-manifold is irreducible and atoroidal. A compact irreducible 3-manifold is atoroidal means 
 every incompressible torus in $M$ is parallel to a component of $\partial M$. The assumption $M$ is closed means 
 $\partial M=\emptyset$, thus there is no incompressible torus in $M$. So every embedded torus in $M$ is compressible. 
 
 Suppose we have a non-separating, compressible torus in $M$. Compress this torus along a meridian curve, then we get a
 nonseparating sphere $S_0$ in $M$. This implies $S_0$ doesn't bound a ball in $M$ from both sides. But $M$ is irreducible
 implies every sphere in $M$ bounds a ball. A contradiction. Thus every compressible torus in $M$ is separating. 
\end{proof}

 A compressible separating torus $T$ in a compact, closed, hyperbolic 3-manifold bounds either a solid torus
 or a knot complement on one side. The case that $T$ bounds a solid torus on one side can be understood 
 straight forward. Suppose $T$ bounds a knot complement on one side. The construction of the knot complement and
 this $T$ is as follows. 
 Let $S$ be an embedded separating sphere in $M$ which bounds a ball $B$ on one side. Let $\alpha$ be 
 a knotted arc embedded in $B$ and $N(\alpha)$ is a tubular neighbourhood of $\alpha$ embedded in $B$. Let 
 $M_1=B\backslash N(\alpha)$ and $M_2=(M\backslash B)\cup N(\alpha)$. Then $\partial M_1=\partial M_2=T$ and $M_1$ is 
 the knot complement. Remove $M_1$, unknotted the arc $\alpha$ and its neighbourhood $N(\alpha)$ and denote the resulting
 manifold by $M_2'$. $M_2'$ is homeomorphic to $M_2$.
 Glue a solid torus back to $M$ in a trivial way, then we get a manifold $M'$ which is homeomorphic to $M$. 
 This implies $M_1$ is homeomorphic to a solid torus.
 Thus the image of $M_1$ in the P-complex is homeomorphic to the image of a solid torus in the P-complex.

\begin{Lem}\label{lemdisk}
 Every 2-cell in a P-complex of a compact, closed, hyperbolic 3-manifold $M$ is a disk. 
\end{Lem}

\begin{proof}
 Take an arbitrary simple closed curve $c$ in an arbitrary 2-cell of a P-complex $\mathcal{PC}$ of $M$. The preimage of $c$ in $M$ is a torus $T$ since 
 the preimage in $M$ of each point in $c$ is an $S^1$. By Lemma \ref{lemcomsep}, $T$ is compressible and separating. 
 This implies $c$ is a separating curve in the P-complex. Cut $M$ along $T$, we get $M_1$ and $M_2$. $c$ separates
 the P-complex into two sub P-complexes $\mathcal{PC}_1$ and $\mathcal{PC}_2$, which are the P-complexes of $M_1$
 and $M_2$ respectively. By the arguments in the previous paragraph,
 $T$ is compressible implies that it bounds a solid torus or a knot complement from one side. If $M_1$ is a solid torus, 
 remove $M_1$ from $M$ and glue back a trivial solid torus, we get a manifold which is homeomorphic to $M$.
 This corresponds to cut $\mathcal{PC}_1$ from $\mathcal{PC}$ along $c$ and glue a trivial disk to $\mathcal{PC}_2$.
 If $M_1$ is a knot complement as constructed above, remove $M_1$ and glue back a trivial torus, we still get a manifold which is 
 homeomorphic to $M$. This also implies we can remove $\mathcal{PC}_1$ from $\mathcal{PC}$ and glue a disk back. 
 Therefore $c$ bounds a disk in both cases. Hence every 2-cell in a P-complex of $M$ is a disk.
\end{proof}

\subsection{P-complexes and Morse 2-functions}\label{secpm}

In this subsection we explain how to associate a Morse 2-function to a P-complex for a 3-manifold $M$, which we will 
use in the proof of the main theorem. 

\begin{Lem}\label{lempcomplex}
 Let $M$ be a closed, connected, orientable 3-manifold. Every Morse 2-function $F:M\to \mathbb{R}^2$ defines a P-complex.
\end{Lem}

\begin{proof}
 By Definition \ref{rc}, each Morse 2-function $F:M\to \mathbb{R}^2$ defines a Reeb complex $\mathcal{RC}$. Since $M$ is closed, all 
 index-one edges of $\mathcal{RC}$ don't contain boundary vertices. Therefore all index-one edges contain only extremal vertices. 
 This means they are  
 in one of the local models described in some of the induced moves in Section \ref{singularities}. Let $C$ be the collection of all 
 2-cells in $\mathcal{RC}$. Choose one interior point from each 2-cell in $C$ and denote the collection by $V$. Let $C_1\subset C$
 be a collection of all 2-cells which contain index-one edges in their boundaries. Let $V_1\subset V$ be the collection of points in 2-cells in $C_1$.
 Let $C_2$ be a collection of all 2-cells adjacent to those 2-cells in $C_1$ and let $V_2\subset V$ be the corresponding subset. For 
 each point in $V$, its preimage in $M$ is a loop. All these loops form a link $L$. Use the induced moves in Section \ref{singularities}
 to simplify $\mathcal{RC}$ such that all index-one edges and their adjacent 2-cells (in $C_1$) are gone. 
 These moves delete the singularities descibed in Section \ref{singularities}, and also delete those 
 index-three edges which cobound these deleted 2-cells together with index-one edges, analogous to the process descibed in Section 
 \ref{morsereeb}  which turns a Reeb graph into a P-graph. These moves change $V$ and $L$ by the following steps: (1) First take out all
 points in $V_1$ from $V$, and correspondingly take out the preimages from $L$ of points in $V_1$. Denote the resulting sets by 
 $V'=V\backslash V_1$ and $L'$. (2) For two 2-cells in $C_2$ which adjacent to the same 2-cell in $C_1$ and the same index-three edge (deleted
 by the moves), say $v_a\in V_2$ and $v_b\in V_2$ are two corresponding points. Take one of them, say $v_b$, out from $V'$ and keep the other. 
 Delete the preimage in $L'$ of $v_b$. Do these to all pairs of adjacent 2-cells in $C_2$ and denote the resulting sets by $V''$ and $L''$.
 These induced moves turn $\mathcal{RC}$ into a 2-dimensional complex $X$. For any interior point in $V''$ of $X$, its preimage
 is a loop in $L''$, thus it is a type-I point. For any edge point $e$ in $X$, note that there is no index-one edge in $X$, thus
 $e$ must be a valence-three point, its preimage in $M$ is a $\theta$-graph. Therefore $e$ is a type-II point. For any vertex point $v\in X$,
 note that the induced moves erased those singularities descibed in Section \ref{singularities}, so the left vertex points are just
 type-III points. Therefore, $X$ is a P-complex.
\end{proof}

\begin{Coro}\label{coro2}
 Let $M$ be a closed, connected, orientable 3-manifold. Every Morse 2-function defines a unique pants-block decomposition of $M$.
\end{Coro}

\begin{proof}
By Lemma \ref{lempcomplex}, a Morse 2-function defines a P-complex. By Lemma \ref{lempcomplex3}, every P-complex defines a unique
pants-block decomposition, thus this corollary is true.
\end{proof}

Given a Morse 2-function $F$, Lemma \ref{lempcomplex} shows that we can construct a P-complex from $F$.
However, this correspondence is not one-to-one. That is, we may have the same P-complex from different Morse 2-functions. For example,
two Morse 2-functions $F_0$ and $F_1$ that are related by a generic homotopy such that their Reeb complexes $\mathcal{RC}_0$ and $\mathcal{RC}_1$
differ by a single induced move 4.

Conversely, given a P-complex $\mathcal{PC}$, we want to know how to obtain a Morse 2-function related to this P-complex. 
We have the following lemma:
\begin{Lem}\label{lempcomplex2}
 Let $M$ be a closed, connected, orientable 3-manifold. For every P-complex $\mathcal{PC}$ of $M$, we can construct a Morse 
 2-function that induces this P-complex.
\end{Lem}

\begin{proof}
 We  first need to define a map $g:\mathcal{PC}\to \mathbb{R}^2$. For any vertex $V\in\mathcal{PC}$, i.e., the intersection of
  edges, $g(V)$ is a vertex in $\mathbb{R}^2$. For any edge $E\in \mathcal{PC}$ connecting two vertices $V_1$ and $V_2$,
 $g(E)$ is a non self-intersected edge in $\mathbb{R}^2$ connecting $g(V_1)$ and $g(V_2)$. However, the preimage under $g$ of an edge
 in $\mathbb{R}^2$ is not necessary to be an edge in $\mathcal{PC}$:  (1)Choose an interior point $V'$ in an edge in $\mathbb{R}^2$, its preimage
 has two possibilities: (1a) If $U_{V_1'}$ is a disk  neighbourhood of $V_1'$, as in the second row of Figure \ref{diskimage}($a$), then the 
 preimage of $U_{V_1'}$ under $g$ is a neighbourhood of an interior point $V_1$ in a valence-three edge $E_1$ of $\mathcal{PC}$,  as in the first row.
 (1b) If $U_{V_2'}$ is a half disk ``neighbourhood'' of $V_2'$ with an edge $E_2'$ containing $V_2'$ , as in the second row of Figure \ref{diskimage}(b), 
 then the preimage  of $V_2'$ could be an interior point $V_2$ of a 2-cell with a disk neighbourhood $U_{V_2}$, such that there is an arc $\alpha$ 
 in $U_{V_2}$ whose image $g(\alpha)$ is homeomorphic to $E_2'$. (2) For a vertex $V_3'\in \mathbb{R}^2$ which is the common end of 
 two edges $E_3'$ and $E_4'$, like a cusp point in Figure \ref{diskimage}(c), take a neighbourhood $U_{V_3'}$, the preimage in $\mathcal{PC}$
 of $U_{V_3'}$ is a neighbourhood of $V_3$,  with two arcs $\beta$ and $\gamma$ whose images $g(\beta)$ and $g(\gamma)$ homeomorphic
 to $E_3'$ and $E_4'$.
 Note that Lemma \ref{lemdisk} shows that 
 every 2-cell in a $\mathcal{PC}$ is  a disk. For any interior point $P$ in a 2-cell in $\mathcal{PC}$, there is a disk neighbourhood
 $U$ of $P$ such that $g(U)$ is a disk, but sometimes the image $g(P)$ could be on the boudary of $g(U)$, like Figure \ref{diskimage}(b).
 For any point $p\in\mathbb{R}^2$, there is a disk neighbourhood of $p$ whose preimage is either a disk in $\mathcal{PC}$, or
 one of the three figures in Figure \ref{diskimage}, up to a homeomorphism.
 \begin{figure}[ht]
\centering
\includegraphics[width=.6\textwidth,height=.2\textwidth]{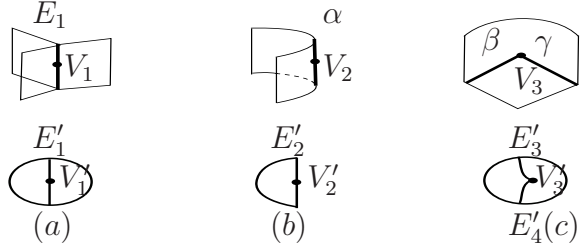}
\put(-207,-10){$(a)$}
\put(-117,-10){$(b)$}
\put(-14,-10){$(c)$}
\put(-207,23){$E_1'$}
\put(-207,70){$E_1$}
\put(-117,23){$E_2'$}
\put(-27,23){$E_3'$}
\put(-27,-10){$E_4'$}
\put(-97,70){$\alpha$}
\put(-37,60){$\beta$}
\put(-17,60){$\gamma$}
\put(-197,8){$V_1'$}
\put(-103,8){$V_2'$}
\put(-17,8){$V_3'$}
\put(-195,50){$V_1$}
\put(-97,50){$V_2$}
\put(-25,45){$V_3$}
\caption{Preimages of disks in $\mathbb{R}^2$ under $g$.}
\label{diskimage}
\end{figure}

We call those edges in Figure \ref{diskimage}(b) and (c) ``virtual'' edges, for the reason that they are not the images under $g$ of edges
in $\mathcal{PC}$. Let $\mathcal{E}$ denote the set of all ``virtual'' edges, and $\mathcal{PE}$ denote the set of preimages
under $g$ of elements in $\mathcal{E}$. A vertex in $\mathbb{R}^2$, the intersection of two edges, is called a ``virtual'' vertex, if at
least one of these two edges is ``virtual''. The intersections of two edges have two possibilities: one is shown in Figure \ref{diskimage}(c),
the  other case is the crossing of two edges. Let $\mathcal{V}$ denote the set of all ``virtual'' vertices, and $\mathcal{PV}$ denote 
the set of preimages under $g$ of elements in $\mathcal{V}$. 

 \begin{figure}[ht]
\centering
\includegraphics[width=.5\textwidth,height=.3\textwidth]{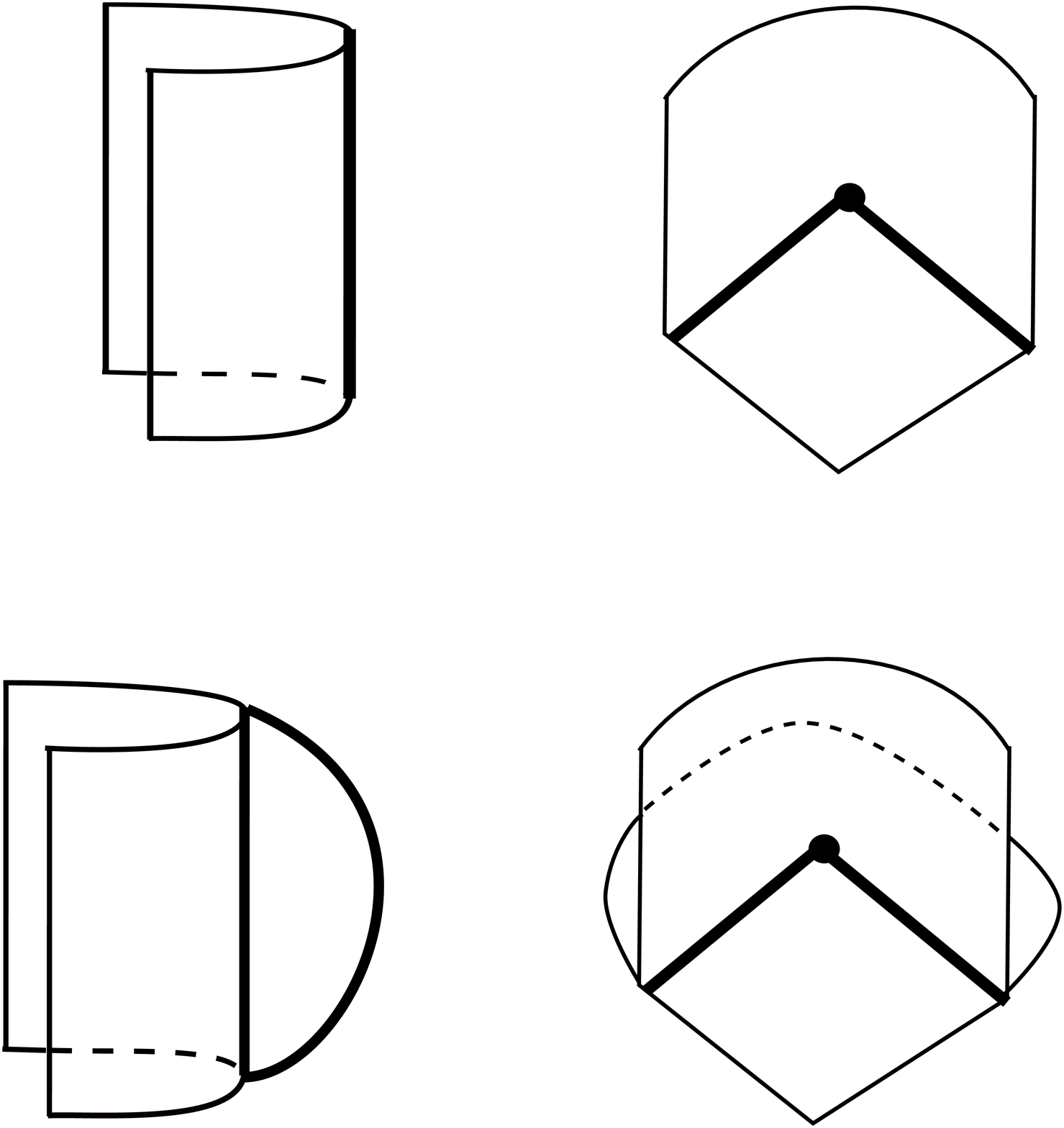}
\put(-117,83){$E_2'$}
\put(-67,85){$E_3'$}
\put(-27,85){$E_4'$}
\put(-137,23){$\tau_1$}
\put(-115,23){$\tau_2$}
\put(-67,23){$\tau_3$}
\put(-27,23){$\tau_4$}
\put(-5,23){$\tau_5$}
\caption{Local models for the construction.}
\label{diskimage2}
\end{figure}

We need to construct a new complex based on $\mathcal{PC}$ and $g$. We will use the same notations as in Figure \ref{diskimage}. Let $E_2'$ be an
element in $\mathcal{E}$ as in (b) and $\alpha$ be the preimage under $g$ of $E_2'$ in $\mathcal{PC}$. Let $D$ be a disk  bounded by two 
arcs $\tau_1$ and $\tau_2$. $\tau_1$ and $\tau_2$ intersect at two endpoints. Glue $\tau_1$ to $\alpha$ with endpoints identified, as in 
Figure \ref{diskimage2}. Do this to all elements in $\mathcal{E}$ whose preimages are $\alpha$. 
Let $E_3'$ and $E_4'$ be elements in $\mathcal{E}$ as in (c), $\beta$ and $\gamma$ be the preimages under $g$ of $E_3'$ and $E_4'$, respectively.
Let $D'$ be  a disk bounded by three arcs $\tau_3$, $\tau_4$ and $\tau_5$. Glue $\tau_3$ to $\beta$, $\tau_4$ to $\gamma$ with endpoints
identified, as in Figure \ref{diskimage2}.
Do this to all elements in $\mathcal{E}$ whose preimages are $\beta$ and $\gamma$. Denote the resulting complex by $\mathcal{PC}'$.

Let $F_*:M\to \mathbb{R}^2$ be a Morse 2-function such that its induced Reeb complex $\mathcal{RC}_*$ is homeomorphic to $\mathcal{PC}'$. Thus
the images of critical points under $F_*$ are homeomorphic to the images in $\mathbb{R}^2$ under $g$. We can use the induced moves to turn 
$\mathcal{RC}_*$ into $\mathcal{PC}$ since the actions of gluing disks described above are contained in the induced moves. Therefore 
$F_*$ is a Morse 2-function that induces $\mathcal{PC}$.

\end{proof}

Given a Morse 2-function $F$, we can construct a P-complex $\mathcal{PC}$ by Lemma \ref{lempcomplex}. 
We can also find a Morse 2-function $F_*$ which induces $\mathcal{PC}$ by Lemma \ref{lempcomplex2}. But $F$ and $F_*$ may not be the same, 
see the discussion after Corollary \ref{coro2}.

\subsection{P-moves between P-complexes}\label{secpmove}
Corollary \ref{coro1} gives the existence of pants-block decompositions for compact, closed, connected, orientable 3-manifolds, 
and Lemma \ref{lem2}
tells us every pants-block decomposition defines a unique P-complex, therefore the path moves between paths in pants complexes
induce P-moves between P-complexes.

By Definition \ref{hls2} and the discussion in Section \ref{sec1}, we can use HLS relations to describe path moves between
paths in the pants complexes, therefore HLS relations induce P-moves between pants-block decompositions for a 3-manifold, and thus 
induce P-moves  between P-complexes.
The goal of this subsection is to give an example of P-moves between pants-block decompositions and the corresponding P-moves between P-complexes
in detail.

\begin{figure}[ht]
\centering
\includegraphics[width=.7\textwidth,height=.3\textwidth]{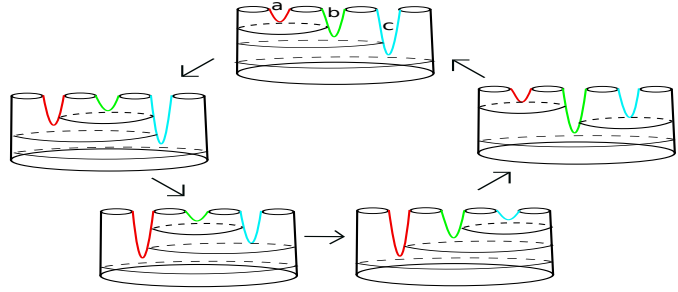}

\caption{A redrawn picture of Figure \ref{fig7}.}
\label{fig27}
\end{figure}

The A-pentagon relation induces two types of path moves, $a\leftrightarrow a^4$ and $a^2\leftrightarrow a^3$. In Section \ref{sec1} we
show that these two path moves are conjugate in the P-move list. Thus we take $a^2\leftrightarrow a^3$ as an example. Figure \ref{fig27}
is a redrawn picture of Figure \ref{fig7} which indicates the heights of different saddles under the height function.

There are two edge paths in the pants complex corresponding to Figure \ref{fig27}. Figure \ref{fig24} is another redrawn picture of 
Figure \ref{fig7} which gives us an intuitive image of pants blocks based on $S_{0,5}$. One edge path in Figure \ref{fig27} contains 3
edges, which corresponds to a collection of three pants blocks on the left of Figure \ref{fig24}. The other edge path in Figure \ref{fig27}
contains two edges, which corresponds to a collection of two pants blocks on the right of Figure \ref{fig24}. Assume that all of the saddle
points in Figure \ref{fig27} are essential saddle points. Figure \ref{r3} shows the Cerf graphics of generic homotopies between Morse 
functions (which are drawn as height functions). Thus the edges in Figure \ref{r3} are indefinite fold edges. Note that each Morse 2-function
is determined by a generic homotopy of Morse functions, thus each Cerf graphic gives us a Morse 2-function which defines a pants-block 
decomposition of the surface-cross-interval in Figure \ref{fig24}.

\begin{figure}[ht]
\centering
\includegraphics[width=.6\textwidth,height=.3\textwidth]{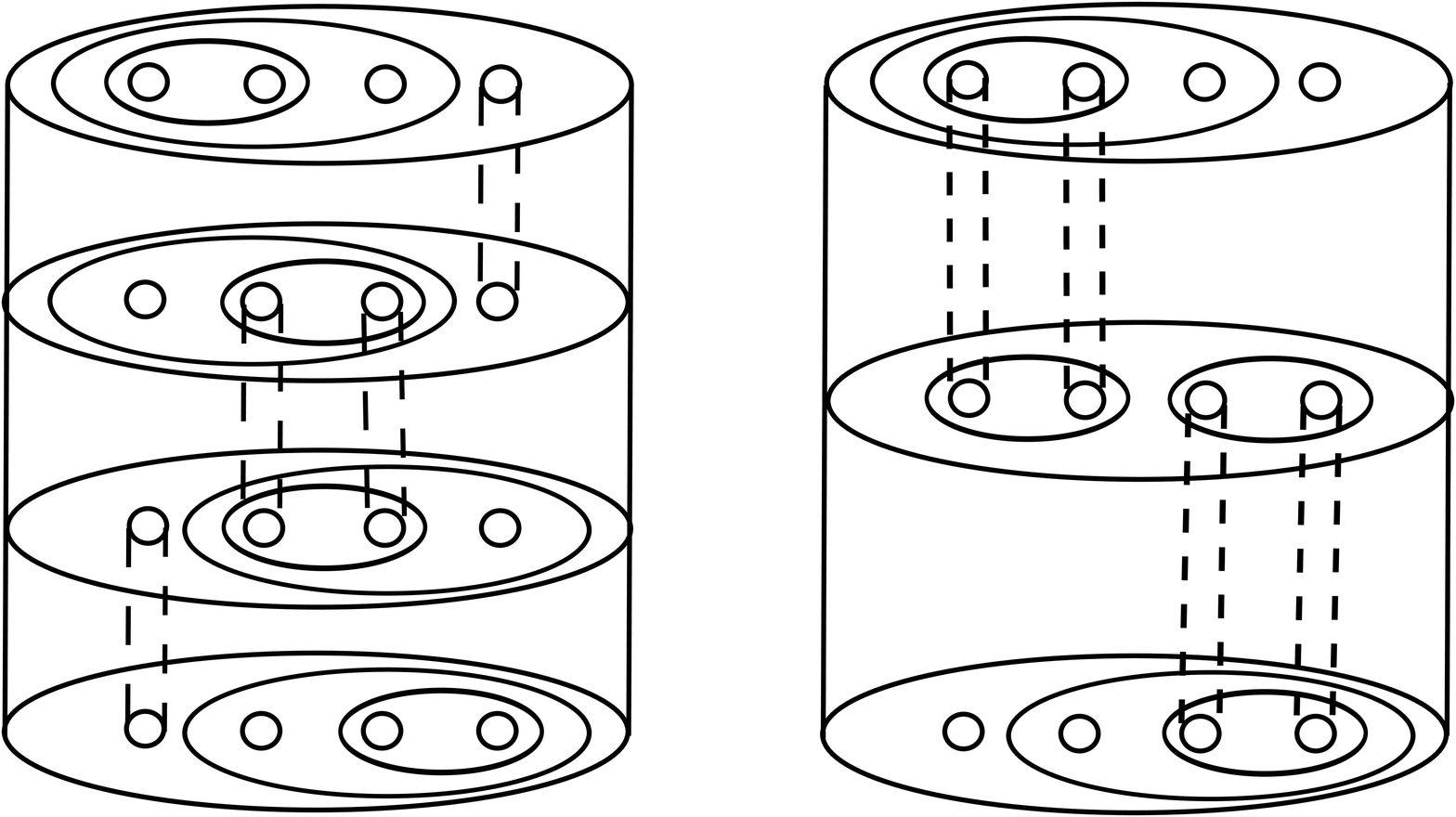}

\caption{P-move: $a^3\leftrightarrow a^2$. The vertical annuli indicate the unchanged boundary components in a pants block based on 
$S_{0,5}$.}
\label{fig24}
\end{figure}

The first crossing between $t=0$ and $t=t_1$ in Figure \ref{r3} implies the heights of saddle points $a$ and $b$ interchange (see Figure
\ref{fig27}, from the top surface to the left). This corresponds to an A-move, which defines the top pants block on the left side of 
Figure \ref{fig24}. Other crossings also define pants blocks respectively except the one in between $t=s_1$ and $t=s_2$. This crossing
doesn't define any pants block because it doesn't define any pants move. The reason that it doesn't  define a pants move is as follows:
At $t=s_1$, the corresponding surface is the right surface on Figure \ref{fig27} with saddle point $a$ higher than saddle point $c$.
At $t=s_2$, the corresponding surface is the same surface but with saddle point $a$ lower than saddle point $c$. This interchange of
heights doesn't define a pants move because the  pants decompositions on the surface are isotopic.
\begin{figure}[ht]
\centering
\includegraphics[width=.8\textwidth,height=.2\textwidth]{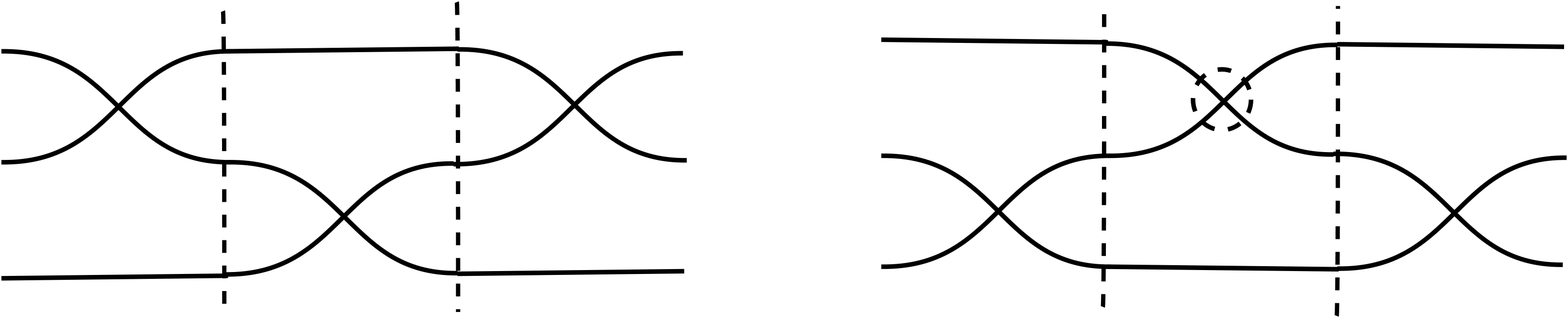}
\put(-300,60){$a$}
\put(-300,30){$b$}
\put(-300,5){$c$}
\put(-160,5){$a$}
\put(-160,30){$b$}
\put(-160,60){$c$}
\put(-135,60){$a$}
\put(-135,30){$b$}
\put(-135,5){$c$}
\put(0,5){$a$}
\put(0,30){$b$}
\put(0,60){$c$}
\put(-305,-10){$t=0$}
\put(-265,-10){$t=t_1$}
\put(-220,-10){$t=t_2$}
\put(-175,-10){$t=1$}
\put(-140,-10){$t=0$}
\put(-95,-10){$t=s_1$}
\put(-50,-10){$t=s_2$}
\put(-5,-10){$t=1$}
\caption{Cerf graphics for the two generic homotopies between Morse functions corresponding to the edge paths in Figure \ref{fig27}.}
\label{r3}
\end{figure}

Note that there is a generic homotopy between the two Morse 2-functions given by the two Cerf graphics, and there is a Reidemeister-III
type singularity in this generic homotopy. This generic homotopy corresponds to the P-move $a^3\leftrightarrow a^2$ which replaces three pants blocks by two pants blocks
in Figure \ref{fig24}.

\begin{figure}[ht]
\centering
\includegraphics[width=.4\textwidth,height=.4\textwidth]{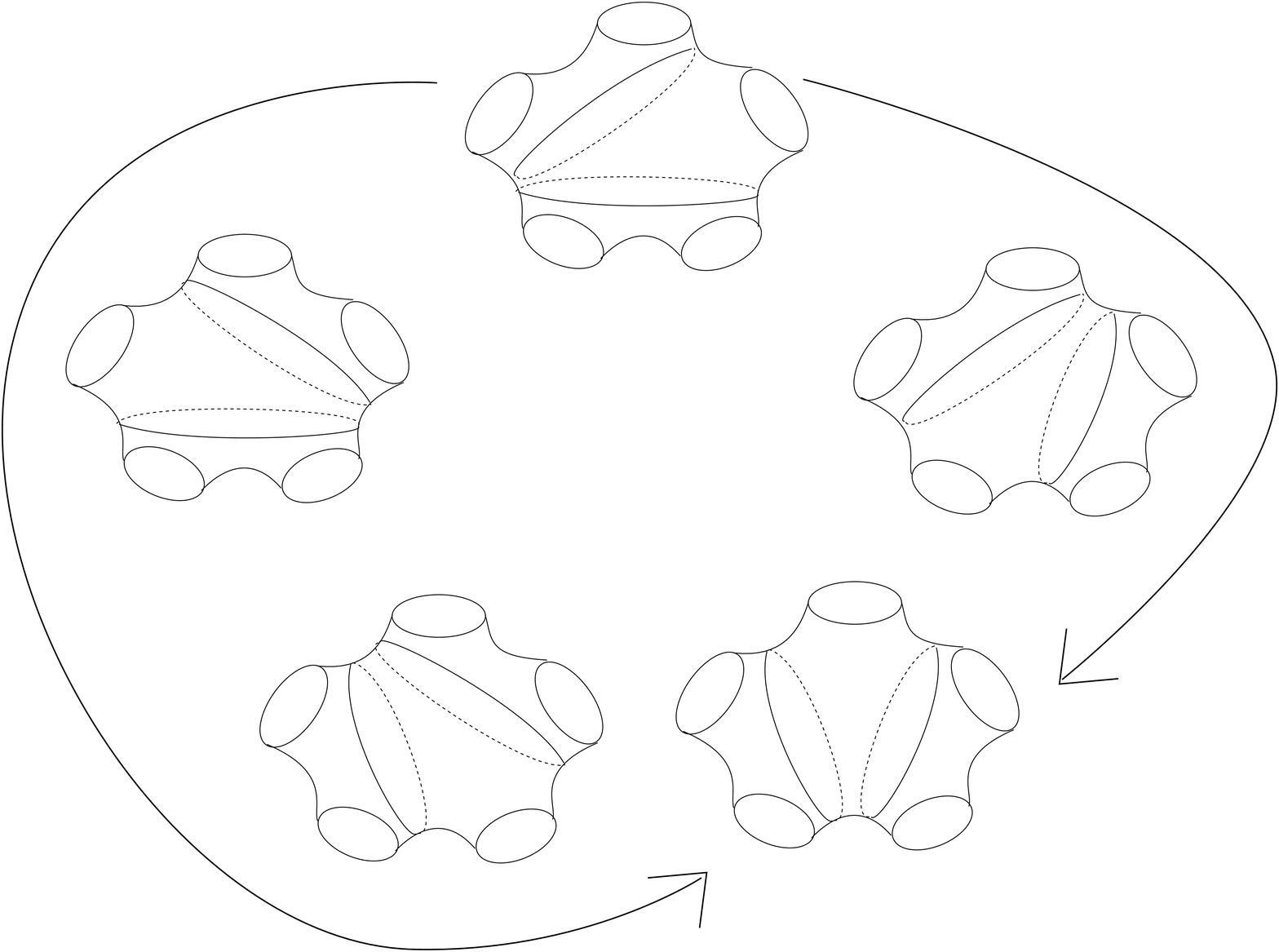}\hspace{.3cm}
\includegraphics[width=.4\textwidth,height=.4\textwidth]{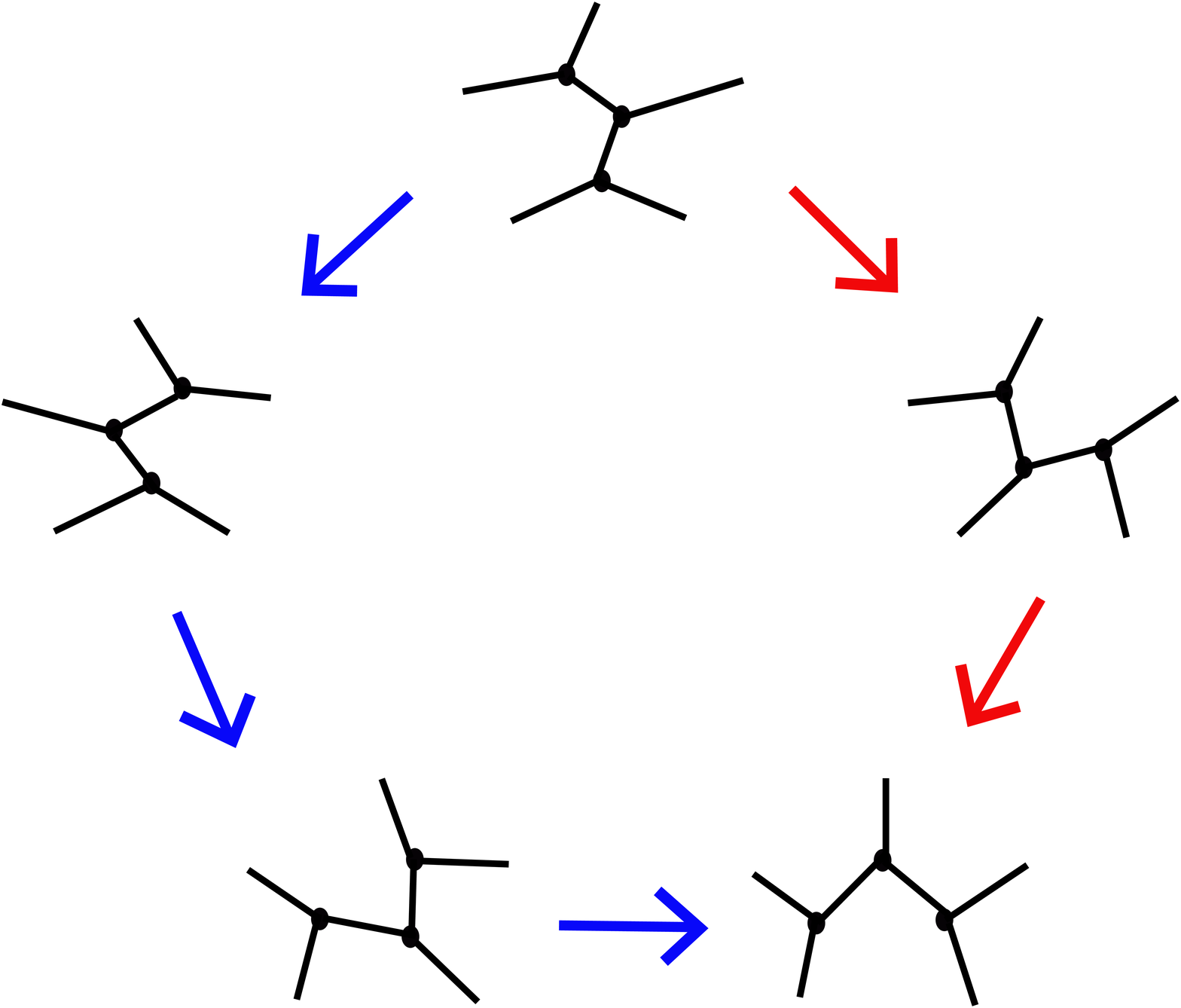}
\caption{The P-graphs of surfaces with pants decompositions in Figure \ref{fig7}.}
\label{fig23}
\end{figure}

As for the  corresponding P-move between P-complexes, we would like to start from the P-graphs. Figure \ref{fig23} shows the corresponding
P-graphs of the surfaces with pants decompositions in Figure \ref{fig7}, and indicates the two paths in Figure \ref{fig27}.
Each pair of adjacent P-graphs in Figure \ref{fig23} differ by an H-I move. This means they are in two regular levels of a local P-model
which are separated by a valence-four vertex. In other words, there is a local model as in Figure \ref{fig21}(a) in between 
each pair of adjacent P-graphs in Figure \ref{fig23}.

There are two different local P-models starting from the top P-graph and ending at the bottom right P-graph in Figure \ref{fig23}.
One corresponds to the left path of Figure \ref{fig23} which contains three valence-four vertices in between the P-graphs. 
The other local P-model corresponds to the right path of Figure \ref{fig23} which contains two valence-four vertices in between P-graphs.
Assume these two local P-models are in two P-complexes $\mathcal{PC}_0$ and $\mathcal{PC}_1$ respectively, such
that the remaining parts in $\mathcal{PC}_0$ and $\mathcal{PC}_1$ are isotopic to each other. We can construct a Morse 2-function $F_i$
which induces $\mathcal{PC}_i$ by Lemma \ref{lempcomplex2}. The Cerf graphics of $F_i$ are in Figure \ref{r3}.
There is a generic homotopy between $F_0$ and $F_1$ such that there is only one Reidemeister-III type singularity. This generic homotopy
realizes a P-move between the P-complexes $\mathcal{PC}_0$ and $\mathcal{PC}_1$.

Above we build a correspondence between a P-move of 2-3 type and a singularity of Reidemeister-III type.
As for the other P-moves, we conclude that the cancelling-pair move corresponds to the singularity of Reidemeister-II type, while the 
rest of P-moves correspond to the singularity of Reidemeister-III type.

\section{Proof of Theorem \ref{MThm}}\label{secproof}

In this section, we want to prove the main theorem. We first need the following lemma.

\begin{Lem}\label{lempmove}
Let $M$ be a closed, orientable 3-manifold. Let $F_0$ and $F_1$ be two Morse 2-functions on $M$ that define two
pants-block decompositions $PB_0$ and $PB_1$ respectively. 
Let $\{F_t\}$ be a generic homotopy between $F_0$ and $F_1$ and assume that there is a single singularity in between 
$F_0$ and $F_1$.  Then $PB_0$ and $PB_1$ either differ by a single P-move or are homotopic 
to each other.
\end{Lem}
\begin{proof}
 Given such a 3-manifold $M$ with two Morse 2-functions $F_0$ and $F_1$ on $M$,  the existence of $PB_0$ and $PB_1$ is
 given by Corollary \ref{coro2}. By Lemma \ref{lempcomplex} and the induced moves, the only types of singularities
 in between $F_0$ and $F_1$ are either of the Reidemeister-II or Reidemeister-III type, since other types of singularities are elimiated
 when  turning a Reeb complex into a P-complex.
 If all of the fold edges of 
 this singularity are index-three fold edges in the P-complexes, then $PB_0$ and $PB_1$ differ by a P-move.
 If one of the fold edges of this singularity is the preimage of a ``virtual'' edge, as described in Lemma \ref{lempcomplex2}, this fold edge doesn't
 contribute to any essential move, so if this singularity is of Reidemeister-II type, then the top slice
 of the left of Figure \ref{3move} degenerates to a single
 index-three fold edge; if this singularity is of Reidemeister-III type, then the top slice of the middle of Figure
 \ref{3move} degenerates to a single crossing.
 In both cases there is no P-move between $PB_0$ and $PB_1$.
\end{proof}

We now prove the main theorem.
\begin{proof}
 Let $M$ be a compact, closed, hyperbolic manifold.  Let $PB_0$ and $PB_1$ be two pants-block decompositions of $M$. 
 Each pants-block decomposition defines a P-complex $\mathcal{PC}_j$
 by Lemma \ref{lem2}, for $j=0,1$. We can construct a Morse 2-function $F_j$ from $\mathcal{PC}_j$ by Lemma \ref{lempcomplex2}.
 By Lemma \ref{lempcomplex} and Corollary \ref{coro2}, $F_j$ also defines a P-complex homeomorphic to $\mathcal{PC}_j$ and a 
 pants-block decomposition homeomorphic to $PB_j$.
 Consider a generic homotopy $\{F_t\}_{0\leq t\leq 1}$. There exists finitely many singularities in this homotopy.
 Assume there is a singularity at $t=s_i$ for $1\leq i\leq n$, such that $0<s_1<s_2<...<s_n<1$.
 Consider $t_i\in [0,1]$ such that $t_0=0$, $t_1=1$ and $s_{i}<t_i<s_{i+1}$ for $i=1,...,n-1$.
 Then $F_{t_i}$ is also a Morse 2-function for $i=1,...,n-1$. Each $F_{t_i}$ defines a P-complex $\mathcal{PC}_i$ and a pants-block 
 decomposition $PB_i$ of $M$ by Lemma \ref{lempcomplex} and Corollary \ref{coro2}. 
 By Lemma \ref{lempmove}, each singularity at $t=s_i$ defines at most one P-move between $PB_{t_{i-1}}$ and 
 $PB_{t_i}$. Thus $PB_0$ and $PB_1$ are related by a finite sequence of P-moves.
\end{proof}


\end{document}